\documentclass[journal,twoside,web]{ieeecolor}

\usepackage{generic}
\usepackage{cite}
\usepackage{amsmath,amssymb,amsfonts}
\usepackage{graphicx}
\usepackage{textcomp}
\usepackage{stmaryrd}
\usepackage{tabularx}

\newcommand{\f}[2]{\frac{#1}{#2}}

\newcommand{\bs}[1]{\boldsymbol{#1}}

\newcommand{\mbf}[1]{\mathbf{ #1}}
\newcommand{\tbf}[1]{\textbf{#1}}
\newcommand{\mb}[1]{\mathbb{#1}}
\newcommand{\mcl}[1]{\mathcal{#1}}
\newcommand{\R}{\mathbb{R}}
\newcommand{\N}{\mathbb{N}}
\newcommand{\norm}[1]{\left\lVert{#1}\right\rVert}
\newcommand{\bmat}[1]{\begin{bmatrix}#1\end{bmatrix}}

\newcommand{\smallbmat}[1]{\left[\scriptsize\begin{smallmatrix}
#1\end{smallmatrix} \right]}
\newcommand{\mat}[1]{\begin{matrix}#1\end{matrix}}
\newcommand{\ip}[2]{\left\langle #1, #2 \right\rangle}

\newcommand{\fvars}[3]{#1\rrbracket^{#2}_{#3}}
\newcommand{\fvarss}[3]{#1\rrbracket^{#2}_{#3}\thinspace}

\newtheorem{thm}{Theorem}
\newtheorem{defn}[thm]{Definition}
\newtheorem{lem}[thm]{Lemma}
\newtheorem{prop}[thm]{Proposition}
\newtheorem{cor}[thm]{Corollary}

\let\bl\bigl
\let\bbl\Bigl
\let\bbbl\biggl
\let\bbbbl\Biggl
\let\br\bigr
\let\bbr\Bigr
\let\bbbr\biggr
\let\bbbbr\Biggr

\title{\LARGE \bf
A PIE Representation of Coupled Linear 2D PDEs and Stability Analysis using LPIs
}

\author{\centering Declan S. Jagt, Matthew M. Peet %
\thanks{\tbf{Acknowledgement:} This work was supported by National Science Foundation grant CMMI-1935453.} %
} %

\begin{document}

\maketitle
\thispagestyle{empty}
\pagestyle{empty}

\begin{abstract}

We introduce a Partial Integral Equation (PIE) representation of Partial Differential Equations (PDEs) in two spatial variables. PIEs are an algebraic state-space representation of infinite-dimensional systems and have been used to model 1D PDEs and time-delay systems without continuity constraints or boundary conditions -- making these PIE representations amenable to stability analysis using convex optimization. To extend the PIE framework to 2D PDEs, we first construct an algebra of Partial Integral (PI) operators on the function space $L_2[x,y]$, providing formulae for composition, adjoint, and inversion. We then extend this algebra to $\R^n \times L_2[x] \times L_2[y] \times L_2[x,y]$ and demonstrate that, for any suitable coupled, linear PDE in 2 spatial variables, there exists an associated PIE whose solutions bijectively map to solutions of the original PDE -- providing conversion formulae between these representations. Next, we use positive matrices to parameterize the convex cone of 2D PI operators -- allowing us to optimize PI operators and solve Linear PI Inequality (LPI) feasibility problems. Finally, we use the 2D LPI framework to provide conditions for stability of 2D linear PDEs. We test these conditions on 2D heat and wave equations and demonstrate that the stability condition has little to no conservatism. 

\end{abstract}

\section{INTRODUCTION}
In this paper, we consider the problem of representation and stability analysis of linear Partial Differential Equations (PDEs) with multiple states evolving in 2 spatial dimensions.

First, consider how a PDE is defined. When we refer to a PDE, we are actually referring to 3 separate governing equations: The partial differential equation itself; a continuity constraint on the solution; and a set of boundary conditions (BCs). Any solution of the PDE is required to satisfy all three constraints at all times -- leading to challenging questions of existence and uniqueness of solutions. Furthermore, suppose we seek to examine whether all solutions to a PDE exhibit a common \textit{evolutionary} trait, such as stability or $L_2$-gain. How does each of the 3 governing equations affect this property? The fact that we have 3 governing equations significantly complicates the analysis and control of PDEs.

For comparison, consider the stability question for a system defined by a linear Ordinary Differential Equation (ODE) in state-space form, $\dot{x}(t)=Ax(t)$, where the ODE itself is the only constraint on the solutions of the system. In this case, a necessary and sufficient (N+S) condition for stability of the solutions of the system is the existence of a quadratic measure of energy (Lyapunov Function (LF)), $V(x) = x^T P x$ with $P> 0$, such that for any solution $x(t)\neq 0$ of the ODE, $V(x(t))$ is decreasing for all $t\ge 0$. A N+S condition for stability will then be existence of a matrix $P>0$ such that $\dot{V}(x(t)) = x^T(t) (PA +A^T P) x(t)\leq 0$ for all $x(t)\in \R^n$. This constraint is in the form of a Linear Matrix Inequality (LMI), and can be solved using convex optimization algorithms~\cite{boyd1994linear}.

Now, let us consider the problems with extending this state-space approach to stability analysis of 2D PDEs. In this case, the state at time $t$ of the PDE is a function $\mbf{u}(t,x,y)$ of 2 spatial variables -- raising the question of how one can parameterize the convex cone of LFs which are positive on this 2D function space without introducing significant conservatism. Moreover, we recall that solutions to the `PDE' are required to satisfy not only the PDE itself (e.g. $\mbf{u}_t=\mbf{u}_{xx}+\mbf{u}_{yy}$), but are also required to satisfy both continuity constraints (e.g. $\mbf{u}(t,\cdot,\cdot)\in H_2$) and boundary conditions (e.g. $\mbf{u}(t,0,y)=\mbf{u}(t,1,y)=\mbf{u}(t,x,0)=\mbf{u}(t,x,1)=0$). The second problem is then, given a Lyapunov function of the form $V(\mbf u)$, how to determine whether $\dot V(\mbf u(t))<0$ for all $\mbf u$ satisfying all three constraints. In particular, how do the BCs and continuity constraints influence $\dot V(\mbf u(t))$?

Regarding the first problem, it is known that for linear PDEs, as was the case for ODEs, existence of a decreasing quadratic LF is N+S for stability of solutions~\cite{curtain1995introduction} - so that we may assume the LF has the form $V(\mbf{u})=\ip{\mbf{u}}{\mcl{P}\mbf{u}}_{L_2}$ for some positive operator $\mcl{P}>0$. However, it is unclear how to parameterize a set of linear operators which is suitably rich so as to avoid significant conservatism, whilst still allowing positivity of the operators to be efficiently enforced. As a result, most prior work has been restricted to employing variations of the identity operator for $\mcl P$. For example, in~\cite{ahmadi2018framework}, a Lypanunov function of the form $V=\alpha \norm{u}_{L_2}^2$ was used. Meanwhile, in \cite{fridman2016new}, \cite{solomon2015stability}, and \cite{wakaiki2019lmi}, the authors assumed $\mcl{P}$ to be a multiplier operator of the form $\mcl{P}\mbf{u} =M\mbf{u}(s)$, with positivity implied by the matrix inequality $M>0$. In \cite{valmorbida2015convex}, the authors extended these functionals somewhat, using polynomial multipliers $\mcl{P}\mbf{u} =M(s)\mbf{u}(s)$ with the Sum-of-Squares (SOS) constraint $M(s)\ge 0$. However, in each of these cases, the use of multiplier operators (analagous to the use of diagonal matrices in an LMI) implies significant conservatism in any stability analysis using such results.

We now turn to the second problem with stability analysis of PDEs: enforcing negativity of the derivative. As stated, for linear ODEs $\dot x(t)=Ax(t)$, this condition is easy to enforce: $\dot V(x(t))=x(t)^T(PA+A^TP)x(t)\le 0$ for all solutions $x(t)$ and all $t \geq 0$ if and only if $PA+A^TP$ is a negative definite matrix - an LMI constraint.
However, for a PDE $\dot{\mbf{u}} = \mcl{A}\mbf{u}$, the state space has more structure. That is, while it is true that if $V=\ip{\mbf{u}}{\mcl P\mbf{u}}$ then $\dot V(\mbf{u})=\ip{\mbf{u}}{[\mcl{P}\mcl{A}+\mcl{A}^*\mcl{P}]\mbf{u}}$, negativity of the operator $\mcl{P}\mcl{A}+\mcl{A}^*\mcl{P}$ is not a N+S condition for stability, as $\dot V(\mbf{u})\leq 0$ need only hold for solutions $\mbf{u}\in X$ satisfying the BCs and continuity constraints.
To account for this, in \cite{fridman2016new}, \cite{solomon2015stability}, and \cite{wakaiki2019lmi}, the authors consider particular types of BCs, allowing the use of known integral inequalities such as Poincare, Wirtinger, etc. to prove negativity. In \cite{valmorbida2015convex}, it was proposed to use a more general set of inqualities defined by Green's functions.
In all these cases, however, the process of identification and application of useful inequalities to prove negativity requires significant expertise and insight.

To avoid having to manually integrate boundary conditions and continuity constraints into the expression for $\dot V$, \cite{shivakumar2020pietools,shivakumar2019generalized,peet2019discussion} suggest representing the PDE as a Partial Integral Equation (PIE). A PIE is a unitary representation of the PDE whose solution is defined on $L_2$ and hence does not require boundary conditions or continuity constraints.

For autonomous systems, PIEs are parameterized by the Partial Integral (PI) operators $\mcl{T}$ and $\mcl{A}$, and take the form $\mcl{T}\dot{\hat{\mbf{u}}}=\mcl{A}\hat{\mbf{u}}$, where $\hat{\mbf{u}}$ is the so-called ``fundamental state''. This \textit{fundamental state} $\hat{\mbf{u}}\in L_2$ is free of boundary and continuity constraints, and is associated to the PDE state through the transformation $\mbf{u}=\mcl{T} \hat{\mbf{u}}$. PI operators form a Banach *-algebra, with analytic expressions for composition, adjoint, etc. Furthermore, positive PI operators can be parameterized by positive matrices -- allowing us to solve Linear PI Inequality (LPI) optimization problems using semidefinite programming. Thus, the use of PIEs and PI operators also resolves the difficulty with parameterizing positive LFs, taking these to be of the form $V(\mbf{u})=\ip{\mbf{u}}{\mcl{P}\mbf{u}}$, where $\mcl{P}$ is a PI operator. In 1D, the Matlab toolbox PIETOOLS~\cite{shivakumar2020pietools} can be used for parsing and solving LPI optimization problems.

Thus, through use of the LPI and PIE framework, stability analysis, as well as tasks like $H_{\infty}$-optimal controller design~\cite{shivakumar2020duality}, can be efficiently performed for almost any 1D PDE using convex optimization.
However, as of yet, none of this architecture exists for 2D PDEs. Specifically, no concept of fundamental state has been defined for 2D PDEs, and there is no known algebra of 2D PI operators which could be used to parameterize a PIE representation, or be incorporated into some form of LPI optimization algorithm.

The goal of this paper, then, is to recreate the PIE framework for linear, 2nd order PDEs on a 2D domain $(x,y)\in[a,b]\times[c,d]$ with non-periodic boundary conditions. To this end, the paper makes the following contributions.

\textbf{1)} \textit{Identify an algebra of 2D PI operators.}\\
In Section~\ref{sec_algebras}, we parameterize a set of PI operators with domain $L_2[x,y]$, which we combine with the algebra of 4-PI operators on $\mb{R}^n\times L_2[x]\times L_2[y]$ (representing the boundary of the domain) to yield a Banach *-algebra of PI operators on $\mb{R}^n\times L_2[x]\times L_2[y]\times L_2[x,y]$. We demonstrate that this set of PI operators is closed under, addition, adjoint, and composition -- deriving analytic expressions for the result of each operation. In Section~\ref{sec_properties}, we further derive analytic expressions for the inverse of a suitable PI operator on $\mb{R}^n\times L_2[x]\times L_2[y]$, and the composition of a differential operator with a suitable 2D PI operator, showing that the result of each is a PI operator.

\textbf{2)} \textit{Identify the fundamental state for a 2D PDE.}\\
In Subsection~\ref{sec_uhat2u}, we differentiate ($\mcl{D}$) the PDE state $\mbf{u}\in X$ up to the maximal degree allowed per the continuity constraints. The resulting $\hat{\mbf{u}}=\mcl{D}\mbf{u}\in L_2$ will then be free of boundary and continuity constraints.

\textbf{3)} \textit{Reconstruct the PDE state from the fundamental state.}\\
Having defined a fundamental state $\hat{\mbf{u}}$, we isolate a set of ``core'' and ''full'' boundary values of the PDE state as $\Lambda_{\text{bc}}\mbf{u}$ and $\Lambda_{\text{bf}}\mbf{u}$ respectively. Using the fundamental theorem of calculus, we can then express $\mbf{u}=\mcl{K}_{1}\Lambda_{\text{bc}}\mbf{u}+\mcl{K}_{2}\hat{\mbf{u}}$ and $\Lambda_{\text{bf}}\mbf{u}=\mcl{H}_{1}\Lambda_{\text{bc}}\mbf{u}+\mcl{H}_{2}\hat{\mbf{u}}$, where $\mcl{K}_1$, $\mcl{K}_2$, $\mcl{H}_1$ and $\mcl{H}_2$ are 2D PI operators. Next, we impose the boundary conditions as $\mcl{B}\Lambda_{\text{bf}}\mbf{u}=0$, where $\mcl{B}$ is a PI operator, allowing us to write $\mcl{B}\mcl{H}_1\Lambda_{\text{bc}}\mbf{u}+\mcl{B}\mcl{H}_2\hat{\mbf{u}}=0$ and thus $\Lambda_{\text{bc}}\mbf{u}=-(\mcl{B}\mcl{H}_1)^{-1}\mcl{B}\mcl{H}_2\hat{\mbf{u}}$. Finally, we retrieve the PDE state as
$\mbf{u}=\mcl{T}\hat{\mbf{u}}=[\mcl{K}_2-(\mcl{B}\mcl{H}_1)^{-1}\mcl{B}\mcl{H}_2\mcl{K}_1]\hat{\mbf{u}}$, where $\mcl{T}$ is a 2D PI operator.

\textbf{4)} \textit{Derive a PIE representation for a standardized PDE.}\\
In Section~\ref{sec_PDE}, we present a standardized format for writing coupled PDEs. In Subsection~\ref{sec_PDE2PIE}, we then use the transformation $\mbf{u}=\mcl{T}\hat{\mbf{u}}$ and the composition rules for differential operators with PI operators, to derive an equivalent PIE representation as $\mcl{T}\dot{\hat{\mbf{u}}}=\mcl{A}\hat{\mbf{u}}$.

\textbf{5)} \textit{Derive an LPI stability test.}\\
In Subsection~\ref{sec_lyapunov_stability_criterion}, we parameterize a LF $V(\hat{\mbf{u}})=\ip{T\hat{\mbf{u}}}{\mcl{P}\mcl{T}\hat{\mbf{u}}}$ as a PI operator $\mcl{P}$. Using this LF, we prove that existence of a PI operator $\mcl{P}>0$ satisfying the LPI $\mcl{A}^*\mcl{P}\mcl{T}+\mcl{T}^*\mcl{P}\mcl{A}<0$ certifies stability of the PDE.

\textbf{6)} \textit{Parameterize the convex cone of positive PI operators.}\\
In Subsection~\ref{sec_positive_PI_parameterization}, we introduce a PI operator $\mcl{Z}$, defined by monomial basis functions $Z(x,y)$. For any positive matrix $P>0$, then, the product $\mcl{P}=\mcl{Z}^*P\mcl{Z}$ will be a PI operator satisfying $\ip{\hat{\mbf{u}}}{\mcl{P}\hat{\mbf{u}}}>0$ for any $\hat{\mbf{u}}\in L_2[x,y]$.

\textbf{7)} \textit{Implement this methodology in PIETOOLS.}\\
In Section~\ref{sec_implementation}, we implement a class of 2D-PI operators in the MATLAB toolbox PIETOOLS 2021b, along with the formulae for constructing the PIE representation of a PDE. This allows an arbitrary PDE to be converted to an equivalent PIE, at which point stability may be tested by solving the LPI $\bbl[\mcl{A}^*\mcl{P}\mcl{T}+\mcl{T}^*\mcl{P}\mcl{A}\bbr]<0$ for a positive operator $\mcl{P}>0$. Parameterizing the cone of positive operators as positive matrices, this problem may then be solved using semidefinite programming. We test this implementation on a heat equation and a wave equation in Section~\ref{sec_examples}.

\begin{table*}[!ht]
	\begin{tabular}{p{2.75cm}|p{8.25cm}|p{5.0cm}}
		Parameter space & Explicit notation & Associated PI mapping	\\\hline
		$\mcl{N}_{1D}^{n\times m}$ & $L_2^{n\times m}[x] \times L_2^{n\times m}[x,\theta] \times L_2^{n\times m}[x,\theta] $  &
		$L_2^{m}[x]\rightarrow L_2^{n}[x]$	\\\hline
		$\mcl{N}_{011}\smallbmat{n_0&m_0\\n_1&m_1}$ & $\bmat{\R^{n_0\times m_0}   &L_2^{n_0\times m_1}[x] &L_2^{n_0\times m_1}[y]\\
			L_2^{n_1\times m_0}[x] &\mcl{N}_{1D}^{n_1\times m_1}   &L_2^{n_1\times m_1}[x,y]\\
			L_2^{n_1\times m_0}[y] &L_2^{n_1\times m_1}[x,y]   &\mcl{N}_{1D}^{n_1\times m_1}}$	 &
		$\bmat{\R^{m_0}\\L_2^{m_1}[x]\\L_2^{m_1}[y]}\rightarrow \bmat{\R^{n_0}\\L_2^{n_1}[x]\\L_2^{n_1}[y]}$ 
		\\\hline
		$\mcl{N}_{2D}^{n\times m}$ & $\bmat{
			L_{2}^{n \times m}[x,y]           &L_{2}^{n \times m}[x,y,\nu]    &L_{2}^{n \times m}[x,y,\nu]\\
			L_{2}^{n \times m}[x,y,\theta]    &L_2^{n\times m}[x,y,\theta,\nu]&L_2^{n\times m}[x,y,\theta,\nu]\\
			L_{2}^{n \times m}[x,y,\theta]    &L_2^{n\times m}[x,y,\theta,\nu]&L_2^{n\times m}[x,y,\theta,\nu]}$ 
		& $L_2^{m}[x,y]\rightarrow L_2^{n}[x,y]$ \\\hline
		$\mcl N_{2D\rightarrow 1D}^{n\times m}$ & $L_2^{n \times m}[x,y] \times L_2^{n \times m}[x,y,\theta] \times L_2^{n \times m}[x,y,\theta]$
		& $L_2^{m}[x,y]\rightarrow L_2^{n}[x]$ \\\hline
		$\mcl N_{1D\rightarrow 2D}^{n\times m}$ & $L_2^{n \times m}[x,y] \times L_2^{n \times m}[x,y,\nu] \times L_2^{n \times m}[x,y,\nu]$ 
		& $L_2^{m}[x]\rightarrow L_2^{n}[x,y]$ \\\hline
		$\mcl{N}_{0112}{\bmat{n_0&m_0\\n_1&m_1\\n_2&m_2}}$ & $\bmat{
			\mcl{N}_{011}{\bmat{n_0&m_0\\n_1&m_1}} 
			&\mat{L_2^{n_0\times m_2}[x,y]\\ \mcl{N}^{n_1\times m_2}_{2D\rightarrow 1D}\\
				\mcl{N}^{n_1\times m_2}_{2D\rightarrow 1D}}\\
			\mat{L_2^{n_2\times m_0}[x,y]& \mcl{N}^{n_2\times m_1}_{1D\rightarrow 2D}&
				\mcl{N}^{n_2\times m_1}_{1D\rightarrow 2D}}
			&\mcl{N}_{2D}^{n_2\times n_2}}$
		& $\smallbmat{\R^{m_0}\\L_2^{m_1}[x]\\L_2^{m_1}[y]\\L_2^{m_2}[x,y]}\rightarrow \smallbmat{\R^{n_0}\\L_2^{n_1}[x]\\L_2^{n_1}[y]\\L_2^{n_2}[x,y]}$
	\end{tabular}
	\caption{Parameter spaces for PI operators introduced in Section~\ref{sec_algebras}}
	\label{tab_Nspaces}
\end{table*}

\section{Notation}

For a given domain $x \in [a,b]$ and $y\in [c,d]$, let $L_2^n[x,y]$ denote the set of $\R^n$-valued square-integrable functions on $[a,b]\times[c,d]$, with the standard inner product. $L_2^n[x]$ is defined similarly and we omit the domain when clear from context. For any $\alpha \in \N^2$, we denote $\norm{\alpha}_{\infty}:=\max \{\alpha_1,\alpha_2\}$. 
Using this notation, we define $H_{k}^n[x,y]$ as a Sobolev subspace of $L_2^n[x,y]$, where
\begin{align*}
	H_{k}^n[x,y]=\{\mbf{u}\mid \partial_x^{\alpha_1}\partial_y^{\alpha_2}\mbf{u}\in L_2^n[x,y],\enspace \forall \|\alpha\|_{\infty}\leq k\}.
\end{align*}
As for $L_2$, we occasionally use $H_{k}^n:=H_{k}^n[x,y]$ or $H_{k}^n:=H_{k}^n[x]$ when the domain is clear from context. For $\mbf{u}\in H_{k}^n[x,y]$, we use the norm
\begin{align*}
 \norm{\mbf{u}}_{H_k}=\sum_{\norm{\alpha}_{\infty}\leq k}\norm{\partial_x^{\alpha_1}\partial_y^{\alpha_2}\mbf{u}}_{L_2}
\end{align*}
For any $\mbf u \in L_2^n[x,y]$, we denote the Dirac delta operators
\[
[\Lambda_x^a \mbf u](y):=\mbf u(a,y)\quad \text{and}\quad [\Lambda_y^c \mbf u](x):=\mbf u(x,c).
\]

\section{Algebras of PI Operators}\label{sec_algebras}
In~\cite{peet2019discussion}, we parameterized an algebra of PI operators whose domain was functions $L_2[x]$ of a single spatial variable, $x\in[a,b]$. These operators took the form
\begin{align}\label{eq:1D_PI_operator} (\mcl{P}[N]\mbf{u})(x)&=N_{0}(x)\mbf{u}(x)+\int_{a}^{x}N_{1}(x,\theta)\mbf{u}(\theta)d\theta    \\
&\qquad+\int_{x}^{b}N_{2}(x,\theta)\mbf{u}(\theta)d\theta.	\nonumber
\end{align}
with polynomial parameters $N_i$. In~\cite{shivakumar2019generalized}, these PI operators were generalized, yielding operators with domain $\R^n \times L_2[x]$.  These extended operators then had the form
\[
\bmat{Pv +\int_{a}^{b}\bl( Q_1(\theta)\mbf{u}(\theta)\br)d\theta
	\\Q_2(x)v+(\mcl{P}[N]\mbf{u})(x)}
\]
with parameters $P,Q_1,Q_2$, and $\mcl{P}[N]$ -- where we note that the third parameter is itself a parameterized PI operator on the domain $L_2$. In this paper, we are tasked with further generalizing the algebra of PI operators to functions of \textit{two} spatial variables -- i.e. $L_2[x,y]$. Operators in this rather more complicated algebra take the form of
\begin{align}\label{eq:2D_PI_operator}
&(\mcl P[N]\mbf u)(x,y):=N_{00}(x,y)\mbf{u}(x,y)\\
& +\int_{a}^{x}N_{10}(x,y,\theta)\mbf{u}(\theta,y)d\theta +\int_{x}^{b}N_{20}(x,y,\theta)\mbf{u}(\theta,y)d\theta \nonumber\\ &+\int_{c}^{y}N_{01}(x,y,\nu)\mbf{u}(x,\nu)d\nu +\int_{y}^{d}N_{02}(x,y,\nu)\mbf{u}(x,\nu)d\nu \nonumber\\
&\qquad+\int_{a}^{x}\int_{c}^{y}N_{11}(x,y,\theta,\nu)\mbf{u}(\theta,\nu)d\nu d\theta \nonumber\\
&\qquad\qquad+\int_{x}^{b}\int_{c}^{y}N_{21}(x,y,\theta,\nu)\mbf{u}(\theta,\nu)d\nu d\theta \nonumber\\
&\qquad\qquad\qquad+\int_{a}^{x}\int_{y}^{d}N_{12}(x,y,\theta,\nu)\mbf{u}(\theta,\nu)d\nu d\theta \nonumber\\
&\qquad\qquad\qquad\qquad+\int_{x}^{b}\int_{y}^{d}N_{22}(x,y,\theta,\nu)\mbf{u}(\theta,\nu)d\nu d\theta.	\nonumber
\end{align}
where we now have 9 polynomial parameters $N_{ij}$. Making matters worse, we will also need to include cross-terms from $\R^n$, $L_2[x]$ and $L_2[y]$ in our algebra. In the following subsections, to make presentation of this class of operators somewhat tractable, we will make heavy use of an ``operator parameterization of operators'' approach, whereby we parameterize simpler algebras of operators using polynomials and embed these simpler algebras in the more complex ones. To aid in keeping track of these algebras, we use, e.g. the term 0112-PI algrebra to indicate its domain and range include $\R^n$ (indicated by `0'), $L_2[x]$ (indicated by $1$), $L_2[y]$  (indicated by $1$), and $L_2[x,y]$ (indicated by $2$).

Note that throughout this article we overload the $\mcl P[N]$ notation for a PI operator with parameters $N \in \mcl N$, where the structure and parameterization of the operator varies depending on the specific parameter space $\mcl N$. The different parameter spaces we consider are listed in Table~\ref{tab_Nspaces}.

\subsection{An Algebra of 2D to 2D PI Operators}\label{sec_2D}
We start by parameterizing operators on $L_2^{m}[x,y]$, for which we define the parameter space
\begin{align*}
\mcl N_{2D}^{n\times m}&\! \! :=\! \!
\left[\! \! \! \begin{array}{lll}
L_{2}^{n \times m}[x,y]           &\! \! \! L_{2}^{n \times m}[x,y,\nu]    &\! \! \! L_{2}^{n \times m}[x,y,\nu]\\
L_{2}^{n \times m}[x,y,\theta]    &\! \! \! L_2^{n\times m}[x,y,\theta,\nu]&\! \! \! L_2^{n\times m}[x,y,\theta,\nu]\\
L_{2}^{n \times m}[x,y,\theta]    &\! \! \! L_2^{n\times m}[x,y,\theta,\nu]&\! \! \! L_2^{n\times m}[x,y,\theta,\nu]
\end{array}\! \! \! \right]\! .
\end{align*}
Then, for any $N:=\smallbmat{N_{00}&N_{01}&N_{02}\\N_{10}&N_{11}&N_{12}\\N_{20}&N_{21}&N_{22}}\in \mcl N_{2D}^{n\times m}$, we let the associated 2D-PI operator $\mcl P[N]:L_{2}^{m}[x,y]\rightarrow L_{2}^{n}[x,y]$ be as in Eqn.~\eqref{eq:2D_PI_operator}. Defining addition and scalar multiplication of the parameters $N,M\in\mcl{N}_{2D}^{n\times m}$ in the obvious manner, it immediately follows that $\mcl{P}[N]+\mcl{P}[M]=\mcl{P}[N+M]$, and $\lambda\mcl{P}[N]=\mcl{P}[\lambda N]$, for any $\lambda\in\R$. Further defining multiplication of 2D-PI operators as in the following lemma, we conclude that the set of 2D-PI operators is an algebra.
\begin{lem}
	For any $N \in \mcl N_{2D}^{n\times p}$ and $M \in \mcl N_{2D}^{p\times m}$, there exists a unique $Q \in \mcl N_{2D}^{n\times m}$ such that $\mcl P[N]\mcl P[M]=\mcl P[Q]$. Specifically, we may choose
	\[
	Q=\mcl L_{2D}(N,M)\in \mcl N_{2D}^{n\times m},
	\]
	where $\mcl L_{2D}:\mcl{N}_{2D}^{n\times p} \times \mcl{N}_{2D}^{p\times m} \rightarrow \mcl{N}_{2D}^{n\times m}$ is defined in Eqn.~\eqref{eq_composition_2Dto2D_1_appendix} in Appendix~\ref{sec_2Dalgebra_appendix}.
\end{lem}

\subsection{An Algebra of 011-PI Operators}\label{sec_01D}
Having defined an algebra of operators on $L_2[x,y]$, we now consider a parameterization of operators on $RL^{n_0,n_1}:=\R^{n_0}\times L_2^{n_1}[x]\times L_2^{n_1}[y]$. To this end, we first let for any $N = \{N_0,N_1,N_2\} \in \mcl N_{1D}^{n\times m} := L_{2}^{n \times m}[x] \times L_{2}^{n \times m}[x,\theta] \times L_{2}^{n \times m}[x,\theta]$, the associated 1D-PI operator be as in~\eqref{eq:1D_PI_operator}. Next, we define a parameter space for 011-PI operators as
\begin{align*}
\mcl{N}_{011}{\smallbmat{n_0&m_0\\n_1&m_1}}\! :=\!
\bmat{\R^{n_0\times m_0}   &\! L_2^{n_0\times m_1}[x] &\! L_2^{n_0\times m_1}[y]\\
	L_2^{n_1\times m_0}[x] &\! \mcl{N}_{1D}^{n_1\times m_1}   &\! L_2^{n_1\times m_1}[x,y]\\
	L_2^{n_1\times m_0}[y] &\! L_2^{n_1\times m_1}[y,x]   &\! \mcl{N}_{1D}^{n_1\times m_1}}\! .
\end{align*}
Then, for any $B:=\smallbmat{B_{00}&B_{01}&B_{02}\\B_{10}&B_{11}&B_{12}\\B_{20}&B_{21}&B_{22}}\in \mcl{N}_{011}{\smallbmat{n_0&m_0\\n_1&m_1}}$, we let the associated PI operator $\mcl{P}[B]:RL^{m_0,m_1}\rightarrow RL^{n_0,n_1}$ be given by
\begin{align*}
\mcl P[B]:=\bmat{
	\text{M}[B_{00}] &\smallint_{x=a}^b [B_{01}]  &\smallint_{y=c}^d [B_{02}]  \\
	\text{M}[B_{10}] &\mcl{P}[B_{11}]        &\smallint_{y=c}^d[B_{12}]\\
	\text{M}[B_{20}] &\smallint_{x=a}^b[B_{21}]   &\mcl{P}[B_{22}]},
\end{align*}
where M is the multiplier operator and $\smallint$ is the integral operator, so that (through some abuse of notation)
\[
(\text{M}[N]\mbf{u})(x,y):=N(x,y)\mbf{u}(y),
\]
and
\[
\left(\smallint_{y=c}^{d}[N]\mbf{u}\right)(x):=\int_{c}^{d}N(x,y)\mbf{u}(y)dy.
\]
Clearly then, $\mcl{P}[B]+\mcl{P}[D]=\mcl{P}[B+D]$ and $\lambda\mcl{P}[B]=\mcl{P}[\lambda B]$ for any $B,D\in\mcl{N}_{011}$ and $\lambda\in\mb{R}$. Moreover, we can also compose 011-PI operators, as per the following lemma.

\begin{lem}
  For any $B \in \mcl{N}_{011}{\smallbmat{n_0&p_0\\n_1&p_1}}$, $D \in \mcl{N}_{011}{\smallbmat{p_0&m_0\\p_1&m_1}}$, there exists a unique
  $R \in \mcl{N}_{011}{\smallbmat{n_0&m_0\\n_1&m_1}}$ such that $\mcl{P}[B]\mcl{P}[D]=\mcl{P}[R]$. Specifically, we may choose
  \[
  R=\mcl{L}_{011}(B,D)\in \mcl{N}_{011}{\smallbmat{n_0&m_0\\n_1&m_1}},
  \]
  where the linear parameter map $\mcl L_{011}:\mcl{N}_{011}{\smallbmat{n_0&p_0\\n_1&p_1}} \times \mcl{N}_{011}{\smallbmat{p_0&m_0\\p_1&m_1}} \rightarrow \mcl{N}_{011}{\smallbmat{n_0&m_0\\n_1&m_1}}$ is defined in Eqn.~\eqref{eq_composition_011to011_1_appendix} in Appendix~\ref{sec_011algebra_appendix}. 
\end{lem}

\subsection{An Algebra of 0112-PI Operators}\label{sec_012D}
We now combine the 011-PI algebra and the 2D-PI algebra to obtain an algebra of operators on $RL^{n_0,n_1}\times L_2^{n_2}[x,y]=\R^{n_0}\times L_2^{n_1}[x] \times L_2^{n_1}[y] \times L_{2}^{n_2}[x,y]$. Specifically, for any
\begin{align*}
&C =\bmat{B &\mat{C_{03} \\ C_{13}\\C_{23}}\\ \mat{C_{30}&C_{31}&C_{32}}&N}  \in \mcl{N}_{0112}{\smallbmat{n_0&m_0\\n_1&m_1\\n_2&m_2}}:=\\
& \bmat{
\mcl{N}_{011}{\smallbmat{n_0&m_0\\n_1&m_1}}
&\mat{L_2^{n_0\times m_2}[x,y]\\ \mcl{N}^{n_1\times m_2}_{2D\rightarrow 1D}\\
 \mcl{N}^{n_1\times m_2}_{2D\rightarrow 1D}}\\
\mat{L_2^{n_2\times m_0}[x,y]& \mcl{N}^{n_2\times m_1}_{1D\rightarrow 2D}&
 \mcl{N}^{n_2\times m_1}_{1D\rightarrow 2D}}
&\mcl{N}_{2D}^{n_2\times n_2}
},
\end{align*}
where
\[
\mcl N_{2D\rightarrow 1D}^{n\times m} := L_2^{n \times m}[x,y] \times L_2^{n \times m}[x,y,\theta] \times L_2^{n \times m}[x,y,\theta],
\]
\[
\mcl N_{1D\rightarrow 2D}^{n\times m} := L_2^{n \times m}[x,y] \times L_2^{n \times m}[x,y,\nu] \times L_2^{n \times m}[x,y,\nu],
\]
we may define the 0112-PI operator $\mcl{P}[C]:RL^{m_0,m_1}[x,y]\times L_{2}^{m_2}[x,y]\rightarrow RL^{n_0,n_1}[x,y] \times L_{2}^{n_2}[x,y]$ as
\[
\mcl P[C]=\bmat{\mcl P[B] & \hspace{-4mm}\mat{\smallint_{x=a}^{b}[I]\circ \smallint_{y=c}^{d}[C_{03}]\\
\mcl{P}[C_{13}] \\\mcl{P}[C_{23}]}\\ \mat{\text{M}[C_{30}]  &\mcl{P}[C_{31}] & \mcl{P}[C_{32}]} &\hspace{-4mm} \mcl P [N]},
\]
where for $D = \{D_0,D_1,D_2\} \in \mcl N_{2D\rightarrow 1D}^{n\times m}$, we have
\begin{align*}
&(\mcl{P}[D]\mbf{u})(x):=
\int_{c}^{d}\bbbl[D_0(x,y)\mbf{u}(x,y) \\
&+\int_{a}^{x}D_{1}(x,y,\theta)\mbf{u}(\theta,y)d\theta     
+\int_{x}^{b}D_{2}(x,y,\theta)\mbf{u}(\theta,y)d\theta\bbbr] dy,
\end{align*}
and for $E = \{E_0,E_1,E_2\} \in \mcl N_{1D\rightarrow 2D}^{n\times m}$, we have
\begin{align*}
&(\mcl{P}[E]\mbf{u})(x,y):=E_{0}(x,y)\mbf{u}(y) \\
&\quad+\int_{c}^{y}E_{1}(x,y,\nu)\mbf{u}(\nu)d\nu
 +\int_{y}^{d}E_{2}(x,y,\nu)\mbf{u}(\nu)d\nu.
\end{align*}

Clearly, the set of operators parameterized in this manner is closed under addition and scalar multiplication. By the following lemma, then, the set of 0112-PI operators is an algebra.

\begin{lem}
  For any $B \in \mcl N_{0112}\smallbmat{n_0&p_0\\n_1&p_1\\n_2&p_2}$ and $D \in \mcl N_{0112}\smallbmat{p_0&m_0\\p_1&m_1\\p_2&m_2}$, there exists a unique $R \in \mcl N_{0112}\smallbmat{n_0&m_0\\n_1&m_1\\n_2&m_2}$ such that $\mcl P[B]\mcl P[D]=\mcl P[R]$. Specifically, we may choose
  \[
  R=\mcl L_{0112}(B,D)\in \mcl N_{0112}\smallbmat{n_0&m_0\\n_1&m_1\\n_2&m_2},
  \]
  where the linear parameter map $\mcl{L}_{0112}:\mcl{N}_{0112}\smallbmat{n_0&p_0\\n_1&p_1\\n_2&p_2} \times \mcl{N}_{0112}\smallbmat{p_0&m_0\\p_1&m_1\\p_2&m_2} \rightarrow \mcl{N}_{0112}\smallbmat{n_0&m_0\\n_1&m_1\\n_2&m_2}$ is defined in Eqn.~\eqref{eq_composition_0112_appendix} in Appendix~\ref{sec_0112algebra_appendix}.\\
\end{lem}

For the purpose of implementation in Section~\ref{sec_implementation}, we will be considering only PI operators parameterized by polynomial functions, exploiting the following result:

\begin{cor}
	For any polynomial parameters $B \in \mcl N_{0112}\smallbmat{n_0&p_0\\n_1&p_1\\n_2&p_2}$ and $D \in \mcl N_{0112}\smallbmat{p_0&m_0\\p_1&m_1\\p_2&m_2}$, the composite parameters $R=\mcl L_{0112}(B,D)\in \mcl N_{0112}\smallbmat{n_0&m_0\\n_1&m_1\\n_2&m_2}$ are also polynomial.
\end{cor}

\section{Useful Properties of PI Operators}\label{sec_properties}
Having defined the different algebras of PI operators to be used in later sections, we now derive several critical properties of such operators. Specifically, we focus on obtaining an analytic expression for the inverse of a 011-PI operator, using a generalization of the formula in~\cite{miao2019inversion}. This result will be used to enforce the boundary conditions when deriving the mapping from fundamental state to PDE state in Section~\ref{sec_PIE}. In addition, we consider composition of differential operators with a PI operator, proving that the result is a PI operator, and obtaining an analytic expression for this operator. This result will be used to relate differential operators in the PDE to PI operators in the equivalent PIE representation in Section~\ref{sec_PIE}. Finally, we give a formula for the adjoint of a PI operator, necessary for deriving and enforcing LPI stability conditions in Section~\ref{sec_stability_LPI}.

\subsection{Inverse of 011-PI Operators}\label{sec_inv}

First, given a 011-PI operator $\mcl P[Q]$ defined by parameters $Q\in\mcl{N}_{011}\smallbmat{n_0&n_0\\n_1&n_1}$, we prove that $\mcl P[Q]^{-1}=\mcl P[\hat{Q}]$ is a 011-PI operator and obtain an analytic expression for the parameters of the inverse $\hat{Q}\in\mcl{N}_{011}\smallbmat{n_0&n_0\\n_1&n_1}$. As is typical, we restrict ourselves to the case where the operator has a separable structure and the parameters are polynomial (and hence have a finite-dimensional parameterization).

\begin{lem}\label{lem_inv}
Suppose
\[
 Q=\bmat{Q_{00}&Q_{0x}&Q_{0y}\\Q_{x0}&Q_{xx}&Q_{xy}\\Q_{y0}&Q_{yx}&Q_{yy}}\in{N}_{011}\smallbmat{n_0&n_0\\n_1&n_1}
\]
where $Q_{xx}=\{Q_{xx}^0,Q^1_{xx},Q^1_{xx}\}\in\mcl{N}_{1D}^{n_1\times n_1}$ (separable) and $Q_{yy}=\{Q^0_{yy},Q^1_{yy},Q^1_{yy}\}\in\mcl{N}_{1D}^{n_1\times n_1}$ (separable). Suppose that $Q_{x0},Q_{y0},Q_{0x},Q_{0y},Q^1_{xx},Q_{xy},Q_{yx},Q^1_{yy}$ can be decomposed as
\begin{align*}
 &\bmat{&Q_{0x}(x)&Q_{0y}(y)\\Q_{x0}(x)&Q_{xx}^{1}(x,\theta)&Q_{xy}(x,y)\\Q_{y0}(y)&Q_{yx}(x,y)&Q_{yy}^{1}(y,\nu)}
 =\\
 &\qquad\bmat{&H_{0x}Z(x)&H_{0y}Z(y)\\
  Z^T(x)H_{x0}&Z^T(x)\Gamma_{xx}Z(\theta)&Z^T(x)\Gamma_{xy}Z(y)\\
  Z^T(y)H_{y0}&Z^T(y)\Gamma_{yx}Z(x)&Z^T(y)\Gamma_{yy}Z(\nu)}
\end{align*}
where $Z \in L_2^{q \times n_1}$ and
\[
 \bmat{&H_{0x}&H_{0y}\\H_{x0}&\Gamma_{xx}&\Gamma_{xy}\\H_{y0}&\Gamma_{yx}&\Gamma_{yy}}\in
 \bmat{&\mb{R}^{n_0\times q}&\mb{R}^{n_0\times q}\\
 \mb{R}^{q\times n_0}&\mb{R}^{q\times q}&\mb{R}^{q\times q}    \\
 \mb{R}^{q\times n_0}&\mb{R}^{q\times q}&\mb{R}^{q\times q}},
\]
for some $q\in \N$. Now suppose that $\hat{Q}=\mcl{L}_{\text{inv}}(Q)\in\mcl{N}_{011}\smallbmat{n_0&n_0\\n_1&n_1}$, with $\mcl{L}_{\text{inv}}:\mcl{N}_{011}\smallbmat{n_0&n_0\\n_1&n_1}\rightarrow \mcl{N}_{011}\smallbmat{n_0&n_0\\n_1&n_1}$ defined as in Eqn.~\eqref{eq_inv_011_appendix} in Appendix~\ref{sec_inv_appendix}. Then for any $\mbf{u}\in\R^{n_0}\times L_2^{n_1}[x]\times L_2^{n_1}[y]$,
\begin{align*}
 (\mcl{P}[\hat Q]\circ\mcl{P}[Q]) \mbf{u}=(\mcl{P}[Q]\circ\mcl{P}[\hat Q]) \mbf{u}=\mbf{u}.
\end{align*}

\end{lem}
\begin{proof}
	See Appendix~\ref{sec_inv_appendix} for a proof.
\end{proof}

\subsection{Differentiation and 2D-PI Operators}\label{sec_div}
While differentiation is an unbounded operator and PI operators are bounded, we now show that for a PI operator with no multipliers, composition of the differential operator with this PI operator is a PI operator and hence bounded. This will be used in Section~\ref{sec_PIE} to show that the dynamics of a PDE can be equivalently represented using only bounded operators (the PIE representation).

\begin{lem}\label{lem_div_operator_x}
Suppose
\begin{align*}
N\! =\! \bmat{0&\! \! 0&\! \! 0\\N_{10}&\! \! N_{11}&\! \! N_{12}\\N_{20}&\! \! N_{21}&\! \! N_{22}}\! &\in\! \bmat{
	0               &\! \! \! 0              &\! \! \!0\\
	H_1^{n\times m} &\! \! \! H_1^{n\times m}&\! \! \! H_1^{n\times m}\\
	H_1^{n\times m} &\! \! \! H_1^{n\times m}&\! \! \! H_1^{n\times m}}\! \subset\! \mcl{N}_{2D}^{n\times m},
\end{align*}
and let
\begin{align}\label{eq_div_operator_x}
M=\bmat{M_{00}&M_{01}&M_{02}\\M_{10}&M_{11}&M_{12}\\M_{20}&M_{21}&M_{22}}\in\mcl{N}_{2D}^{n\times m},
\end{align}
where
\begin{align*}
&M_{00}(x,y)=N_{10}(x,y,x)-N_{20}(x,y,x),    \\
&M_{01}(x,y,\nu)=N_{11}(x,y,x,\nu)-N_{21}(x,y,x,\nu),   \\
&M_{02}(x,y,\nu)=N_{12}(x,y,x,\nu)-N_{22}(x,y,x,\nu),   \\
&M_{10}(x,y,\theta)=\partial_x N_{10}(x,y,\theta), \\
&M_{20}(x,y,\theta)=\partial_x N_{20}(x,y,\theta), \\
&M_{11}(x,y,\theta,\nu)=\partial_x N_{11}(x,y,\theta,\nu),\\
&M_{21}(x,y,\theta,\nu)=\partial_x N_{21}(x,y,\theta,\nu),\\
&M_{12}(x,y,\theta,\nu)=\partial_x N_{12}(x,y,\theta,\nu),\\
&M_{22}(x,y,\theta,\nu)=\partial_x N_{22}(x,y,\theta,\nu).
\end{align*}
Then, for any $\mbf{u}\in L_2^{m}[x,y]$,
\begin{align*}
\partial_x(\mcl{P}[N]\mbf{u})(x,y)
=(\mcl{P}[M]\mbf{u})(x,y).\\
\end{align*}
\end{lem}

\begin{proof}
To prove this result, we exploit the linearity of the PI operator, splitting
\begin{align*}
\mcl{P}[N]=\mcl{P}\smallbmat{0&0&0\\N_{10}&0&0\\N_{20}&0&0}
+\mcl{P}\smallbmat{0&0&0\\0&N_{11}&0\\0&N_{21}&0}
+\mcl{P}\smallbmat{0&0&0\\0&0&N_{12}\\0&0&N_{22}}.
\end{align*}
Recall now the Leibniz integral rule, stating that, for arbitrary $P\in L_2[x,\theta]$,
\begin{align*}
&\frac{d}{dx}\left(\int_{a(x)}^{b(x)} P(x,\theta)d\theta\right)
=\int_{a(x)}^{b(x)} \partial_x P(x,\theta)d\theta \\
&\hspace*{2.0cm}+P(x,b(x))\frac{d}{dx}b(x)
-P(x,a(x))\frac{d}{dx}a(x).
\end{align*}
Then, for arbitrary $\mbf{u}\in L_2[x,y]$,
\begin{align*}
&\partial_{x}\left(\mcl{P}\smallbmat{0&0&0\\N_{10}&0&0\\N_{20}&0&0}\mbf{u}\right)(x,y)    \\
&=\partial_x\bbbl(\int_{a}^{x}N_{10}(x,y,\theta)\mbf{u}(\theta)d\theta
+\int_{x}^{b}N_{20}(x,y,\theta)\mbf{u}(\theta)d\theta
\bbbr)    \\
&=N_{10}(x,y,x)\mbf{u}(x)
+\int_{a}^{x}\partial_{x}N_{10}(x,y,\theta)\mbf{u}(\theta)d\theta \\
&\qquad-N_{20}(x,y,x)\mbf{u}(x)
+\int_{x}^{b}\partial_{x}N_{20}(x,y,\theta)\mbf{u}(\theta)d\theta \\
&=\left(\mcl{P}\smallbmat{M_{00}&0&0\\M_{10}&0&0\\M_{20}&0&0}\mbf{u}\right)(x,y)
\end{align*}
Repeating these steps, it also follows that
\begin{align*}
&\partial_{x}\left(\mcl{P}\smallbmat{0&0&0\\0&N_{11}&0\\0&N_{21}&0}\mbf{u}\right)(x,y)
=\left(\mcl{P}\smallbmat{0&M_{01}&0\\0&M_{11}&0\\0&M_{21}&0}\mbf{u}\right)(x,y),\\
&\partial_{x}\left(\mcl{P}\smallbmat{0&0&0\\0&0&N_{12}\\0&0&N_{22}}\mbf{u}\right)(x,y)
=\left(\mcl{P}\smallbmat{0&0&M_{02}\\0&0&M_{12}\\0&0&M_{22}}\mbf{u}\right)(x,y).
\end{align*}
Finally, by linearity of the derivative and PI operators,
\begin{align*}
&\partial_x \left(\mcl{P}[N]\mbf{u}\right)(x,y)
=\partial_x\left(\mcl{P}\smallbmat{0&0&0\\N_{10}&0&0\\N_{20}&0&0}\mbf{u}\right)(x,y) \\
&\quad+\partial_x\left(\mcl{P}\smallbmat{0&0&0\\0&N_{11}&0\\0&N_{21}&0}\mbf{u}\right)(x,y)
+\partial_x\left(\mcl{P}\smallbmat{0&0&0\\0&0&N_{12}\\0&0&N_{22}}\mbf{u}\right)(x,y)  \\
&=\left(\mcl{P}\smallbmat{M_{00}&0&0\\M_{10}&0&0\\M_{20}&0&0}\mbf{u}\right)(x,y)
+\left(\mcl{P}\smallbmat{0&M_{01}&0\\0&M_{11}&0\\0&M_{21}&0}\mbf{u}\right)(x,y)   \\
&\hspace*{2.0cm}+\left(\mcl{P}\smallbmat{0&0&M_{02}\\0&0&M_{12}\\0&0&M_{22}}\mbf{u}\right)(x,y)
=\left(\mcl{P}[M]\mbf{u}\right)(x,y).
\end{align*}

\end{proof}

\begin{lem}\label{lem_div_operator_y}
Suppose
\begin{align*}
N\! =\! \bmat{0&\! \!  N_{01}&\! \! N_{02}\\0&\! \! N_{11}&\! \! N_{12}\\0&\! \! N_{21}&\! \! N_{22}}\! &\in\! \bmat{
	0 &\! H_1^{n\times m} &\! H_1^{n\times m}\\
	0 &\! H_1^{n\times m} &\! H_1^{n\times m}\\
	0 &\! H_1^{n\times m} &\! H_1^{n\times m}}\! 
\subset\!  \mcl{N}_{2D}^{n\times m},
\end{align*}
and let
\begin{align}\label{eq_div_operator_y}
M=\bmat{M_{00}&M_{01}&M_{02}\\M_{10}&M_{11}&M_{12}\\M_{20}&M_{21}&M_{22}}\in\mcl{N}_{2D}^{n\times m},
\end{align}
where
\begin{align*}
&M_{00}(x,y)=N_{01}(x,y,y)-N_{02}(x,y,y),    \\
&M_{10}(x,y,\theta)=N_{11}(x,y,\theta,y)-N_{12}(x,y,\theta,y),   \\
&M_{20}(x,y,\theta)=N_{21}(x,y,\theta,y)-N_{22}(x,y,\theta,y),   \\
&M_{01}(x,y,\nu)=\partial_y N_{01}(x,y,\nu), \\
&M_{02}(x,y,\nu)=\partial_y N_{02}(x,y,\nu), \\
&M_{11}(x,y,\theta,\nu)=\partial_y N_{11}(x,y,\theta,\nu),\\
&M_{21}(x,y,\theta,\nu)=\partial_y N_{21}(x,y,\theta,\nu),\\
&M_{12}(x,y,\theta,\nu)=\partial_y N_{12}(x,y,\theta,\nu),\\
&M_{22}(x,y,\theta,\nu)=\partial_y N_{22}(x,y,\theta,\nu).
\end{align*}
Then for any $\mbf{u}\in L_2^{m}[x,y]$,
\begin{align*}
\partial_y(\mcl{P}[N]\mbf{u})(x,y)
=(\mcl{P}[M]\mbf{u})(x,y).\\
\end{align*}
\end{lem}
\begin{proof}
The proof follows along the same lines as that of Lemma~\ref{lem_div_operator_x}.
\end{proof}

Whenever a 2D-PI operator $\mcl{P}[N]$ has no multipliers along the $x$-direction (as in Lemma~\ref{lem_div_operator_x}), or along the $y$-direction (as in Lemma~\ref{lem_div_operator_y}), we will denote the composition of the differential operator $\partial_x$ or $\partial_y$ with this PI operator in the obvious manner as $\partial_x\mcl{P}[N]$ or $\partial_y\mcl{P}[N]$ respectively, such that, e.g.
\[
 \left[\bl(\partial_x\mcl{P}[N]\br)\mbf{u}\right](x,y)
 =\partial_x\bl(\mcl{P}[N]\mbf{u}\br)(x,y).
\]

\subsection{Adjoint of 2D-PI Operators}\label{sec_adj}

Finally, we give an expression for the adjoint of a 2D-PI operator.

\begin{lem}\label{lem_adjoint}
	Suppose $ N \in \mcl N_{2D}^{n\times m}$ and define $ \hat{N} \in \mcl N_{2D}^{m\times n}$ such that
	\begin{align}\label{adj_operator}
	&\hat{N}(x,y,\theta,\nu)  \nonumber\\
	&=\bmat{\hat{N}_{00}(x,y)&\hat{N}_{01}(x,y,\nu)&\hat{N}_{02}(x,y,\nu)\\ \hat{N}_{10}(x,y,\theta)&\hat{N}_{11}(x,y,\theta,\nu)&\hat{N}_{12}(x,y,\theta,\nu)\\ \hat{N}_{20}(x,y,\theta)&\hat{N}_{21}(x,y,\theta,\nu)&\hat{N}_{22}(x,y,\theta,\nu)} \nonumber\\
	&=\bmat{N_{00}^T(x,y)&N^T_{02}(x,\nu,y)&N^T_{01}(x,\nu,y)\\ N^T_{20}(\theta,y,x)&N^T_{22}(\theta,\nu,x,y)&N^T_{21}(\theta,\nu,x,y)\\ N^T_{10}(\theta,y,x)&N^T_{12}(\theta,\nu,x,y)&N^T_{11}(\theta,\nu,x,y)}.
	\end{align}
	Then for any $\mbf u\in L_2^{m}[x,y]$ and $\mbf{v}\in L_2^{n}[x,y]$,
	\[
	\ip{\mbf{v}}{\mcl{P}[N]\mbf{u}}_{L_2}=\ip{\mcl{P}[\hat{N}]\mbf{v}}{\mbf{u}}_{L_2}.
	\]
\end{lem}
\begin{proof}
	See Appendix~\ref{sec_adjoint_appendix} for a proof.
\end{proof}

\section{A Standardized PDE Format in 2D}\label{sec_PDE}

Having formulated the required hierarchy of algebras of PI operators, we now introduce a class of linear 2D PDEs, the solutions of which may be represented using 2D PIEs. These 2D PDEs are represented in a standardized format, allowing for efficient construction of a general mapping between the PDE and PIE state spaces -- this construction being found in Section~\ref{sec_PIE}.\\

We consider a coupled PDE in the following compact representation
\begin{align}\label{eq_PDE}
 &\dot{\mbf{u}}(t,x,y)=\sum_{i,j=0}^{2}A_{ij}\thinspace\partial_x^i \partial_y^j \bl(N_{\max\{i,j\}}\mbf{u}(t,x,y)\br),
\end{align}
where the matrices
\begin{align*}
 N_0&=I_{n_0+n_1+n_2},    \\
 N_1&=\bmat{0_{(n_1+n_2)\times n_0}&I_{n_1+n_2}}, \\
 N_2&=\bmat{0_{n_2\times n_0}&0_{n_2\times n_1}&I_{n_2}},
\end{align*}
allow us to partition the states according to differentiability, so that
$\mbf{u}(t) \in X(\mcl{B})$ for all $t\ge 0$, where $X(\mcl{B})$ also includes boundary conditions, parameterized by the 011-PI operator $\mcl B:=\mcl{P}[B]$ as
\begin{align}\label{eq_setX}
 X(\mcl{B}):=
 &\left\{\bmat{\mbf{u}_0\\\mbf{u}_{1}\\\mbf{u}_{2}}\in \bmat{L_2^{n_0}\\ H_1^{n_1}\\ H_2^{n_2}}\thinspace\bbbbr\vert~ \mcl{B}\Lambda_{\text{bf}}\mbf{u}=0\right\},
\end{align}
where $B\in \mcl{N}_{011}\smallbmat{n_1+4n_2& 4n_1+16n_2\\n_1+2n_2& 2n_1+4n_2}$ and
where $\Lambda_{\text{bf}}$ allows us to list all the possible boundary values for the state components $\mbf{u}_1$ and $\mbf{u}_2$, and as limited by differentiability. In particular,
\begin{align}\label{eq_Lambda_bf}
 \Lambda_{\text{bf}}=\bmat{\mcl{C}_1\\\mcl{C}_2\\\mcl{C}_3}:L_2^{n_0}\times H_1^{n_1}\times H_2^{n_2}\rightarrow \bmat{\R^{4n_1+16n_2}\\ L_2^{2n_1+4n_2}[x]\\ L_2^{2n_1+4n_2}[y]},
\end{align}
where
\begin{align*}
 \mcl C_1&:=
 \bmat{\begin{array}{l}
	0~~\Lambda_1~\thinspace 0\\
	0~~0~~~\Lambda_1 \\
	0~~0~~~\Lambda_1\partial_x \\
	0~~0~~~\Lambda_1\partial_y  \\
	0~~0~~~\Lambda_1\partial_{xy}
	\end{array}}
 ,\quad
 \bmat{ \mcl C_2\\ \mcl C_3}:=
 \bmat{\begin{array}{l}
 	0~~\Lambda_2 \partial_x~\thinspace 0   \\
 	0~~0~~~~~~\Lambda_2 \partial_x^2  \\
 	0~~0~~~~~~\Lambda_2\partial_x^2\partial_y \\ 
 	0~~\Lambda_3 \partial_y~\thinspace 0\\
 	0~~0~~~~~~\Lambda_3\partial_y^2   \\
 	0~~0~~~~~~\Lambda_3\partial_x\partial_y^2
 	\end{array}}
\end{align*}
where we use the Dirac operators $\Lambda_k$ defined as
\[
\Lambda_1=\bmat{
 \Lambda_x^a\Lambda_y^c\\
 \Lambda_x^b\Lambda_y^c\\
 \Lambda_x^a\Lambda_y^d\\
 \Lambda_x^b\Lambda_y^d
 },\quad \Lambda_2= \bmat{
 \Lambda_y^c    \\
 \Lambda_y^d
 },\quad \Lambda_3= \bmat{
 \Lambda_x^a    \\
 \Lambda_x^b
 }.
\]

This formulation is very general and allows us to express almost any linear 2D PDE. For examples of this representation, see the examples in Section~\ref{sec_examples}.

\paragraph{Definition of Solution} For a given initial condition $\mbf{u}_{\text{I}}\in X(\mcl{B})$, we say that a function $\mbf{u}(t)$ satisfies the PDE defined by $\{A_{ij}, \mcl{B}\}$ if $\mbf{u}$ is Frech\'et differentiable, $\mbf{u}(0)=\mbf{u}_{\text{I}}$,  $\mbf{u}(t)\in X$ for all $t \ge 0$ and Equation~\eqref{eq_PDE} is satisfied for all $t \ge0$.

\begin{defn}\label{defn_PDE_stability}
 We say that a solution $\mbf{u}$ with initial condition $\mbf{u}_{\text{I}}$ of the PDE defined by $\{A_{ij}, \mcl B\}$ is exponentially stable in $L_2$ if there exist constants $K,\gamma>0$ such that
 \[
  \norm{\mbf{u}(t)}_{L_2}\leq K e^{-\gamma t} \norm{\mbf{u}_{\text{I}}}_{L_2}
 \]
 We say the PDE defined by $\{A_{ij}, \mcl{B}\}$ is exponentially stable in $L_2$ if any solution $\mbf{u}$ of the PDE is exponentially stable in $L_2$.
\end{defn}

\section{The Fundamental State on 2D}\label{sec_PIE}
In this section, we provide the main technical result of the paper, wherein we show that for a suitably well-posed set of boundary conditions, $\mcl{ B}$, there exists a unitary 2D-PI operator $\mcl{T}:L_2 \rightarrow X(\mcl{B})$ (where $X(\mcl{B})$ is defined in Eqn.~\eqref{eq_setX}) such that if we define the differentiation operator
\begin{align}\label{eq_Dmap}
\mcl D:=\bmat{I&&\\&\partial_{x}\partial_{y}&\\&&\partial_{x}^2\partial_{y}^2}
\end{align}
then for any ${\mbf u}\in X(\mcl{B})$ and $\hat{\mbf u}\in L_{2}^{n_0+n_1+n_2}$, we have
\[
\mbf u = \mcl T \mcl D \mbf u \qquad \text{and}\qquad \hat{\mbf{u}} = \mcl D \mcl T \hat{\mbf{u}}.
\]
This implies that for any $\mbf{u} \in X(\mcl{B})$, there exists a \textit{unique} $\hat{\mbf{u}} \in L_2$ where the map from $\hat{\mbf{u}}$ to $\mbf{u}$ is defined by a 2D-PI operator. Because differentiation of a PI operator is a PI operator (Section~\ref{sec_div}), this implies that derivatives of $\mbf{u}$ can be expressed in terms of a PI operator acting on $\hat{\mbf{u}}$. Using these results, in Thm.~\ref{thm_Tmap}, we show  that for any suitable PDE defined by $\{A_{ij},\mcl{B}\}$, there exist 2D-PI operators $\mcl{T},\mcl{A}$ such that $\hat{\mbf{u}}\in L_2$ satisfies 
\[
\mcl{T}\dot{\hat{\mbf{u}}}(t)=\mcl{A}\hat{\mbf{u}}(t)
\]
if and only if $\mcl{T}\hat{\mbf{u}}\in X(\mcl{B})$ satisfies the PDE.

\begin{figure*}[!ht]
	\hrulefill
	\footnotesize
	\begin{flalign}
	&K_1=\bmat{K_{30} \\ \{0,K_{31},0\} \\ \{0,K_{32},0\}}\in \bmat{L_2 \\ \mcl{N}_{1D\rightarrow 2D} \\ \mcl{N}_{1D\rightarrow 2D}},		&
	&K_2= \bmat{T_{00}&0&0\\0&K_{33}&0\\0&0&0}\in\mcl{N}_{2D},	 \label{eq_K_mat} \\
	H_1 &= \bmat{H_{00}&H_{01}&H_{02}\\0&\{H_{11},0,0\}&0\\0&0&\{H_{22},0,0\}}\in\mcl{N}_{011},	&
	H_2 &= \bmat{H_{03}\\\{H_{13},0,0\}\\\{H_{23},0,0\}}\in\bmat{L_2\\\mcl{N}_{2D\rightarrow 1D}\\\mcl{N}_{2D\rightarrow 1D}},		&	& \label{eq_H_mat}
	\end{flalign}
	where $T_{00}$ is as defined in~\eqref{eq_Tmat} in Fig.~\ref{fig_Tmap_matrices}, and
	{\tiny
		\begin{align}
		&\mat{K_{31}(x,y,\theta)=
			\bmat{
				0&0&0\\
				I_{n_1}&0&0\\
				0&(x-\theta)&(y-c)(x-\theta)
			},		&
			&\hspace*{1.65cm}K_{32}(x,y,\nu)=
			\bmat{
				0&0&0\\
				I_{n_1}&0&0\\
				0&(y-\nu)&(x-a)(y-\nu)
			}, 	&
			&T_{00}=\bmat{I_{n_0}&0&0\\0&0&0\\0&0&0}},	\nonumber\\
		&\mat{K_{30}(x,y)=
			\bmat{
				0&0&0&0&0\\
				I_{n_1}&0&0&0&0\\
				0&I_{n_2}&(x-a)&(y-c)&(y-c)(x-a)
			},	&	
			\hspace{2.0cm}
			&K_{33}(x,y,\theta,\nu)=
			\bmat{
				0&0&0\\
				0&I_{n_1}&0\\
				0&0&(x-\theta)(y-\nu)
		}}.  \nonumber
		\end{align}}	
	and where,
	{\tiny
		\begin{align*}
		&H_{00} =
		\bmat{I_{n_1}&0&0&0&0\\
			I_{n_1}&0&0&0&0\\
			I_{n_1}&0&0&0&0\\
			I_{n_1}&0&0&0&0\\
			0&I_{n_2}&0&0&0\\
			0&I_{n_2}&(b-a)&0&0\\
			0&I_{n_2}&0&(d-c)&0\\
			0&I_{n_2}&(b-a)&(d-c)&(d-c)(b-a)\\
			0&0&I_{n_2}&0&0\\
			0&0&I_{n_2}&0&0\\
			0&0&I_{n_2}&0&(d-c)\\
			0&0&I_{n_2}&0&(d-c)\\
			0&0&0&I_{n_2}&0\\
			0&0&0&I_{n_2}&(b-a)\\
			0&0&0&I_{n_2}&0\\
			0&0&0&I_{n_2}&(b-a)\\
			0&0&0&0&I_{n_2}\\
			0&0&0&0&I_{n_2}\\
			0&0&0&0&I_{n_2}\\
			0&0&0&0&I_{n_2}},    &
		&H_{01}(x) =
		\bmat{
			0&0&0\\I_{n_1}&0&0\\0&0&0\\I_{n_1}&0&0\\
			0&0&0\\ 0&(b-x)&0\\ 0&0&0\\
			0&(b-x)&(d-c)(b-x)\\
			0&0&0\\ 0&I_{n_2}&0 \\0&0&0\\ 0&I_{n_2}&(d-c)\\
			0&0&0\\ 0&0&(b-x) \\0&0&0\\ 0&0&(b-x)\\
			0&0&0\\ 0&0&I_{n_2} \\0&0&0\\ 0&0&I_{n_2}
		},   &
		&H_{02}(y) =
		\bmat{
			0&0&0\\0&0&0\\I_{n_1}&0&0\\I_{n_1}&0&0\\
			0&0&0\\ 0&0&0\\ 0&(d-y)&0\\
			0&(d-y)&(b-a)(d-y)\\
			0&0&0\\ 0&0&0 \\0&0&(d-y)\\ 0&0&(d-y)\\
			0&0&0\\ 0&0&0 \\0&I_{n_2}&0\\ 0&I_{n_2}&(b-a)\\
			0&0&0\\ 0&0&0 \\0&0&I_{n_2}\\ 0&0&I_{n_2}
		},   \nonumber\\
		&\begin{array}{l}
		H_{11}=
		\bmat{
			I_{n_1}&0&0\\I_{n_1}&0&0\\
			0&I_{n_2}&0\\0&I_{n_2}&(d-c)\\0&0&I_{n_2}\\0&0&I_{n_2}
		},\\
		\\
		H_{22}=
		\bmat{
			I_{n_1}&0&0\\I_{n_1}&0&0\\
			0&I_{n_2}&0\\0&I_{n_2}&(b-a)\\0&0&I_{n_2}\\0&0&I_{n_2}
		}
		,
		\end{array}
		&
		&\begin{array}{l}
		H_{13}(y)=
		\bmat{
			0&0&0\\0&I_{n_1}&0\\
			0&0&0\\ 0&0&(d-y)\\0&0&0\\0&0&I_{n_2}
		},\\
		\\
		H_{23}(x)=
		\bmat{
			0&0&0\\0&I_{n_1}&0\\
			0&0&0\\ 0&0&(b-x)\\0&0&0\\0&0&I_{n_2}
		},
		\end{array}
		&
		&H_{03}(x,y)=
		\bmat{
			0&0&0\\0&0&0\\0&0&0\\0&I_{n_1}&0\\
			0&0&0\\0&0&0\\0&0&0\\ 0&0&(d-y)(b-x)\\0&0&0\\0&0&0\\0&0&0\\ 0&0&(d-y)\\0&0&0\\0&0&0\\0&0&0\\ 0&0&(b-x)\\
			0&0&0\\0&0&0\\0&0&0\\0&0&I_{n_2}
		}.
		\end{align*}
	}
	\hrulefill
	\caption{Parameters $K_1$, $K_2$, $H_1$ and $H_2$ defining the mappings in Lemma~\ref{lem_uhat_to_u} and in Corollary~\ref{cor_uhat_to_BC}}
	\label{fig_H_K_matrices}
\end{figure*}

\subsection{Map From Fundamental State to PDE State}\label{sec_uhat2u}

As mentioned above, given the boundary-constrained ``PDE state'' $\mbf{u}\in X(\mcl{B})$, we associate a corresponding ``fundamental state'' $\hat{\mbf{u}}\in L_2^{n_0+n_1+n_2}$, defined as
\begin{align}\label{eq_uhatu}
 \hat{\mbf{u}}(t):=\bmat{\hat{\mbf{u}}_{0}(t)\\\hat{\mbf{u}}_{1}(t)\\\hat{\mbf{u}}_{2}(t)}=\bmat{I&&\\&\partial_{x}\partial_{y}&\\&&\partial_{x}^2\partial_{y}^2}\bmat{\mbf{u}_0(t)\\\mbf{u}_1(t)\\\mbf{u}_2(t)}
\end{align}
In the following lemma, we temporarily ignore boundary conditions and use the fundamental theorem of calculus to express any $\mbf{u}\in L_2\times H_1\times H_2$ in terms of $\hat{\mbf{u}}$, and a set of ``core'' boundary values.

\begin{lem}\label{lem_uhat_to_u}
Let $\mbf{u}\in L_2[x,y] \times H_1[x,y] \times H_2[x,y]$. If $\mcl K_1 = \mcl P[K_1]$ and $\mcl K_2 = \mcl P[K_2]$ with $K_1 \in L_2\times \mcl N_{1D \rightarrow 2D}\times \mcl N_{1D \rightarrow 2D}$ and $K_2 \in \mcl N_{2D}$ as defined in Eqn.~\eqref{eq_K_mat} in Fig.~\ref{fig_H_K_matrices}, then
\[
\mbf u = \mcl K_1 \Lambda_{\text{bc}} \mbf u+\mcl K_2 \hat{\mbf u}
\]
 where
$\hat{\mbf u} = \mcl D \mbf u$ and
\[
\Lambda_{\text{bc}} :=\bmat{\begin{array}{lll}
0&\Lambda_x^a \Lambda_y^c&0\\
0&0&\Lambda_x^a\Lambda_y^c \\
0&0&\Lambda_x^a\Lambda_y^c \partial_x   \\
0&0&\Lambda_x^a\Lambda_y^c  \partial_y   \\
0&0&\Lambda_x^a\Lambda_y^c  \partial_x\partial_y  \\
0& \Lambda_y^c \partial_x & 0\\
0&0&\Lambda_y^c  \partial_x^2  \\
0&0&\Lambda_y^c  \partial_x^2\partial_y  \\
0& \Lambda_x^a \partial_y & 0\\
0&0& \Lambda_x^a   \partial_y^2 \\
0&0& \Lambda_x^a \partial_x \partial_y^2
\end{array}}
\]
\end{lem}

\begin{proof}
	The proof follows directly from the identities
	\begin{align*}
	\mbf{u}(x,y) &= \mbf{u}(a,c) + \int_{a}^{x}\mbf{u}_x(\theta,c)d\theta    \\
	&\quad+ \int_{c}^{y} \mbf{u}_{y}(a,\nu)d\nu
	+ \int_{c}^{y}\int_{a}^{x}\mbf{u}_{xy}(\theta,\nu)d\theta d\nu, & &
	\end{align*}
	and
	\begin{align*}
	\mbf{u}(x,y) &= \mbf{u}(a,c)
	+ (x-a)\mbf{u}_{x}(a,c)      \nonumber\\
	&\hspace*{0.25cm}+ (y-c)\mbf{u}_{y}(a,c)
	+ (y-c)(x-a)\mbf{u}_{xy}(a,c)    \nonumber\\
	&\hspace*{0.5cm}+\int_{a}^{x}(x-\theta)\mbf{u}_{xx}(\theta,c)d\theta \nonumber\\
	&\hspace*{0.75cm}+\int_{c}^{y}(y-\nu)\mbf{u}_{yy}(a,\nu)d\nu    \nonumber\\
	&\hspace*{1.0cm}+(y-c)\int_{a}^{x}(x-\theta)\mbf{u}_{xxy}(\theta,c)d\theta \nonumber\\
	&\hspace*{1.25cm}+(x-a)\int_{c}^{y}(y-\nu)\mbf{u}_{xyy}(a,\nu)d\nu    \nonumber\\
	&\hspace*{1.5cm}+\int_{c}^{y}\int_{a}^{x}(y-\nu)(x-\theta)
	\mbf{u}_{xxyy}(\theta,\nu)d\theta d\nu,
	\end{align*}
	which follow from the fundamental theorem of calculus.
\end{proof}

\begin{cor}\label{cor_uhat_to_BC}
Let $\mbf{u}\in L_2[x,y] \times H_1[x,y] \times H_2[x,y]$. Then, if $\mcl{H}_1 = \mcl{P}[H_1]$ and $\mcl{H}_2 = \mcl{P}[H_2]$ with $H_1 \in \mcl N_{011}$ and $H_2 \in L_2 \times \mcl N_{2D \rightarrow 1D}\times \mcl N_{2D \rightarrow 1D}$ as defined in Eqn.~\eqref{eq_H_mat} in Fig.~\ref{fig_H_K_matrices}, then
\[
\Lambda_{\text{bf}} \mbf u = \mcl H_1 \Lambda_{\text{bc}} \mbf u+\mcl H_2 \hat{\mbf u}
\]
where
$\hat{\mbf u} = \mcl D \mbf u$, and $\Lambda_{\text{bf}}$ is as defined in Eqn.~\eqref{eq_Lambda_bf}.
\end{cor}

With these definitions, we can express an arbitrary PDE state $\mbf{u}\in X(\mcl{B})$ in terms of a corresponding state $\hat{\mbf{u}}\in L_2^{n_0\times n_1\times n_2}$ and $\Lambda_{\text{bc}}\mbf{u}$. In the following theorem, we describe this relation as a PI operator, incorporating the boundary conditions to describe a map from $L_2^{n_0\times n_1\times n_2}$ to $X(\mcl{B})$. In doing so, we will require the operator $\mcl{B}$ to be well-posed, defining sufficient boundary conditions for the solution to the PDE to be uniquely defined. We express this restriction through invertibility of a 011-PI operator.

\begin{defn}\label{defn_well-posed}
Let $\mcl{H}_1=\mcl{P}[H_1]$, where $H_1\in\mcl{N}_{011}$ is as defined in Eqn.~\eqref{eq_H_mat} in Fig.~\ref{fig_H_K_matrices}. Then, we say that $\mcl{B}=\mcl{P}[B]$ for $B\in\mcl{N}_{011}$ defines a set of \textit{well-posed} boundary conditions if the operator $\mcl{B}\mcl{H}_1$ is invertible, so that there exist parameters $\hat{E}\in\mcl{N}_{011}$ such that $\mcl{P}[\hat{E}]=(\mcl{B}\mcl{H}_1)^{-1}$. 
\end{defn}

\begin{figure*}[!ht]
	\hrulefill\\
	\footnotesize
	Define
	\begin{flalign}\label{eq_Tmat}
	T:=\bmat{T_{00}&0&0\\0&T_{11}&T_{12}\\0&T_{21}&T_{22}}\in\mcl{N}_{2D}
	\end{flalign}
	where
	\begin{flalign}
	&\mat{\begin{array}{l}
	T_{11}(x,y,\theta,\nu) = K_{33}(x,y,\theta,\nu)
	+ T_{21}(x,y,\theta,\nu) + T_{12}(x,y,\theta,\nu)
	-T_{22}(x,y,\theta,\nu), \nonumber\\
	T_{21}(x,y,\theta,\nu) = - K_{32}(x,y,\nu)G_{23}^{0}(\theta,\nu) + T_{22}(x,y,\theta,\nu), \nonumber\\
	T_{12}(x,y,\theta,\nu) = - K_{31}^{1}(x,y,\theta)G_{13}^{0}(\theta,\nu) + T_{22}(x,y,\theta,\nu), \nonumber\\
	T_{22}(x,y,\theta,\nu) = -K_{30}(x,y)G_{03}(\theta,\nu)
	-\int_{a}^{x}K_{31}(x,y,\eta)G_{13}^{1}(\eta,\nu,\theta)d\eta
	-\int_{c}^{y}K_{32}(x,y,\mu)G_{23}^{1}(\theta,\mu,\nu)d\mu, 
	\end{array}	&
	T_{00}=\bmat{I_{n_0}&0&0\\0&0&0\\0&0&0}},
	\intertext{
	with the functions $K_{ij}$ as defined in~\eqref{eq_K_mat}, and}
	&\enspace G_{0}(x,y)=\hat{E}_{00}F_{0}(x,y) + \hat{E}_{01}(x)F_{1}^{0}(x,y) + \int_{a}^{b}\hat{E}_{01}(\theta)F_{1}^{1}(\theta,y,x)d\theta
	+ \hat{E}_{02}(y)F_{2}^{0}(x,y) + \int_{c}^{d}\hat{E}_{02}(\nu)F_{2}^{1}(x,\nu,y),    \nonumber\\
	&
	\begin{array}{l}
	G_{1}^{0}(x,y)=\hat{E}_{11}^{0}(x)F_{1}^{0}(x,y),
	\\
	G_{1}^{1}(x,y,\theta)=\hat{E}_{10}(x)F_{0}(\theta,y) + \hat{E}_{11}^{0}(x)F_{1}^{1}(x,y,\theta) 
	\\
	\quad+ \hat{E}_{11}^{1}(x,\theta)F_{1}^{0}(\theta,y)  
	+\int_{a}^{b}\hat{E}_{11}^{1}(x,\eta)F_{1}^{1}(\eta,y,\theta)d\eta
	\\
	\qquad+\hat{E}_{12}(x,y)F_{2}^{0}(\theta,y) + \int_{c}^{d}\hat{E}_{12}(x,\nu)F_{2}^{1}(\theta,\nu,y)d\nu,
	\end{array}
	\hspace{2.5cm}
	\begin{array}{l}
	G_{2}^{0}(x,y)=\hat{E}_{22}^{0}(y)F_{2}^{0}(x,y), 
	\\
	G_{2}^{1}(x,y,\nu)=\hat{E}_{20}(y)F_{0}(x,\nu) + \hat{E}_{22}^{0}(y)F_{2}^{1}(x,y,\nu) 
	\\
	\quad+ \hat{E}_{22}^{1}(y,\nu)F_{2}^{0}(x,\nu)  
	+\int_{c}^{d}\hat{E}_{22}^{1}(y,\mu)F_{2}^{1}(x,\mu,\nu)d\mu
	\\
	\qquad+\hat{E}_{21}(x,y)F_{1}^{0}(x,\nu) + \int_{a}^{b}\hat{E}_{21}(\theta,y)F_{1}^{1}(\theta,\nu,x)d\theta,
	\end{array} \label{eq_G_mat}
	\intertext{with}
	&\begin{array}{l l}
	F_{0}(x,y)=B_{00}H_{03}(x,y) + B_{01}(x)H_{13}(x,y) + B_{02}(y)H_{23}(x,y),    \\
	F_{1}^{1}(x,y,\theta)=B_{10}(x)H_{03}(\theta,y) + B_{11}^{1}(x,\theta)H_{13}(\theta,y)
	+B_{12}(x,y)H_{23}(\theta,y),   & \quad
	F_{1}^{0}(x,y)=B_{11}^{0}(x)H_{13}(x,y), \\
	F_{2}^{1}(x,y,\nu)=B_{20}(y)H_{03}(x,\nu) + B_{22}^{1}(y,\nu)H_{23}(x,\nu)
	+B_{21}(x,y)H_{13}(x,\nu),   & \quad
	F_{2}^{0}(x,y)=B_{22}^{0}(y)H_{23}(x,y),
	\end{array} \label{eq_F_mat}
	\intertext{and $\smallbmat{\hat{E}_{00}&\hat{E}_{01}&\hat{E}_{02}\\\hat{E}_{10}&\hat{E}_{11}&\hat{E}_{12}\\\hat{E}_{20}&\hat{E}_{21}&\hat{E}_{22}}=\mcl{L}_{\text{inv}}\left(\smallbmat{E_{00}&E_{01}&E_{02}\\E_{10}&E_{11}&E_{12}\\E_{20}&E_{21}&E_{22}}\right)\in\mcl{N}_{011}$, where $\mcl{L}_{\text{inv}}:\mcl{N}_{011}\rightarrow\mcl{N}_{011}$ is defined as in Equation~\eqref{eq_inv_011_appendix} in~\ref{sec_inv_appendix}, and} &\enspace E_{11}=\{E_{11}^{0},E_{11}^{1},E_{11}^{1}\}\in\mcl{N}_{1D},\qquad \text{and}\qquad E_{22}=\{E_{22}^{0},E_{22}^{1},E_{22}^{1}\}\in\mcl{N}_{1D},\qquad \text{with}\nonumber\\
	&\begin{array}{l l}
	E_{00}=B_{00}H_{00} + \int_{a}^{b}B_{01}(x)H_{10}(x)dx + \int_{c}^{d}B_{02}(y)H_{20}(y)dy, \\
	E_{01}(x)=B_{00}H_{01}(x) + B_{01}(x)H_{11}(x),  &
	E_{02}(y)=B_{00}H_{02}(y) + B_{02}(y)H_{22}(y),  \\
	E_{10}(x)=B_{10}(x)H_{00}, &
	E_{20}(y)=B_{20}(y)H_{00}, \\
	E_{11}^{0}(x)=B_{11}^{0}H_{11},   &
	E_{22}^{0}(y)=B_{22}^{0}H_{22},   \\
	E_{11}^{1}(x,\theta)= B_{10}(x)H_{01}(\theta) + B_{11}^{1}(x,\theta)H_{11}, &
	E_{22}^{1}(y,\nu)= B_{20}(y)H_{02}(\nu) + B_{22}^{1}(y,\nu)H_{22}, \\
	E_{12}(x,y)=B_{10}(x)H_{02}(y) + B_{12}(x,y)H_{22}(y), &
	E_{21}(x,y)=B_{20}(y)H_{01}(x) + B_{21}(x,y)H_{11}(x), 
	\end{array}\label{eq_E_mat}
	\end{flalign}
	where the functions $H_{ij}$ are as defined in~\eqref{eq_H_mat}.\\
	
	\hrulefill
	\caption{Parameters $T$ describing PI operator $\mcl{T}=\mcl{P}[T]$ mapping the fundamental state back to the PDE state in Theorem~\ref{thm_Tmap}}
	\label{fig_Tmap_matrices}
\end{figure*}

%

\begin{thm}\label{thm_Tmap}
Let $B=\smallbmat{B_{00}&B_{01}&B_{02}\\B_{10}&B_{11}&B_{12}\\B_{20}&B_{21}&B_{22}} \in \mcl N_{011}$ be given, where $B_{11}=\{B_{11}^{0},B_{11}^{1},B_{11}^{1}\}\in\mcl{N}_{1D}$ and $B_{22}=\{B_{22}^{0},B_{22}^{1},B_{22}^{1}\}\in\mcl{N}_{1D}$ are separable, and such that $\mcl{B}:=\mcl{P}[B]$ defines a set of well-posed boundary conditions as $\mcl{B}\Lambda_{\text{bf}}\mbf{u}=0$. Let associated parameters $T \in \mcl N_{2D}$ be as defined in Eqn.~\eqref{eq_Tmat} in Fig~\ref{fig_Tmap_matrices}. Then if $\mcl{T}=\mcl{P}[T]$, for any $\mbf{u}\in X(\mcl{B})$ and $\hat{\mbf u} \in L_2[x,y]$, we have
\begin{align*}
 \mbf{u}=\mcl{T}\mcl{D}\mbf{u} \qquad \text{and} \qquad   \hat{\mbf{u}}=\mcl{D}\mcl{T}\hat{\mbf{u}}.\\
\end{align*}
\end{thm}

\begin{proof}
	Suppose $\mbf{u}\in X(\mcl{B})$, and define $\hat{\mbf{u}}=\mcl{D}\mbf{u}\in L_2^{n_0+n_1+n_2}[x,y]$. Furthermore,
	let $K_1$ and $K_2$ be as defined in Eqn.~\eqref{eq_K_mat}, and $H_1$ and $H_2$ be as defined in Eqn.~\eqref{eq_H_mat}, such that (by Lemma~\ref{lem_uhat_to_u} and Corollary~\ref{cor_uhat_to_BC})
	\begin{align}\label{eq_uhat_to_u}
	\Lambda_{\text{bf}} \mbf{u} &= \mcl{H}_1 \Lambda_{\text{bc}} \mbf{u}+\mcl{H}_2 \hat{\mbf{u}} \nonumber\\
	\mbf{u} &= \mcl{K}_1 \Lambda_{\text{bc}} \mbf{u}+\mcl{K}_2 \hat{\mbf{u}},
	\end{align}
	where $\mcl{H}_1=\mcl{P}[H_1]$, $\mcl{H}_2=\mcl{P}[H_2]$, $\mcl{K}_1=\mcl{P}[K_1]$, and $\mcl{K}_2=\mcl{P}[K_2]$. Enforcing the boundary conditions $\mcl{B}\Lambda_{\text{bf}}\mbf{u}=0$, we may use the composition rules of PI operators to express
	\begin{align*}
	0=\mcl{B}\Lambda_{\text{bf}}\mbf{u}
	=\mcl{B}\mcl{H}_1 \Lambda_{\text{bc}} \mbf{u}+\mcl{B}\mcl{H}_2 \hat{\mbf{u}}
	=\mcl{E}\Lambda_{\text{bc}}\mbf{u}
	+\mcl{F}\hat{\mbf{u}},
	\end{align*}
	where $\mcl{E}=\mcl{P}[E]$ and $\mcl{F}=\mcl{P}[F]$, with
	\begin{align*}
	E&\! =\! \bmat{E_{00}&E_{01}&E_{02}\\E_{10}&E_{11}&E_{12}\\E_{20}&E_{21}&E_{22}}\! \in\mcl{N}_{011},    &
	F&\! =\! \bmat{F_{0}\\F_{1}\\F_{2}}\! \in\bmat{L_2\\\mcl{N}_{2D\rightarrow 1D}\\\mcl{N}_{2D\rightarrow 1D}},
	\end{align*}
	and
	\begin{align*}
	E_{11}&=\{E_{11}^{0},E_{11}^{1},E_{11}^{1}\}\in\mcl{N}_{1D}
	\\
	E_{22}&=\{E_{22}^{0},E_{22}^{1},E_{22}^{1}\}\in\mcl{N}_{1D} \\
	F_{1}&=\{F_{1}^{0},F_{1}^{1},F_{1}^{1}\}\in\mcl{N}_{2D\rightarrow 1D} \\
	F_{2}&=\{F_{2}^{0},F_{2}^{1},F_{2}^{1}\}\in\mcl{N}_{2D\rightarrow 1D}
	\end{align*}
	defined as in Equations~\eqref{eq_F_mat} and~\eqref{eq_E_mat}. By well-posedness of the boundary conditions, operator $\mcl{E}$ is invertible, so that the boundary state $\Lambda_{\text{bc}}\mbf{u}$ may be expressed directly in terms of the fundamental state $\hat{\mbf{u}}$ as
	\begin{align*}
	\Lambda_{\text{bc}}\mbf{u}=-\mcl{E}^{-1}\mcl{F}\hat{\mbf{u}}=-\mcl{G}\hat{\mbf{u}},
	\end{align*}
	where $\mcl{G}=\mcl{P}[G]$ with
	\begin{align*}
	G&=\bmat{G_{0}\\G_{1}\\G_{2}}\in\bmat{L_2\\\mcl{N}_{2D\rightarrow 1D}\\\mcl{N}_{2D\rightarrow 1D}}\\
	G_{1}&=\{G_{1}^{0},G_{1}^{1},G_{1}^{1}\}\in\mcl{N}_{2D\rightarrow 1D} \\
	G_{2}&=\{G_{2}^{0},G_{2}^{1},G_{2}^{1}\}\in\mcl{N}_{2D\rightarrow 1D}
	\end{align*}
	defined as in Equations~\eqref{eq_G_mat}. Finally, substituting this expression into Equation~\eqref{eq_uhat_to_u}, we obtain
	\begin{align*}
	\mbf{u} = \mcl{K}_1 \Lambda_{\text{bc}} \mbf{u}+\mcl{K}_2 \hat{\mbf{u}}
	&=-\mcl{K}_1 \mcl{G}\hat{\mbf{u}}+\mcl{K}_2 \hat{\mbf{u}}	\\
	&=(\mcl{K}_2-\mcl{K}_1\mcl{G})\hat{\mbf{u}}
	=\mcl{T}\hat{\mbf{u}}=\mcl{T}\mcl{D}\mbf{u},
	\end{align*}
	as desired.	
	
	The converse result $\hat{\mbf{u}}=\mcl{D}\mcl{T}\hat{\mbf{u}}$ for $\hat{\mbf{u}}\in L_2$ may be derived using the composition rules for differential operators with PI operators (Lemmas~\ref{lem_div_operator_x} and ~\ref{lem_div_operator_y}), showing that $\mcl{D}\mcl{T}$ describes an identity operation. A full proof of this result can be found in Appendix~\ref{sec_Tmap_appendix}.

\end{proof}

\begin{cor}
	Let $\mcl T$ be as defined in Theorem~\ref{thm_Tmap}. Then $\mcl T:L_2 \rightarrow X(\mcl{B})$ is unitary with respect to
	\[
	\ip{\mbf u}{\mbf v}_X := \ip{\mcl{D}\mbf{u}}{\mcl{D}\mbf{v}}_{L_2}
	\]
\end{cor}
\begin{proof}
	By Thm.~\ref{thm_Tmap}, for any $\mbf{u}\in X(\mcl{B})$, there exists $\hat{\mbf{u}}=\mcl{D}\mbf{u}\in L_2$ such that $\mbf{u}=\mcl{T}\hat{\mbf{u}}$, hence $\mcl{T}$ is surjective. Furthermore, for any $\hat{\mbf{u}},\hat{\mbf{v}}\in L_2$, 
	\begin{align*}
	 \ip{\mcl{T}\hat{\mbf{u}}}{\mcl{T}\hat{\mbf{v}}}_{X}=\ip{\mcl{D}\mcl{T}\hat{\mbf{u}}}{\mcl{D}\mcl{T}\hat{\mbf{v}}}_{L_2}=\ip{\hat{\mbf{u}}}{\hat{\mbf{v}}}_{L_2},
	\end{align*}
	concluding the proof.
\end{proof}

\subsection{PDE to PIE conversion}\label{sec_PDE2PIE}
We now demonstrate that, given a PDE defined by $\{A_{ij},\mcl B\}$, for appropriate choice of $\mcl A,\mcl T$, we may define a Partial Integral Equation (PIE) whose solutions are equivalent to those of the PDE. Specifically, for given PI operators $\mcl T,\mcl A$, and an initial condition $\hat{\mbf u}_{\text{I}}$, we say $\hat{\mbf u}(t)$ solves the PIE defined by $\{\mcl T,\mcl A\}$ for initial condition $\hat{\mbf u}_{\text{I}}$ if $\hat{\mbf u}(0)=\hat{\mbf u}_{\text{I}}$, $\hat{\mbf{u}}(t)\in L_2^{n}[x,y]$ for all $t\ge 0$ and for all $t\ge 0$
\begin{align}\label{PIE}
 \mcl{T}\dot{\hat{\mbf{u}}}(t)=\mcl{A}\hat{\mbf{u}}(t).
\end{align}
The following result shows that if $\mcl T$ is as defined in Theorem~\ref{thm_Tmap}, then $\mbf u(t)$ satisfies the PDE defined by $\{A_{ij},\mcl B\}$ if and only if $\hat{\mbf u}(t) = \mcl T \mbf u(t)$ satisfies the PIE defined by $\{\mcl{A},\mcl{T}\}$.

\begin{lem}\label{lem_PIE}
 Suppose $\mcl T$ is as defined in Theorem~\ref{thm_Tmap}, and let 
 \begin{align}\label{eq_A_operator}
 &\mcl{A}=\sum_{i,j=0}^{2}A_{ij}\bbl(\partial_x^i \partial_y^j\bbl[N_{\max\{i,j\}}\mcl{T}\bbr]\bbr),
 \end{align}
 where
 \begin{align*}
  N_0&=I_{n_0+n_1+n_2},    \\
  N_1&=\bmat{0_{(n_1+n_2)\times n_0}&I_{n_1+n_2}}. \\
  N_2&=\bmat{0_{n_2\times n_0}&0_{n_2\times n_1}&I_{n_2}}.
 \end{align*}
 Then, given $\hat{\mbf{u}}_{\text{I}}\in L_2^{n_0+n_1+n_2}[x,y]$, $\hat{\mbf{u}}(t)$ solves the PIE~\eqref{PIE} defined by $\{\mcl{T},\mcl{A}\}$ with the initial condition $\hat{\mbf{u}}_{\text{I}} $ if and only if $\mbf{u}(t)=\mcl{T}\hat{\mbf{u}}(t)$ satisfies the PDE defined by $\{A_{ij},\mcl B\}$ with the initial condition $\mbf{u}_{\text{I}}=\mcl{T}\hat{\mbf{u}}_{\text{I}}$.\\
\end{lem}

\begin{proof}
	
	Let $\hat{\mbf{u}}\in L_2$ be such that $\mbf{u}=\mcl{T}\hat{\mbf{u}}\in X(\mcl{B})$ is a solution to the PDE defined by $\{A_{ij},\mcl{B}\}$, and with initial condition $\mbf{u}_{\text{I}}=\mcl{T}\hat{\mbf{u}}_{\text{I}}$. Then, by Theorem~\ref{thm_Tmap},
	\[
	\hat{\mbf{u}}(t=0)=\mcl{D}\mbf{u}(t=0)=\mcl{D}\mbf{u}_{\text{I}}=\mcl{D}\mcl{T}\hat{\mbf{u}}_{\text{I}}=\hat{\mbf{u}}_{\text{I}},
	\]
	and, invoking Eqn.~\eqref{eq_PDE} for the standardized PDE,
	\begin{align*}
	\mcl{T}\dot{\hat{\mbf{u}}}(t)&=\dot{\mbf{u}}(t)    \\
	&=\sum_{i,j=0}^{2}A_{ij}\bbl(\partial_x^i \partial_y^j\bbl[N_{\max\{i,j\}}\mbf{u}(t)\bbr]\bbr)    \\
	&=\sum_{i,j=0}^{2}A_{ij}\bbl(\partial_x^i \partial_y^j\bbl[N_{\max\{i,j\}}\mcl{T}\hat{\mbf{u}}(t)\bbr]\bbr)   \\
	&=\sum_{i,j=0}^{2}A_{ij}\bbl(\partial_x^i \partial_y^j\bbl[N_{\max\{i,j\}}\mcl{T}\bbr]\bbr)\hat{\mbf{u}}(t)
	=\mcl{A}\hat{\mbf{u}}(t),
	\end{align*}
	suggesting $\hat{\mbf{u}}$ is a solution to the PIE defined by $\{\mcl{T},\mcl{A}\}$.\\
	
	Conversely, let $\hat{\mbf{u}}$ be a solution to the PIE defined by $\{\mcl{T},\mcl{A}\}$, and with initial condition $\hat{\mbf{u}}_{\text{I}}$. Then, by Theorem~\ref{thm_Tmap},
	\[
	\mbf{u}(t=0)=\mcl{T}\hat{\mbf{u}}(t=0)=\mcl{T}\hat{\mbf{u}}_{\text{I}}=\mbf{u}_{\text{I}},
	\]
	and, invoking Eqn.~\eqref{eq_PDE} for the standardized PDE,
	\begin{align*}
	\dot{\mbf{u}}(t)&=\mcl{T}\dot{\hat{\mbf{u}}}(t)    
	=\mcl{A}\hat{\mbf{u}}(t)  \\
	&=\left[\sum_{i,j=0}^{2}A_{ij}\bbl(\partial_x^i \partial_y^j\bbl[N_{\max\{i,j\}}\mcl{T}\bbr]\bbr)\right]\hat{\mbf{u}}(t)    \\
	&=\sum_{i,j=0}^{2}A_{ij}\bbl(\partial_x^i \partial_y^j\bbl[N_{\max\{i,j\}}\mcl{T}\hat{\mbf{u}}(t)\bbr]\bbr) \\
	&=\sum_{i,j=0}^{2}A_{ij}\bbl(\partial_x^i \partial_y^j\bbl[N_{\max\{i,j\}}\mbf{u}(t)\bbr]\bbr),
	\end{align*}
	suggesting $\mbf{u}=\mcl{T}\hat{\mbf{u}}$ is a solution to the PDE defined by $\{A_{ij},\mcl{B}\}$, as desired.

\end{proof}

Specific examples of PDEs and their PIE equivalents are given in Section~\ref{sec_examples}. In the following section, we propose stability conditions for the PIE which can be enforced using LMIs.

\section{Stability as an LPI}\label{sec_stability_LPI}

Having derived an equivalent PIE representation of PDEs, we now show how this representation can be used for stability analysis. First, we show that existence of a quadratic Lyapunov function for a PIE can be posed as a convex Linear PI Inequality (LPI) optimization problem, with variables of the form $\mcl{P}=\mcl{P}[P]$ for $P\in\mcl{N}_{2D}$, and inequality constraints of the form $\mcl{P}\geq 0$.
Next, we show how to use LMIs to parameterize the cone of positive semidefinite 2D-PI operators - allowing us to test the Lyapunov stability criterion. Finally, we will discuss a PIETOOLS numerical implementation of this stability test, which will be applied to several numerical examples in Section~\ref{sec_examples}.

\subsection{Lyapunov Stability Criterion}\label{sec_lyapunov_stability_criterion}

We first express the problem of existence of a quadratic Lyapunov function as an LPI, whose feasibility implies stability of the associated PIE and PDE. Specifically, the following theorem tests for existence of a quadratic Lyapunov function of the form $V(\mbf u)=\ip{\mbf u}{\mcl P \mbf u}_{L_2} = \ip{\hat{\mbf u}}{\mcl T^*\mcl P \mcl T\hat{\mbf u}}_{L_2} \ge  \alpha \norm{\mbf u}^2_{L_2}$ such that $\dot V(\mbf u(t))\le \delta \norm {\mbf u(t)}^2_{L_2}$ for any solution $\mbf u$ of the PDE defined by $\{A_{ij},\mcl B\}$.

\begin{thm}\label{thm_stability_as_LPI}

 Suppose $\mcl{T}$ and $\mcl{A}$ are as defined in Theorem~\ref{thm_Tmap} and Lemma~\ref{lem_PIE} respectively, and that there exist $\epsilon,\delta>0$ and $P\in\mcl{N}_{2D}$ such that the PI operator $\mcl{P}:=\mcl{P}[P]$ satisfies $\mcl{P}=\mcl{P}^*$, $\mcl{P}\geq \epsilon I$, and
 \[
  \mcl{A}^*\mcl{P}\mcl{T}+\mcl{T}^*\mcl{P}\mcl{A}\leq -\delta\mcl{T}^*\mcl{T}.
 \]
 Then, any solution $\mbf{u}(t)\in X(\mcl{B})$ of the PDE defined by $\{A_{ij},\mcl{B}\}$ satisfies
 \[
  \|\mbf{u}(t)\|_{L_2}^2\leq \f{\zeta}{\epsilon}\|\mbf{u}(0)\|_{L_2}^2 e^{-\f{\delta}{\zeta}t},
 \]
 where $\zeta=\|\mcl{P}\|_{\mcl{L}_{L_2}}$.

\end{thm}
\begin{proof}
	
	Let $\mbf{u}\in X$ be an arbitrary solution to the PDE defined by $\{A_{ij},\mcl{B}\}$, so that $\hat{\mbf{u}}=\mcl{D}\mbf{u}\in L_2$ is a solution to the PIE defined by $\{\mcl{T},\mcl{A}\}$. Consider the candidate Lyapunov function $V:L_2\rightarrow\R$ defined as
	\begin{align*}
	V(\hat{\mbf{v}})=\ip{\mcl{T}\hat{\mbf{v}}}{\mcl{P}\mcl{T}\hat{\mbf{v}}}_{L_2}\geq \epsilon \|\mcl{T}\hat{\mbf{v}}\|^2_{L_2}.
	\end{align*}
	Since $\|\mcl{P}\|_{\mcl{L}_{L_2}}=\zeta$, this function is bounded from above as
	\begin{align*}
	V(\hat{\mbf{v}})=
	\ip{\mcl{T}\hat{\mbf{v}}}{\mcl{P}\mcl{T}\hat{\mbf{v}}}_{L_2}\leq \zeta\|\mcl{T}\hat{\mbf{v}}\|_{L_2}^2.
	\end{align*}
	In addition, since $\hat{\mbf{u}}$ is a solution to the PIE, the temporal derivative of $V$ along $\hat{\mbf{u}}$ satisfies
	\begin{align*}
	\dot{V}(\hat{\mbf{u}})
	&=\ip{\mcl{T}\dot{\hat{\mbf{u}}}}{\mcl{P}\mcl{T}\hat{\mbf{u}}}_{L_2}
	+\ip{\mcl{T}\hat{\mbf{u}}}{\mcl{P}\mcl{T}\dot{\hat{\mbf{u}}}}_{L_2}    \\
	&=\ip{\mcl{A}\hat{\mbf{u}}}{\mcl{P}\mcl{T}\hat{\mbf{u}}}_{L_2}
	+\ip{\mcl{T}\hat{\mbf{u}}}{\mcl{P}\mcl{A}\hat{\mbf{u}}}_{L_2}    \\
	&=\ip{\hat{\mbf{u}}}{\left(\mcl{A}^*\mcl{P}\mcl{T} + \mcl{T}^*\mcl{P}\mcl{A}\right)\hat{\mbf{u}}}_{L_2}    \\
	&\leq -\delta\|\mcl{T}\hat{\mbf{u}}\|^2_{L_2}
	\leq -\f{\delta}{\zeta}V(\hat{\mbf{u}}).
	\end{align*}
	Applying the Gr\"onwell-Bellman inequality, it immediately follows that
	\begin{align*}
	V(\hat{\mbf{u}}(t))\leq V(\hat{\mbf{u}}(0))e^{-\f{\delta}{\zeta} t},
	\end{align*}
	implying
	\begin{align*}
	\|\mcl{T}\hat{\mbf{u}}(t)\|^2_{L_2}\leq \f{\zeta}{\epsilon}\|\mcl{T}\hat{\mbf{u}}(0)\|_{L_2}^2 e^{-\f{\delta}{\zeta} t},
	\end{align*}
	and thus
	\begin{align*}
	\|\mbf{u}(t)\|^2_{L_2}\leq \f{\zeta}{\epsilon}\|\mbf{u}(0)\|_{L_2}^2 e^{-\f{\delta}{\zeta} t}.
	\end{align*}
	
\end{proof}

In this stability condition, the decision variable is $P\in\mcl{N}_{2D}^{n\times n}$ and the constraints are operator inequalities on the inner product space $L_2$. While the decision variables may be readily parameterized using polynomials, to numerically enforce the inequality constraints, we need to parameterize the cone of operators in $\mcl{N}_{2D}^{n\times n}$ which are positive semidefinite. This problem will be addressed in the following subsection.

\subsection{A Parameterization of Positive PI Operators}\label{sec_positive_PI_parameterization}

Having posed the PDE stability problem as an LPI, we now show how to parameterize the cone of positive 2D-PI operators using positive matrices. Specifically, we have the following result.

\begin{prop}\label{prop_pos_PI}
For any $Z\in L_2^{q\times n}[x,y,\theta,\nu]$ and scalar function $g\in L_2[x,y]$ satisfying $g(x,y)\ge 0$ for all $x,y \in [a,b] \times [c,d]$ let $\mcl{L}_{\text{PI}}: \R^{9q \times 9q} \rightarrow \mcl N_{2D}^{n\times n}$ be as defined in Eqn.~\eqref{eq_posmat_to_posPI_appendix} in Appendix~\ref{sec_pos_PI_appendix}. Then for any $P\geq 0$, if $N =\mcl{L}_{\text{PI}} (P)$, we have that $\mcl{P}^*[N]=\mcl{P}[N]$, and $\ip{\mbf{u}}{\mcl{P}[N]\mbf{u}}_{L_2}\geq 0$ for any $\mbf{u}\in L_2^n[x,y]$.
\end{prop}
\paragraph{Outline of Proof}
Given $Z$ and scalar function $g$, a 2D-PI operator $\mcl{Z}:L_2^n[x,y]\rightarrow L_2^n[x,y]$ is defined as in Eqn.~\eqref{eq_Zop_appendix}, where each of the defining parameters is a product of $Z$ and $\sqrt{g}$.
Then, if $N=\mcl{L}_{\text{PI}}(P)$ for some matrix $P$, by definition of the map $\mcl{L}_{\text{PI}}$, the associated PI operator is such that $\mcl{P}[N]=\mcl{Z}^* P\mcl{Z}$. It follows that, for any $P\geq 0$, 
\begin{align*}
\ip{\mbf{u}}{\mcl{P}[N]\mbf{u}}_{L_2}
&=\ip{\mcl{Z}\mbf{u}}{P\mcl{Z}\mbf{u}}_{L_2}   
=\ip{P^{\frac{1}{2}}\mcl{Z}\mbf{u}}{P^{\frac{1}{2}}\mcl{Z}\mbf{u}}_{L_2}\geq 0
\end{align*}
for any $\mbf{u}\in L_2^n[x,y]$, as desired.
For the full proof of this proposition, please see Appendix~\ref{sec_pos_PI_appendix}.

Using Prop.~\ref{prop_pos_PI}, we may enforce an LPI with variable $\mcl{P}[\mcl L_{\text{PI}} (P)]$ using an LMI constraint on $P$. For this test, we use a monomial basis, $Z_d$ for $Z$, thus implying the parameters in $N =\mcl L_{\text{PI}} (P)$ will be polynomial. For our choice of $g(x,y)\geq 0$, we may choose $g(x,y)=1$ (implying the inequality is valid over any domain), or $g(x,y)=(x-a)(b-x)(y-c)(d-y)$, implying the operator is positive only on the domain $(x,y)\in [a,b]\times[c,d]$.

In the following section, we will apply Prop~\ref{prop_pos_PI} to obtain an LMI for stability of a given PDE. For this section we use the notation
\begin{align*}
 \Omega_d:=\{&\mcl{P}[N]+\mcl{P}[M] \mid N,M\in\mcl{N}_{2D}\text{ are of Form~\eqref{eq_posmat_to_posPI_appendix}}\\
 &~\text{with } Z_i=Z_d \text{ and } g(x,y)=1 \text{ and respectively}\\ &\hspace*{1.1cm} g(x,y)=(x-a)(b-x)(y-c)(d-y) \}
\end{align*}
where now $\mcl{P}\in\Omega_d$ is an LMI constraint which implies $\mcl{P}:L_2[x,y] \rightarrow L_2[x,y]$ is a positive operator on $L_2[x,y]$.\\

\section{PIETOOLS Implementation}\label{sec_implementation}

In this section, we show how the PIETOOLS 2021b toolbox may be used to perform stability analysis of PDEs. This toolbox offers a framework for implementation and manipulation of PI operators in MATLAB, allowing e.g. Lyapunov stability analysis~\cite{shivakumar2019generalized}, robust stability analysis~\cite{das2020robust}, and $H_{\infty}$-optimal control~\cite{shivakumar2020duality} of PDEs. For a detailed manual of the PIETOOLS toolbox we refer to~\cite{shivakumar2020pietools}.

To implement PI operators in MATLAB, the \texttt{dpvar} class of polynomial objects is used to define the polynomial functions $N$ parameterizing PI operators $\mcl{P}[N]$. A class of 0112-PI operators is then defined as \texttt{opvar2d} objects, and overloaded with standard operations such as multiplication (*), addition (+) and adjoint (') presented in earlier sections. Defining decision operators \texttt{dopvar2d} in terms of positive matrices, we may also enforce positivity conditions $\mcl{P}\in\Omega_d$, allowing stability to be tested with any LMI solver.

An overview of the steps performed in this process is provided below.

\begin{enumerate}
 \item Define the independent polynomial variables. These are the spatial variables in the PDE. Also define the ``dummy'' variables $\theta$ and $\nu$.
 \begin{verbatim} pvar x y tt nu; \end{verbatim}

 \item Initialize an optimization program structure X.
 \begin{verbatim} X = sosprogram([x y tt nu]); \end{verbatim}

 \item Construct the PDE, defining the sizes $n_0,n_1,n_2$ of the state variables, the matrices $A_{ij}$ defining the PDE, and an \texttt{opvar2d} object \texttt{Ebb} defining $\mcl{B}$. Convert these to a corresponding PIE using \texttt{convert\_PIETOOLS\_PDE}, and extract \texttt{opvar2d} objects $\mcl{T}$ and $\mcl{A}$.
 \begin{verbatim}
 PDE.n.n_pde = [n0,n1,n2];   
 PDE.dom = [a,b;c,d];
 PDE.PDE.A = ...;	PDE.BC.Ebb = ...;
 PIE = convert_PIETOOLS_PDE(PDE); 
 T = PIE.T;        A = PIE.A; \end{verbatim}

 \item Declare a positive operator $\mcl{P}\in\Omega_d$ as a \texttt{dopvar2d} object, using maximal monomial degree $d\in \N$, and spatial domain \texttt{dom}. Add a small constant $\epsilon<<1$ to ensure strict positivity. Impose the additional requirement $-\delta\mcl{T}^*\mcl{T}-\mcl{A}^*\mcl{P}\mcl{T}-\mcl{T}^*\mcl{P}\mcl{A}\in\Omega_d$ for some $\delta>0$.
 \begin{verbatim}
  [X, P] = poslpivar_2d(X,n,dom,d);
  P = P + eps;
  D = -del*(T'*T) - A'*P*T - T'*P*A;
  X = lpi_ineq_2d(X,D); \end{verbatim}

 \item Call the SDP solver.
 \begin{verbatim}
  X = sossolve(X); \end{verbatim}

 \item Get the solution $\mcl{P}$, certifying stability.
 \begin{verbatim}
  Psol = getsol_lpivar_2d(X,P); \end{verbatim}

\end{enumerate}

\section{Illustrative Examples}\label{sec_examples}

To illustrate the techniques described in the previous sections, we will apply them to several simple examples. In each case, we will show how the problem may be expressed in the standardized format, as well as how the corresponding PIE is defined, and we will numerically test stability.

\subsection{Heat Equation}

As a first example, we consider a 2D heat equation
\begin{align*}
u_t(t,x,y)&=u_{xx}(t,x,y)+u_{yy}(t,x,y),\\
u(x,0)&=u_y(x,0)=u(0,y)=u_x(0,y)=0.
\end{align*}
To describe this system in the standardized Format~\eqref{eq_PDE}, we use PDE state $\mbf{u}=\mbf{u}_2=u\in H_2^{n_2}[x,y]$ with $n_2=1$, and define $A_{20}=A_{02}=1$, setting $A_{ij}=0$ for all other $i,j\in\{0,1,2\}$. To enforce the boundary conditions, we require $\mbf{u}(1,0)=\mbf{u}_x(0,0)=\mbf{u}_y(0,1)=\mbf{u}_{xy}(0,0)=0$, as well as $\mbf{u}_{xx}(x,0)=\mbf{u}_{xxy}(x,0)=\mbf{u}_{yy}(0,y)=\mbf{u}_{xyy}(0,y)=0$, which can be expressed as $B\Lambda_{\text{bf}}\mbf{u}=0$ where
\begin{align*}
B&=
\smallbmat{
	0&1&0&0& 0&0&0&0 &0&0&0&0 &0&0&0&0 &0&0&0&0 &0&0&0&0\\
	0&0&0&0& 1&0&0&0 &0&0&0&0 &0&0&0&0 &0&0&0&0 &0&0&0&0\\
	0&0&0&0& 0&0&0&0 &0&0&1&0 &0&0&0&0 &0&0&0&0 &0&0&0&0\\
	0&0&0&0& 0&0&0&0 &0&0&0&0 &1&0&0&0 &0&0&0&0 &0&0&0&0\\
	0&0&0&0& 0&0&0&0 &0&0&0&0 &0&0&0&0 &1&0&0&0 &0&0&0&0\\
	0&0&0&0& 0&0&0&0 &0&0&0&0 &0&0&0&0 &0&0&1&0 &0&0&0&0\\
	0&0&0&0& 0&0&0&0 &0&0&0&0 &0&0&0&0 &0&0&0&0 &1&0&0&0\\
	0&0&0&0& 0&0&0&0 &0&0&0&0 &0&0&0&0 &0&0&0&0 &0&0&1&0}\in
\R^{8\times 24}
\end{align*}
Here we note that, any matrix $B\in\R^{m_0+2m_1\times n_0+2n_1}$ may be represented as a diagonal PI operator $\mcl{B}=\mcl{P}\smallbmat{B_{00}&0&0\\0&B_{11}&0\\0&0&B_{22}}$ through appropriate choice of the parameters $B_{00}\in\R^{m_0\times n_0}$, $B_{11}=\{B_{11}^{0},0,0\}\in\mcl{N}_{1D}^{m_1\times n_1}$ and $B_{22}=\{B_{22}^{0},0,0\}\in\mcl{N}_{1D}^{m_1\times n_1}$, so that the boundary conditions $B\Lambda_{\text{bf}}\mbf{u}=0$ may also be written in the standardized format $\mcl{B}\Lambda_{\text{bf}}\mbf{u}=0$.
Then, we may describe the system as a PDE defined by $\{A_{ij},B\}$ (or really $\{A_{ij},\mcl{B}\}$), for which we obtain a corresponding PIE representation
\begin{align*}
 \mcl{T}\dot{\hat{\mbf{u}}}&=\int_{0}^{x}\int_{0}^{y}(x-\theta)(y-\nu)\dot{\hat{\mbf{u}}}(\theta,\nu)d\nu d\theta   \\
 &=\int_{0}^{x}(x-\theta)\hat{\mbf{u}}(\theta,y)d\theta
 +\int_{0}^{y}(y-\nu)\hat{\mbf{u}}(x,\nu)d\nu=\mcl{A}\hat{\mbf{u}},
\end{align*}
where $\hat{\mbf{u}}=u_{xxyy}$ is the fundamental state.\\

For the purpose of testing accuracy of the stability analysis, we consider the following system, as presented in~\cite{holmes1994partial},
\begin{align*}
u_t(t,x,y)&=u_{xx}(t,x,y) + u_{yy}(t,x,y) + r u(t,x,y),\\
u(x,0)&=u(0,y)=u(x,1)=u(1,y)=0,
\end{align*}
where $h,r> 0$ are positive constants. This system may be represented in the standardized format by letting $A_{20}=A_{02}=h$, as well as $A_{00}=r$, with all other matrices $A_{ij}=0$. For the boundary conditions, we require $\mbf{u}(0,0)=\mbf{u}_y(1,0)=\mbf{u}_x(1,1)=\mbf{u}_{y}(0,1)=0$, as well as $\mbf{u}_{xx}(x,0)=\mbf{u}_{xx}(x,1)=\mbf{u}_{yy}(0,y)=\mbf{u}_{yy}(1,y)=0$, 
which may be enforced as $B\Lambda_{\text{bc}}\mbf{u}=0$, where
\begin{align*}
B&=
\smallbmat{1&0&0&0&0&0&0&0&0&0&0&0&0&0&0&0 &0&0&0&0 &0&0&0&0\\
		   0&0&0&0&0&0&0&1&0&0&0&0&0&0&0&0 &0&0&0&0 &0&0&0&0\\
		   0&0&0&0&0&0&0&0&0&1&0&0&0&0&0&0 &0&0&0&0 &0&0&0&0\\
		   0&0&0&0&0&0&0&0&0&0&1&0&0&0&0&0 &0&0&0&0 &0&0&0&0\\
	   	   0&0&0&0&0&0&0&0&0&0&0&0&0&0&0&0 &1&0&0&0 &0&0&0&0\\
   	   	   0&0&0&0&0&0&0&0&0&0&0&0&0&0&0&0 &0&1&0&0 &0&0&0&0\\
      	   0&0&0&0&0&0&0&0&0&0&0&0&0&0&0&0 &0&0&0&0 &1&0&0&0\\
      	   0&0&0&0&0&0&0&0&0&0&0&0&0&0&0&0 &0&0&0&0 &0&1&0&0}\in
         \R^{8\times 24}
\end{align*}
Implementing this system in PIETOOLS 2021b, we obtain a PIE representation
\begin{align*}
 &\int_{0}^{x}\! g(x,\theta)\bbbl[\int_{0}^{y}\! g(y,\nu)\dot{\hat{\mbf{u}}}(\theta,\nu)d\nu\! 
 +\! \int_{y}^{1}\! g(\nu,y)\dot{\hat{\mbf{u}}}(\theta,\nu)d\nu
 \bbbl]d\theta\\
 &+\! \int_{x}^{1}\! g(\theta,x)\bbbl[\int_{0}^{y}\! g(y,\nu)\dot{\hat{\mbf{u}}}(\theta,\nu)d\nu\! 
 +\! \int_{y}^{1}\! g(\nu,y)\dot{\hat{\mbf{u}}}(\theta,\nu)d\nu
 \bbbl]d\theta\\
 &=\! r\int_{0}^{x}\! g(x,\theta)\bbbl[\int_{0}^{y}\! g(y,\nu)\hat{\mbf{u}}(\theta,\nu)d\nu\! 
 +\! \int_{y}^{1}\! g(\nu,y)\hat{\mbf{u}}(\theta,\nu)d\nu
 \bbbl]d\theta\\
 &+\! r\int_{x}^{1}\! g(\theta,x)\bbbl[\int_{0}^{y}\! g(y,\nu)\hat{\mbf{u}}(\theta,\nu)d\nu\! 
 +\! \int_{y}^{1}\! g(\nu,y)\hat{\mbf{u}}(\theta,\nu)d\nu
 \bbbl]d\theta\\
 &\quad+\! h\int_{0}^{x}\! g(x,\theta)\hat{\mbf{u}}(\theta,y)d\theta\! 
 +\! \int_{x}^{1}\! g(\theta,x)\hat{\mbf{u}}(\theta,y)d\theta \\
 &\qquad+\! h\int_{0}^{y}\! g(y,\nu)\hat{\mbf{u}}(x,\nu)d\nu\! 
 +\! \int_{y}^{1}\! g(\nu,y)\hat{\mbf{u}}(x,\nu)d\nu,
\end{align*}
where we define
\begin{align*}
 g(s,\eta)=(s-1)\eta
\end{align*}
Letting $\ell=1$ and $h=1$, stability of this system can be proven analytically whenever $r \leq 2\pi^2=19.739...$. Using PIETOOLS 2021b, performing a bisection search over $r$ and using monomial degree $d=3$, exponential stability can be verified for any $r\leq 19.736$.

\subsection{Wave Equation}

As a second example, we consider a 2D wave equation
\begin{align*}
\ddot{u}(t,x,y) &= u_{xx}(t,x,y)+u_{yy}(t,x,y),\\
u(x,0)&=u_y(x,0)=u_x(0,y)=u(0,y)=0.
\end{align*}
To write this system in the standardized form, we define $\mbf{u}_1=u$ and $\mbf{u}_2=\dot{u}$, so that the PDE may be denoted as
\begin{align*}
\bmat{\dot{\mbf{u}}_1\\\dot{\mbf{u}}_2}
&= \bmat{0&1\\0&0}\bmat{\mbf{u}_1\\\mbf{u}_2}
+\bmat{0&0\\1&0}\partial_{x}^2\bmat{\mbf{u}_1\\\mbf{u}_2}
+\bmat{0&0\\1&0}\partial_{y}^2\bmat{\mbf{u}_1\\\mbf{u}_2} \\
&= A_{00}\mbf{u} + A_{20}\partial_{x}^2\mbf{u} + A_{02}\partial_{y}^2\mbf{u}
\end{align*}
Requiring $\mbf{u}(1,0)=\mbf{u}_x(0,0)=\mbf{u}_y(0,1)=\mbf{u}_{xy}(0,0)=0$, and $\mbf{u}_{xx}(x,0)=\mbf{u}_{xxy}(x,0)=\mbf{u}_{yy}(0,y)=\mbf{u}_{xyy}(0,y)=0$, we may write the boundary conditions as $B\Lambda_{\text{bf}}\mbf{u} =0$, where
\begin{align*}
B &=\smallbmat{0&I&0&0&0&0&0&0&0&0&0&0&0&0&0&0&0&0&0&0&0&0&0&0\\
	0&0&0&0&I&0&0&0&0&0&0&0&0&0&0&0&0&0&0&0&0&0&0&0\\
	0&0&0&0&0&0&0&0&0&0&I&0&0&0&0&0&0&0&0&0&0&0&0&0\\
	0&0&0&0&0&0&0&0&0&0&0&0&I&0&0&0&0&0&0&0&0&0&0&0\\
	0&0&0&0&0&0&0&0&0&0&0&0&0&0&0&0&I&0&0&0&0&0&0&0\\
	0&0&0&0&0&0&0&0&0&0&0&0&0&0&0&0&0&0&I&0&0&0&0&0\\
	0&0&0&0&0&0&0&0&0&0&0&0&0&0&0&0&0&0&0&0&0&I&0&0\\
	0&0&0&0&0&0&0&0&0&0&0&0&0&0&0&0&0&0&0&0&0&0&0&I}\in\mb{R}^{16\times48}.
\end{align*}
Letting $\hat{\mbf{u}}=\mbf{u}_{xxyy}$, the associated PIE representation is
\begin{align*}
&\mcl{T}\dot{\hat{\mbf{u}}}\! =\! \int_{0}^{x}\! \int_{0}^{y}\bmat{(x-\theta)(y-\nu)&\! \! \! 0\\0&\! \! \! (x-\theta)(y-\nu)}\dot{\hat{\mbf{u}}}(\theta,\nu)d\nu d\theta \\
&=\! \int_{0}^{x}\bmat{0&\! \! (x-\theta)\\0&\! \! 0}\hat{\mbf{u}}(\theta,y)d\theta
+\int_{0}^{y}\bmat{0&\! \! (y-\nu)\\0&\! \! 0}\hat{\mbf{u}}(x,\nu)d\nu \\
&\qquad+\int_{0}^{x}\int_{0}^{y}\bmat{0&\! 0\\(x-\theta)(y-\nu)&\! 0}\hat{\mbf{u}}(\theta,\nu)d\nu d\theta=\mcl{A}\hat{\mbf{u}},
\end{align*}
where $\hat{\mbf{u}}=\mbf{u}_{xxyy}$ is the fundamental state.
Simulation suggests the system is neutrally stable ($\gamma=0$ in Defn.~\ref{defn_PDE_stability}) with the alternative boundary conditions $u(x,0)=u(x,1)=u(0,y)=u(1,y)=0$. Setting $\delta=0$ (see Thm.~\ref{thm_stability_as_LPI}), and using a monomial degree $d=3$, this can be verified with PIETOOLS.

\subsection{Coupled ODE-PDE System}

As a final example, we consider a diffusion equation coupled to an ODE, a s appearing in~\cite{krstic2009compensating},
\begin{align*}
	\dot{X}(t) &= (A+BK)X(t) + Bu(t,0,0)	\\
	\dot{u}(t,x,y) &= u_{xx}(t,x,y) + u_{yy}(t,x,y), \\
	u_x(t,0,y) &= u(t,1,y) = 0,	\\
	u_y(t,x,0) &= u(t,x,1) = 0,
\end{align*}
which is stable whenever the matrix $A+BK$ is Hurwitz.
An equivalent PIE representation of this system may be obtained by first deriving a PIE representation of the PDE subsystem. To this end, we use as PDE state $\mbf{u}=\mbf{u}_2=u\in H_2^{n_2}[x,y]$ with $n_2=1$, and let $A_{20}=A_{02}=1$ to describe the PDE in the standardized format~\eqref{eq_PDE}. For the boundary conditions, we enforce $\mbf{u}(0,1)=\mbf{u}_{x}(1,1)=\mbf{u}_{y}(1,0)=\mbf{u}_{xy}(0,0)$, as well as $\mbf{u}_{xx}(x,1)=\mbf{u}_{xxy}(x,0)=\mbf{u}_{yy}(1,y)=\mbf{u}_{xyy}(0,y)=0$, enforcing the conditions as $B\Lambda_{\mbf{bf}}\mbf{u}=0$, where
\begin{align*}
B &=\smallbmat{0&0&1&0&0&0&0&0&0&0&0&0&0&0&0&0&0&0&0&0&0&0&0&0\\
	0&0&0&0&0&0&0&1&0&0&0&0&0&0&0&0&0&0&0&0&0&0&0&0\\
	0&0&0&0&0&0&0&0&0&1&0&0&0&0&0&0&0&0&0&0&0&0&0&0\\
	0&0&0&0&0&0&0&0&0&0&0&0&1&0&0&0&0&0&0&0&0&0&0&0\\
	0&0&0&0&0&0&0&0&0&0&0&0&0&0&0&0&0&1&0&0&0&0&0&0\\
	0&0&0&0&0&0&0&0&0&0&0&0&0&0&0&0&0&0&1&0&0&0&0&0\\
	0&0&0&0&0&0&0&0&0&0&0&0&0&0&0&0&0&0&0&0&0&1&0&0\\
	0&0&0&0&0&0&0&0&0&0&0&0&0&0&0&0&0&0&0&0&0&0&1&0}\in\mb{R}^{8\times24}.
\end{align*}
Using PIETOOLS 2021b
, the associated PIE representation is found to be
\begin{align*}
&\mcl{T}\dot{\hat{\mbf{u}}}=\\
&\int_{0}^{x}(1-x)\bbbl[\int_{0}^{y}(1-y)\dot{\hat{\mbf{u}}}(\theta,\nu)d\nu
+\int_{y}^{1}(1-\nu)\dot{\hat{\mbf{u}}}(\theta,\nu)d\nu
\bbbl]d\theta\\
&+\int_{x}^{1}(1-\theta)\bbbl[\int_{0}^{y}(1-y)\dot{\hat{\mbf{u}}}(\theta,\nu)d\nu
+\int_{y}^{1}(1-\nu)\dot{\hat{\mbf{u}}}(\theta,\nu)d\nu
\bbbl]d\theta\\
&=\int_{0}^{x}(x-1)\hat{\mbf{u}}(\theta,y)d\theta
+\int_{x}^{1}(\theta-1)\hat{\mbf{u}}(\theta,y)d\theta \\
&\qquad+\int_{0}^{y}(y-1)\hat{\mbf{u}}(x,\nu)d\nu
+\int_{y}^{1}(\nu-1)\hat{\mbf{u}}(x,\nu)d\nu=\mcl{A}{\hat{\mbf{u}}},
\end{align*}
where $\hat{\mbf{u}}=\mbf{u}_{xxyy}$. Now, to incorporate the ODE dynamics, we note that the boundary value $\mbf{u}(0,0)$ may be written in terms of the PIE state $\hat{\mbf{u}}$ as
\[
\mbf{u}=\int_{0}^{1}\int_{0}^{1}(1-x)(1-y)\hat{\mbf{u}}(x,y)dx dy,
\]
allowing the ODE dynamics to described using 0112-PI operators as
\begin{align*}
\bmat{I&0\\0&\mcl{T}}\bmat{\dot{X}\\\dot{\hat{\mbf{u}}}}&=\bmat{A+BK&\smallint_{x=0}^{1}[I]\circ\smallint_{y=0}^{1}[\tilde{B}]\\0&\mcl{A}}\bmat{X\\\hat{\mbf{u}}},
\end{align*}
where $\tilde{B}(x,y)=(1-x)(1-y)B$. In this representation, stability analysis can also be performed as discussed in Section~\ref{sec_stability_LPI}, using a Lyapunov function based on the 0112-PI operators describing the system. This was done for a simple scalar case where $A=B=1$, and $K=k$, in which case stability can be proven analytically whenever $k\leq -1$. Using degree $d=2$ and performing bisection on $k$, stability could be verified for $k\leq -1.0410$.

\section{Conclusion}

In this paper, we have shown that any well-posed, linear, second order 2D PDE can be converted to an equivalent PIE, and we have provided the formulae describing this conversion. To derive these formulae, we have introduced different PI operators in 2D, showing that the product, inverse, adjoint, and composition with differential operators of these operators are described by PI operators as well. Exploiting these relations, we derived a mapping $\mbf{u}=\mcl{T}\hat{\mbf{u}}$ between the PDE state $\mbf{u}\in X(\mcl{B})$, constrained by the boundary conditions as described by $\mcl{B}$, and a fundamental state $\hat{\mbf{u}}\in L_2$, free of any such constraints. Accordingly, the solution to the PIE is not constrained by boundary or continuity constraints, which allowed us to derive LPI conditions for stability of the system. Finally, by paramaterizing PI operators as matrices, we showed how stability may be tested using semidefinite programming, as implemented in the MATLAB toolbox PIETOOLS 2021b.\\
Having demonstrated that the PIE representation can be extended to 2D PDEs, an obvious next step would be to extend it to 3D, and possibly ND systems. In addition, we may expand upon our results to include more complicated systems, such as those involving higher order spatial derivatives or alternative boundary conditions. Moreover, different LPIs may also be introduced for $H_{\infty}$-gain analysis and (robust) controller synthesis, as was done for 1D systems. Finally, alternative spatial domains may also be considered, extending the PIE methodology to more complex geometries.

\bibliographystyle{IEEEtran}
\bibliography{bibfile}

\clearpage


\begin{appendices}
\section{Algebras of PI Operators}\label{sec_PI_algebras_appendix}
\subsection{Notation}
In this section, we will prove the composition rules of PI operators in 2D, as outlined in the article. In describing these results, we will (for the sake of compactness) use subscripts and superscripts to denote the free variables of a function $N\in L_2[x,y,\theta,\nu]$, so that:
\begin{align*}
\fvars{N}{xy\theta\nu}{}\! 
=\fvars{N}{xy\theta}{\nu}\! 
=\fvars{N}{xy}{\theta\nu}\! 
=\fvars{N}{x}{y\theta\nu}\! 
=\fvars{N}{}{xy\theta\nu}\! 
=N(x,y,\theta,\nu).
\end{align*}
For the sake of simplicity, we will also let $(x,y)\in[0,1]^2$, though the results presented in this section extend to any domain $(x,y)\in[a,b]\times[c,d]$. We denote the integral of a function $N\in L_2[x,y,\theta,\nu]$ accordingly as:
\begin{align*}
	\int_{\theta,\nu=0}^{1}\bbl(\fvars{N}{xy}{\theta\nu}\bbr)=\int_{0}^{1}\int_{0}^{1}N(x,y,\theta,\nu)d\nu d\theta
\end{align*}
In addition, we will rely on the functions $\bs{\delta}$, $\mbf{I}$, and to a lesser extent $\mbf{H}$, to limit the domains of integration in subsequent results. Here, $\fvars{\bs{\delta}}{x}{\theta}$ denotes a Dirac delta function, such that, for any $N\in L_2[x]$,
\begin{align*}
	\int_{\theta=0}^{1}\bbl(\fvarss{\bs{\delta}}{x}{\theta}\fvars{N}{\theta}{}\bbr) = \int_{\theta=0}^{1}\bbl(\fvarss{\bs{\delta}}{\theta}{x}\fvars{N}{\theta}{}\bbr) = \fvars{N}{x}{},
\end{align*}
for $x\in[a,b]$. Similarly, $\mbf{I}\in L_2[x,\theta]$ denotes an indicator function, 
\begin{align*}
\fvars{\mbf{I}}{x}{\theta}=\begin{cases}
1  &\text{if }\theta<x,  \\
0  &\text{otherwise.}
\end{cases}
\end{align*}
so that
\begin{align*}
\int_{\theta=0}^{1}\bbl(\fvarss{\mbf{I}}{x}{\theta}\fvars{N}{x}{\theta}\bbr) = 
\int_{\theta=0}^{x}\bbl(\fvars{N}{x}{\theta}\bbr)
\end{align*}
and
\begin{align*}
\int_{\theta=0}^{1}\bbl(\fvarss{\mbf{I}}{\theta}{x}\fvars{N}{x}{\theta}\bbr) = \int_{\theta=x}^{1}\bbl(\fvars{N}{x}{\theta}\bbr)
\end{align*}
Finally, we let $\mbf{H}\in L_2[x,\eta,\theta]$ denote a ``rectangular'' function,
\begin{align*}
\mbf{H}^{x}_{\eta\theta}=\begin{cases}
1   &\text{if } \theta < \eta < x, \\
0   &\text{otherwise}.
\end{cases}
\end{align*}
so that
\begin{align*}
\int_{\eta=0}^{1}\bbl(\fvarss{\mbf{H}}{x}{\eta\theta}\fvars{N}{x}{\eta\theta}\bbr) = \int_{\eta=\theta}^{x}\bbl(\fvars{N}{x}{\eta\theta}\bbr).
\end{align*}
Based on these definitions, the following identities follow trivially:
\begin{align}\label{eq_indfun_identities_appendix}
\fvarss{\bs{\delta}}{x}{\eta}\fvars{\bs{\delta}}{\eta}{\theta}
&=\fvarss{\bs{\delta}}{x}{\theta}\fvars{\bs{\delta}}{x}{\eta},	\nonumber\\
\fvarss{\bs{\delta}}{x}{\eta}\fvars{\mbf{I}}{\eta}{\theta}
&=\fvarss{\mbf{I}}{x}{\theta}\fvars{\bs{\delta}}{x}{\eta},    &
\fvarss{\bs{\delta}}{x}{\eta}\fvars{\mbf{I}}{\theta}{\eta}
&=\fvarss{\mbf{I}}{\theta}{x}\fvars{\bs{\delta}}{x}{\eta},   \nonumber\\
\fvarss{\mbf{I}}{x}{\eta}\fvars{\bs{\delta}}{\eta}{\theta}
&=\fvarss{\mbf{I}}{x}{\theta}\fvars{\bs{\delta}}{\theta}{\eta},       &
\fvarss{\mbf{I}}{\eta}{x}\fvars{\bs{\delta}}{\eta}{\theta}
&=\fvarss{\mbf{I}}{\theta}{x}\fvars{\bs{\delta}}{\theta}{\eta},	\nonumber\\
\fvarss{\mbf{I}}{x}{\eta}\fvars{\mbf{I}}{\eta}{\theta}
&=\fvarss{\mbf{I}}{x}{\theta}\fvars{\mbf{H}}{x}{\eta\theta},      &
\fvarss{\mbf{I}}{\eta}{x}\fvarss{\mbf{I}}{\eta}{\theta}
&=\fvarss{\mbf{I}}{x}{\theta}\fvars{\mbf{I}}{\eta}{x}
+ \fvarss{\mbf{I}}{\theta}{x}\fvars{\mbf{I}}{\eta}{\theta},     \nonumber\\
\fvarss{\mbf{I}}{x}{\eta}\fvars{\mbf{I}}{\theta}{\eta}
&=\fvarss{\mbf{I}}{x}{\theta}\fvars{\mbf{I}}{\theta}{\eta}	
+ \fvarss{\mbf{I}}{\theta}{x}\fvars{\mbf{I}}{x}{\eta},       &
\fvarss{\mbf{I}}{\eta}{x}\fvars{\mbf{I}}{\theta}{\eta}
&=\fvarss{\mbf{I}}{\theta}{x}\fvars{\mbf{H}}{\theta}{\eta x}.
\end{align}

\subsection{Preliminaries}

Using the notation introduced in the previous subsection, we may compactly define different parameterizations of PI operators, as well as their compositions. For example, recall that for $N=\{N_0,N_1,N_2\}\in\mcl{N}_{1D}^{n\times m}$, we define an associated PI operator as
\begin{align*}
\fvars{\bl(\mcl{P}[N]\mbf{u}\br)}{x}{}
&=\fvarss{N_0}{x}{}\fvars{\mbf{u}}{x}{}
+\int_{\theta=0}^{x}\bbl(\fvarss{N_1}{x}{\theta}\fvars{\mbf{u}}{\theta}{}\bbr)	\\
&\quad+\int_{\theta=x}^{1}\bbl(\fvarss{N_2}{x}{\theta}\fvars{\mbf{u}}{\theta}{}\bbr),
\end{align*}
for arbitrary $\mbf{u}\in L_2^{m}[x]$. Using the Dirac delta function $\bs{\delta}$, and the indicator function $\mbf{I}$ as introduced earlier, we may equivalently denote this operation as
\begin{align*}
\fvars{\bl(\mcl{P}[N]\mbf{u}\br)}{x}{}
&=\int_{\theta=0}^{1}\bbbl(\bbl[\fvarss{\bs{\delta}}{x}{\theta}\fvars{N_0}{x}{}
+\fvarss{\mbf{I}}{x}{\theta}\fvars{N_1}{x}{\theta}	
+\fvarss{\mbf{I}}{\theta}{x}\fvars{N_2}{x}{\theta}\bbl]\fvars{\mbf{u}}{\theta}{}\bbbr).
\end{align*}
Then, defining functions $\bs{\Phi}_i$ for $i=\{0,1,2\}$ as
\begin{align}\label{eq_Phi_appendix}
	\fvars{\bs{\Phi}_0}{x}{\theta}&=\fvars{\bs{\delta}}{x}{\theta},	&
	\fvars{\bs{\Phi}_1}{x}{\theta}&=\fvars{\mbf{I}}{x}{\theta},		&
	\fvars{\bs{\Phi}_2}{x}{\theta}&=\fvars{\mbf{I}}{\theta}{x},
\end{align}
we can describe the operation using a single sum as
\begin{align*}
\fvars{\bl(\mcl{P}[N]\mbf{u}\br)}{x}{}
&=\int_{\theta=0}^{1}\bbbl(\sum_{i=0}^{2}\bbl[\fvarss{\bs{\Phi}_{i}}{x}{\theta}\fvars{N_i}{x}{\theta}\bbl]\fvars{\mbf{u}}{\theta}{}\bbbr).
\end{align*}
This notation will allow us to compactly write the composition rules for different PI operators, using the following corollary.
\begin{cor}\label{cor_integral_Phi12_appendix}
	Let $\mbf{\Phi}_{i}$ for $i\in\{0,1,2\}$ be as defined in Eqn.~\ref{eq_Phi_appendix}. Then, for any $W\in L_2[x,\theta]$,
	\begin{align*}
	 \int_{\theta=0}^{1}\bbl(\fvars{W}{x}{\theta}\bbr)
	 =\int_{\theta=0}^{1}\bbbbl(\sum_{i=1}^{2}\fvarss{\bs{\Phi}_i}{x}{\theta}\fvars{W}{x}{\theta}\bbbbr)\\
	\end{align*}
\end{cor}

In addition, we will rely heavily on the following proposition.
\begin{prop}\label{prop_Psi_appendix}
	Let $\mbf{\Phi}_{i}$ for $i\in\{0,1,2\}$ be as defined in Eqn.~\ref{eq_Phi_appendix}. Then, for any $W_{ij}\in L_2[x,\theta]$ with $i,j\in\{0,1,2\}$,
	\begin{align*}
	\sum_{i,j=0}^{2}
	\fvarss{\bs{\Phi}_i}{x}{\eta}
	\fvarss{\bs{\Phi}_j}{\eta}{\theta}
	\fvars{W_{ij}}{x}{\eta\theta}
	=\sum_{k=0}^{2}
	\fvars{\bs{\Phi}_k}{x}{\theta}\bbbbl(
	\sum_{i,j=0}^{2}
	\fvarss{\bs{\Psi}_{kij}}{\eta}{\theta}
	\fvars{W_{ij}}{x}{\eta\theta}\bbbbr),
	\end{align*}
	where
	\begin{align}\label{eq_Psi_appendix}
	\fvars{\bs{\Psi}_{000}}{x}{\eta\theta}
	&=\fvars{\bs{\delta}}{x}{\eta}	\nonumber\\
	\fvars{\bs{\Psi}_{101}}{x}{\eta\theta}
	&=\fvars{\bs{\delta}}{x}{\eta} &
	\fvars{\bs{\Psi}_{202}}{x}{\eta\theta}
	&=\fvars{\bs{\delta}}{x}{\eta}    \nonumber\\	
	\fvars{\bs{\Psi}_{110}}{x}{\eta\theta}
	&=\fvars{\bs{\delta}}{\theta}{\eta}    &
	\fvars{\bs{\Psi}_{220}}{x}{\eta\theta}
	&=\fvars{\bs{\delta}}{\theta}{\eta}    \nonumber\\	
	\fvars{\bs{\Psi}_{111}}{x}{\eta\theta}
	&=\fvars{\mbf{H}}{x}{\eta\theta}   &
	\fvars{\bs{\Psi}_{222}}{x}{\eta\theta}
	&=\fvars{\mbf{H}}{\theta}{\eta x}   \nonumber\\	
	\fvars{\bs{\Psi}_{112}}{x}{\eta\theta}
	&=\fvars{\mbf{I}}{\theta}{\eta}    &
	\fvars{\bs{\Psi}_{212}}{x}{\eta\theta}
	&=\fvars{\mbf{I}}{x}{\eta}    \nonumber\\	
	\fvars{\bs{\Psi}_{121}}{x}{\eta\theta}
	&=\fvars{\mbf{I}}{\eta}{x} &
	\fvars{\bs{\Psi}_{221}}{x}{\eta\theta}
	&=\fvars{\mbf{I}}{\eta}{\theta},	
	\end{align}
	and $\fvars{\bs{\Psi}_{kij}}{x}{\eta\theta}=0$ for any other indices $k,i,j\in\{0,1,2\}$.
\end{prop}

\begin{proof}
	Expanding the sums, invoking the definition of $\bs{\Phi}$, and applying identities~\eqref{eq_indfun_identities_appendix}, we find
	\begin{align*}
	&\sum_{i,j=0}^{2}
	\fvarss{\bs{\Phi}_i}{x}{\eta}
	\fvarss{\bs{\Phi}_j}{\eta}{\theta}
	\fvars{W_{ij}}{x}{\eta\theta}  \\
	&=\fvarss{\bs{\delta}}{x}{\eta}\fvarss{\bs{\delta}}{\eta}{\theta}
	\fvars{W_{00}}{x}{\eta\theta}
	+\fvarss{\bs{\delta}}{x}{\eta}\fvarss{\mbf{I}}{\eta}{\theta}
	\fvars{W_{01}}{x}{\eta\theta}
	+\fvarss{\bs{\delta}}{x}{\eta}\fvarss{\mbf{I}}{\theta}{\eta}
	\fvars{W_{02}}{x}{\eta\theta}  \\
	&\quad+\fvarss{\mbf{I}}{x}{\eta}\fvarss{\bs{\delta}}{\eta}{\theta}
	\fvars{W_{10}}{x}{\eta\theta}
	+\fvarss{\mbf{I}}{x}{\eta}\fvarss{\mbf{I}}{\eta}{\theta}
	\fvars{W_{11}}{x}{\eta\theta}
	+\fvarss{\mbf{I}}{x}{\eta}\fvarss{\mbf{I}}{\theta}{\eta}
	\fvars{W_{12}}{x}{\eta\theta}  \\
	&\qquad+\fvarss{\mbf{I}}{\eta}{x}\fvarss{\bs{\delta}}{\eta}{\theta}
	\fvars{W_{20}}{x}{\eta\theta}
	+\fvarss{\mbf{I}}{\eta}{x}\fvarss{\mbf{I}}{\eta}{\theta}
	\fvars{W_{21}}{x}{\eta\theta}
	+\fvarss{\mbf{I}}{\eta}{x}\fvarss{\mbf{I}}{\theta}{\eta}
	\fvars{W_{22}}{x}{\eta\theta}      \\
	&=\fvarss{\bs{\delta}}{x}{\theta}\fvarss{\bs{\delta}}{x}{\eta}
	\fvars{W_{00}}{x}{\eta\theta}
	+\fvarss{\mbf{I}}{x}{\theta}\fvarss{\bs{\delta}}{x}{\eta}
	\fvars{W_{01}}{x}{\eta\theta}
	+\fvarss{\mbf{I}}{\theta}{x}\fvarss{\bs{\delta}}{x}{\eta}
	\fvars{W_{02}}{x}{\eta\theta}      \\
	&\quad+\fvarss{\mbf{I}}{x}{\theta}\fvarss{\bs{\delta}}{\theta}{\eta}
	\fvars{W_{10}}{x}{\eta\theta}
	+\fvarss{\mbf{I}}{x}{\theta}\fvarss{\mbf{H}}{x}{\eta\theta}
	\fvars{W_{11}}{x}{\eta\theta}		\\
	&\qquad+\bbl(\fvarss{\mbf{I}}{x}{\theta}\fvars{\mbf{I}}{\theta}{\eta} + \fvarss{\mbf{I}}{\theta}{x}\fvars{\mbf{I}}{x}{\eta}\bbr)
	\fvars{W_{12}}{x}{\eta\theta}  
	+\fvarss{\mbf{I}}{\theta}{x}\fvarss{\bs{\delta}}{\theta}{\eta}
	\fvars{W_{20}}{x}{\eta\theta}		\\
	&\qquad\quad +\bbl(\fvarss{\mbf{I}}{x}{\theta}\fvars{\mbf{I}}{\eta}{x} + \fvarss{\mbf{I}}{\theta}{x}\fvars{\mbf{I}}{\eta}{\theta}\bbr)
	\fvars{W_{21}}{x}{\eta\theta}
	+\fvarss{\mbf{I}}{\theta}{x}\fvarss{\mbf{H}}{\theta}{\eta x}
	\fvars{W_{22}}{x}{\eta\theta}  \\
	&=\fvarss{\bs{\delta}}{x}{\theta}\fvarss{\bs{\Psi}_{000}}{x}{\eta\theta}
	\fvars{W_{00}}{x}{\eta\theta}  \\
	&\quad+\fvars{\mbf{I}}{x}{\theta}\bbbl(
	\fvarss{\bs{\Psi}_{101}}{x}{\eta\theta}
	\fvars{W_{01}}{x}{\eta\theta}
	+\fvarss{\bs{\Psi}_{110}}{x}{\eta\theta}
	\fvars{W_{10}}{x}{\eta\theta}	\\
	&\hspace{1.5cm}	+\fvarss{\bs{\Psi}_{112}}{x}{\eta\theta}
	\fvars{W_{12}}{x}{\eta\theta}
	+\fvarss{\bs{\Psi}_{111}}{x}{\eta\theta}
	\fvars{W_{11}}{x}{\eta\theta}	\\
	&\hspace{4.3cm} +\fvarss{\bs{\Psi}_{121}}{x}{\eta\theta}
	\fvars{W_{21}}{x}{\eta\theta}
	\bbbr) \\
	&\quad+\fvars{\mbf{I}}{\theta}{x}\bbbl(
	\fvarss{\bs{\Psi}_{202}}{x}{\eta\theta}
	\fvars{W_{02}}{x}{\eta\theta}
	+\fvarss{\bs{\Psi}_{220}}{x}{\eta\theta}
	\fvars{W_{20}}{x}{\eta\theta}	\\
	&\hspace{1.5cm} +\fvarss{\bs{\Psi}_{212}}{x}{\eta\theta}
	\fvars{W_{12}}{x}{\eta\theta}
	+\fvarss{\bs{\Psi}_{222}}{x}{\eta\theta}
	\fvars{W_{22}}{x}{\eta\theta}	\\
	&\hspace{4.3cm} +\fvarss{\bs{\Psi}_{221}}{x}{\eta\theta}
	\fvars{W_{21}}{x}{\eta\theta}
	\bbbr)     \\
	&=\fvars{\bs{\Phi}_0}{x}{\theta}
	\sum_{i=0}^2\sum_{j=0}^2
	\fvarss{\bs{\Psi}_{0ij}}{x}{\eta\theta}
	\fvars{W_{ij}}{x}{\eta\theta}	\\
	&\qquad +\fvars{\bs{\Phi}_1}{x}{\theta}
	\sum_{i=0}^2\sum_{j=0}^2
	\fvarss{\bs{\Psi}_{1ij}}{x}{\eta\theta}
	\fvars{W_{ij}}{x}{\eta\theta}	\\
	&\qquad\qquad +\fvars{\bs{\Phi}_2}{x}{\theta}
	\sum_{i=0}^2\sum_{j=0}^2
	\fvarss{\bs{\Psi}_{2ij}}{x}{\eta\theta}
	\fvars{W_{ij}}{x}{\eta\theta}     \\
	&=\sum_{k=0}^{2}
	\fvars{\bs{\Phi}_k}{x}{\theta}
	\sum_{i,p=0}^{2}
	\fvars{\bs{\Psi}_{kij}}{x}{\eta\theta}
	\fvarss{W_{ij}}{x}{\eta\theta},
	\end{align*}
	as desired.
\end{proof}
Applying this result, the composition rules of 1D-PI operators follow immediately.

\begin{lem}\label{lem_1Dalgebra_appendix}
	For any $N:=\{N_0,N_1,N_2\}\in\mcl{N}_{1D}$ and $M:=\{M_0,M_1,M_2\}\in\mcl{N}_{1D}$, there exists a unique
	$Q \in \mcl{N}_{1D}$ such that $\mcl{P}[N]\circ\mcl{P}[M]=\mcl{P}[Q]$. Specifically, we may choose $Q=\mcl{L}_{1D}(N,M)\in\mcl{N}_{1D}$, where the linear parameter map $\mcl{L}_{1D}:\mcl{N}_{1D}\times\mcl{N}_{1D}\rightarrow\mcl{N}_{1D}$ is such that
	\begin{align*}
	 \mcl{L}_{1D}(N,M)=\{Q_0,Q_1,Q_2\}\in\mcl{N}_{1D},
	\end{align*}
	where, defining functions $\bs{\Psi}_{kij}$ as in Eqn.~\eqref{eq_Psi_appendix},
	\begin{align*}
	\fvars{Q_k}{x}{\theta}=
	\int_{\eta=0}^{1}\bbbbl(\sum_{i,j=0}^{2}
	\fvarss{\bs{\Psi}_{kij}}{x}{\eta\theta}
	\fvarss{N_{i}}{x}{\eta}
	\fvars{M_{j}}{\eta}{\theta}\bbbbr),
	\end{align*}
	for each $k\in\{0,1,2\}$.
\end{lem}
\begin{proof}
	Let $\mbf{u}\in L_2[x]$ be arbitrary. Then, applying the results from Proposition~\ref{prop_Psi_appendix}, we find
	\begin{align*}
	&\fvars{\bl(\mcl{P}[N]\mcl{P}[M]\mbf{u}\br)}{x}{}	\\
	&=\int_{\eta=0}^{1}\bbbbl[\sum_{i=0}^{2}
	\fvarss{\bs{\Phi}_{i}}{x}{\eta}
	\fvars{N_{i}}{x}{\eta}
	\int_{\theta=0}^{1}\bbbl(\sum_{j=0}^{2}
	\fvarss{\bs{\Phi}_{j}}{\eta}{\theta}
	\fvarss{M_{j}}{\eta}{\theta}
	\fvars{\mbf{u}}{\theta}{}\bbbr)
	\bbbbr]      \\
	&=\int_{\theta=0}^{1}\bbbbl[
	\int_{\eta=0}^{1}\bbbbl(\sum_{i,j=0}^{2}
	\fvarss{\bs{\Phi}_{i}}{x}{\eta}
	\fvarss{\bs{\Phi}_{j}}{\eta}{\theta}
	\fvarss{N_{i}}{x}{\eta}
	\fvars{M_{j}}{\eta}{\theta}
	\bbbbr)\fvars{\mbf{u}}{\theta}{}
	\bbbbr]    \\
	&=\int_{\theta=0}^{1}\bbbbl[
	\sum_{k=0}^{2}
	\fvars{\bs{\Phi}_{k}}{x}{\theta}
	\int_{\eta=0}^{1}\bbbbl(\sum_{i,j=0}^{2}
	\fvarss{\bs{\Psi}_{kij}}{x}{\eta\theta}
	\fvarss{N_{i}}{x}{\eta}
	\fvars{M_{j}}{\eta}{\theta}
	\bbbbr)\fvars{\mbf{u}}{\theta}{}
	\bbbbr]  \\
	&=\int_{\theta=0}^{1}\bbbbl(
	\sum_{k=0}^{2}
	\fvarss{\bs{\Phi}_{k}}{x}{\theta}
	\fvarss{Q_{k}}{x}{\theta}
	\fvars{\mbf{u}}{\theta}{}
	\bbbbr)
	=\fvars{(\mcl{P}[Q]\mbf{u})}{x}{}
	\end{align*}
\end{proof}
Expanding the terms in this expression for the composition, it is easy to verify that these composition rules for 1D-PI operators match those (for 3-PI operators) presented in~\cite{peet2021partial}. Using this same approach, we can also derive composition rules for PI operators on additional dimensions.

\subsection{An Algebra of 011-PI Operators}\label{sec_011algebra_appendix}

Recall that we defined a parameter space for 011-PI operators as
\begin{align}\label{eq_N011_space_appendix}
\mcl{N}_{011}{\smallbmat{n_0&m_0\\n_1&m_1}}&:=\bmat{\R^{n_0\times m_0}   &\! L_2^{n_0\times m_1}[\theta] &\! L_2^{n_0\times m_1}[\nu]\\
	L_2^{n_1\times m_0}[x] &\! \mcl{N}_{1D}^{n_1\times m_1}   &\! L_2^{n_1\times m_1}[x,\nu]\\
	L_2^{n_1\times m_0}[y] &\! L_2^{n_1\times m_1}[y,\theta]   &\! \mcl{N}_{1D}^{n_1\times m_1}}.
\end{align}
Then, for any 
\begin{align*}
B&=\bmat{B_{00}&B_{01}&B_{02}\\B_{10}&B_{11}&B_{12}\\B_{20}&B_{21}&B_{22}}\in \mcl{N}_{011}{\smallbmat{n_0&m_0\\n_1&m_1}},\\
\end{align*}
we may write the associated PI operation using the functions $\bs{\Phi}_i$ as defined in Eqn.~\eqref{eq_Phi_appendix} as 
\begin{align*}
 &\fvars{\bl(\mcl{P}[B]\mbf{u}\br)}{xy}{}	\\
 &=\bmat{B_{00}u_0 
 &\int_{\theta=0}^{1}\bbl(\fvarss{B_{01}}{}{\theta}\fvars{\mbf{u}_1}{\theta}{}\bbr)
 &\int_{\nu=0}^{1}\bbl(\fvarss{B_{02}}{}{\nu}\fvars{\mbf{u}_2}{\nu}{}\bbr)
 \\\fvarss{B_{10}}{x}{}u_0
 &\fvars{\bl(\mcl{P}[B_{11}]\mbf{u}_1\br)}{x}{}
 &\int_{\nu=0}^{1}\bbl(\fvarss{B_{12}}{x}{\nu}\fvars{\mbf{u}_2}{\nu}{}\bbr)
 \\\fvarss{B_{20}}{y}{}u_0
 &\int_{\theta=0}^{1}\bbl(\fvarss{B_{21}}{y}{\theta}\fvars{\mbf{u}_1}{\theta}{}\bbr)
 &\fvars{\bl(\mcl{P}[B_{22}]\mbf{u}_2\br)}{y}{}},
\end{align*}
for any $\mbf{u}=\bmat{u_0\\\mbf{u}_1\\\mbf{u}_2}\in\bmat{\R^{m_0}\\L_2^{m_1}[x]\\L_2^{m_1}[y]}$.

\begin{lem}\label{lem_011algebra_appendix}
	For any $B=\smallbmat{B_{00}&B_{01}&B_{02}\\B_{10}&B_{11}&B_{12}\\B_{20}&B_{21}&B_{22}}\in\mcl{N}_{011}\smallbmat{n_0&p_0\\n_1&p_1}$ and $D=\smallbmat{D_{00}&D_{01}&D_{02}\\D_{10}&D_{11}&D_{12}\\D_{20}&D_{21}&D_{22}}\in\mcl{N}_{011}\smallbmat{p_0&m_0\\p_1&m_1}$, where $B_{11}=\{B_{11}^{0},B_{11}^{1},B_{11}^{2}\}\in\mcl{N}_{1D}$, $B_{22}=\{B_{22}^{0},B_{22}^{1},B_{22}^{2}\}\in\mcl{N}_{1D}$ and $D_{11}=\{D_{11}^{0},D_{11}^{1},D_{11}^{2}\}\in\mcl{N}_{1D}$, $D_{22}=\{D_{22}^{0},D_{22}^{1},D_{22}^{2}\}\in\mcl{N}_{1D}$, there exists a unique $R\in\mcl{N}_{011}\smallbmat{n_0&m_0\\n_1&m_1}$ such that $\mcl{P}[B]\circ\mcl{P}[D]=\mcl{P}[R]$. Specifically, we may choose $R=\mcl{L}_{011}(B,D)\in\mcl{N}_{011}\smallbmat{n_0&m_0\\n_1&m_1}$, where the linear parameter map $\mcl{L}_{011}:\mcl{N}_{011}\times\mcl{N}_{011}\rightarrow\mcl{N}_{011}$ is defined such that
	\begin{align}\label{eq_composition_011to011_1_appendix}
		\mcl{L}_{011}(B,D)=\bmat{R_{00}&R_{01}&R_{02}\\R_{10}&R_{11}&R_{12}\\R_{20}&R_{21}&R_{22}}\in\mcl{N}_{011}\smallbmat{n_0&m_0\\n_1&m_1},
	\end{align}
	where
\begin{align*}
	&R_{00}=B_{00}D_{00} + \int_{\eta=0}^{1}\bbl(\fvarss{B_{01}}{}{\eta}\fvars{D_{10}}{\eta}{}\bbr)+\int_{\mu=0}^{1}\bbl(\fvarss{B_{02}}{}{\eta}\fvars{D_{20}}{\eta}{}\bbr) \\ 
	&\fvars{R_{01}}{}{\theta}= B_{00}\fvars{D_{01}}{}{\theta} + \int_{\eta=0}^{1}\bbbbl(\fvars{B_{01}}{}{\eta}\sum_{i=0}^{2}\bbl[\fvarss{\bs{\Phi}_i}{\eta}{\theta}\fvars{D_{11}^{i}}{\eta}{\theta}\bbr]\bbbbr)\\
	&\qquad\quad+ \int_{\mu=0}^{1}\bbbl(\fvarss{B_{02}}{}{\mu}\fvars{D_{21}}{\mu}{\theta}\bbbr)	\\
	&\fvars{R_{02}}{}{\nu}=B_{00}\fvars{D_{02}}{}{\nu} + \int_{\eta=0}^{1}\bbbl(\fvarss{B_{01}}{}{\eta}\fvars{D_{12}}{\eta}{\nu}\bbbr)	\\
	&\qquad\quad+ \int_{\mu=0}^{1}\bbbbl(\fvars{B_{02}}{}{\mu}\sum_{i=0}^{2}\bbl[\fvarss{\bs{\Phi}_i}{\mu}{\nu}\fvars{D_{22}^{i}}{\mu}{\nu}\bbr]\bbbbr)	\\
	&\fvars{R_{10}}{x}{}= \fvarss{B_{10}}{x}{}D_{00} + \int_{\eta=0}^{1}\bbbbl(\sum_{i=0}^{2}\bbl[\fvarss{\bs{\Phi}_i}{x}{\eta}\fvars{B_{11}^{i}}{x}{\eta}\bbr]\fvars{D_{10}}{\eta}{}\bbbbr) \\
	&\qquad\quad+ \int_{\mu=0}^{1}\bbbl(\fvarss{B_{12}}{x}{\mu}\fvars{D_{20}}{\mu}{}\bbbr) \\
	&R_{11}=\mcl{L}_{101}\bl(B_{10},D_{01}\br) + \mcl{L}_{1D}\bl(B_{11},D_{11}\br) + \mcl{L}_{111}\bl(B_{12},D_{21}\br),\\
	&\fvars{R_{12}}{x}{\nu}= \fvarss{B_{10}}{x}{}\fvars{D_{02}}{}{\nu} + \int_{\eta=0}^{1}\bbbbl(\sum_{i=0}^{2}\bbl[\fvarss{\bs{\Phi}_i}{x}{\eta}\fvars{B_{11}^{i}}{x}{\eta}\bbr]\fvars{D_{12}}{\eta}{\nu}\bbbbr) \\
	&\qquad\quad+ \int_{\mu=0}^{1}\bbbbl(\fvars{B_{12}}{x}{\mu}\sum_{i=0}^{2}\bbl[\fvarss{\bs{\Phi}_i}{\mu}{\nu}\fvars{D_{22}^i}{\mu}{\nu}\bbr]\bbbbr) \\
	&\fvars{R_{20}}{y}{}= \fvarss{B_{20}}{y}{}D_{00} + \int_{\eta=0}^{1}\bbbbl(\fvarss{B_{21}}{y}{\eta}\fvars{D_{10}}{\eta}{}\bbbbr) \\
	&\qquad\quad+ \int_{\mu=0}^{1}\bbbbl(\sum_{i=0}^{2}\bbl[\fvarss{\bs{\Phi}_i}{y}{\mu}\fvars{B_{22}^{i}}{y}{\mu}\bbr]\fvars{D_{20}}{\mu}{}\bbbbr) \\
	&\fvars{R_{21}}{y}{\theta}= \fvarss{B_{20}}{y}{}\fvars{D_{01}}{}{\theta} + \int_{\eta=0}^{1}\bbbbl(\fvars{B_{21}}{y}{\eta}\sum_{i=0}^{2}\bbl[\fvarss{\bs{\Phi}_i}{\eta}{\theta}\fvars{D_{11}^i}{\eta}{\theta}\bbr]\bbbbr) \\
	&\qquad\quad+ \int_{\mu=0}^{1}\bbbbl(\sum_{i=0}^{2}\bbl[\fvarss{\bs{\Phi}_i}{y}{\mu}\fvars{B_{22}^{i}}{y}{\mu}\bbr]\fvars{D_{21}}{\mu}{\theta}\bbbbr) \\
	&R_{22}=\mcl{L}_{101}\bl(B_{20},D_{02}\br) + \mcl{L}_{111}\bl(B_{21},D_{12}\br) + \mcl{L}_{1D}\bl(B_{22},D_{22}\br), 
\end{align*}
and where $\mcl{L}_{101}:L_2[x]\times L_2[x]\rightarrow\mcl{N}_{1D}$ is defined such that, for arbitrary $N,M\in L_2[x]$, $\mcl{L}_{101}(N,M)=Q=\{0,Q_1,Q_2\}$, where
\begin{align*}
	\fvars{Q_{1}}{x}{\theta}&=\fvars{Q_{2}}{x}{\theta}
	=\fvarss{N}{x}{}\fvars{M}{}{\theta},
\end{align*}
and $\mcl{L}_{111}:L_2[x,y]\times L_2[y,x]\rightarrow\mcl{N}_{1D}$ is defined such that, for arbitrary $N\in L_2[x,y]$ and $M\in L_2[y,\theta]$, $\mcl{L}_{111}(N,M)=Q=\{0,Q_1,Q_2\}$, where
\begin{align*}
\fvars{Q_{1}}{x}{\theta}&=\fvars{Q_{2}}{x}{\theta}
=\int_{\mu=0}^{1}\bbbl(\fvarss{N}{x}{\mu}\fvars{M}{\mu}{\theta}\bbbr)
\end{align*}
\end{lem}

\begin{proof}
	To prove this lemma, we will exploit the linear structure of 011-PI operators, allowing us to express
	\begin{align*}
	\mcl{P}\smallbmat{B_{00}&B_{01}&B_{02}\\B_{10}&B_{11}&B_{12}\\B_{20}&B_{21}&B_{22}} = \mcl{P}\smallbmat{B_{00}&0&0\\0&0&0\\0&0&0} + \hdots + \mcl{P}\smallbmat{0&0&0&\\0&0&0\\0&0&B_{22}}.
	\end{align*}
	Then, $\mcl{P}[B]\circ\mcl{P}[D]$ maps a vector $u_0\in\R^{m_0}$, and functions $\mbf{u}_1\in L_2^{m_1}[x]$ and $\mbf{u}_2\in L_2^{m_1}[y]$ to $\R^{n_0}$ as
	\begin{align*}
	&\mcl{P}\smallbmat{B_{00}&B_{01}&B_{02}\\()&()&()\\()&()&()} \mcl{P}\smallbmat{D_{00}&()&()\\D_{10}&()&()\\D_{20}&()&()}u_0 = B_{00}D_{00}u_0 \\
	&\quad + \int_{\eta=0}^{1}\bbl(\fvarss{B_{01}}{}{\eta}\fvarss{D_{10}}{\eta}{}u_0\bbr) + \int_{\mu=0}^{1}\bbl(\fvarss{B_{02}}{}{\mu}\fvarss{D_{20}}{\mu}{}u_0\bbr) \\
	&\hspace*{4.75cm}= R_{00} u_0 = \mcl{P}\smallbmat{R_{00}&()&()\\()&()&()\\()&()&()}\! u_0,
	\end{align*}
	and	
	\begin{align*}
	&\mcl{P}\smallbmat{B_{00}&B_{01}&B_{02}\\()&()&()\\()&()&()} \mcl{P}\smallbmat{()&D_{01}&()\\()&D_{11}&()\\()&D_{21}&()}\mbf{u}_1 = B_{00}\int_{\theta=0}^{1}\bbl(\fvarss{D_{01}}{}{\theta}\fvars{\mbf{u}_1}{\theta}{}\bbr) \\
	&\qquad + \int_{\eta=0}^{1}\bbbbl[\fvars{B_{01}}{}{\eta}\int_{\theta=0}^{1}\bbbl(\sum_{i=0}^{2}\fvarss{\bs{\Phi}_i}{\eta}{\theta}\fvarss{D_{11}^{i}}{\eta}{\theta}\fvars{\mbf{u}_1}{\theta}{}\bbbr)\bbbbr] \\
	&\qquad + \int_{\mu=0}^{1}\bbbl[\fvars{B_{02}}{}{\mu}\int_{\theta=0}^{1}\bbl(\fvarss{D_{21}}{\mu}{\theta}\fvars{\mbf{u}_1}{\theta}{}\bbr)\bbbr] \\
	&\hspace*{2.5cm}= \int_{\theta=0}^{1}\bbl(\fvarss{R_{01}}{}{\theta}\fvars{\mbf{u}_1}{\theta}{}\bbr) = \mcl{P}\smallbmat{()&R_{01}&()\\()&()&()\\()&()&()}\mbf{u}_1,
	\end{align*}
	and finally
	\begin{align*}
	&\mcl{P}\smallbmat{B_{00}&B_{01}&B_{02}\\()&()&()\\()&()&()} \mcl{P}\smallbmat{()&()&D_{02}\\()&()&D_{12}\\()&()&D_{22}}\mbf{u}_2 = B_{00}\int_{\nu=0}^{1}\bbl(\fvarss{D_{02}}{}{\nu}\fvars{\mbf{u}_2}{\nu}{}\bbr) \\
	&\qquad + \int_{\eta=0}^{1}\bbbl[\fvars{B_{01}}{}{\eta}\int_{\nu=0}^{1}\bbl(\fvarss{D_{12}}{\eta}{\nu}\fvars{\mbf{u}_2}{\nu}{}\bbr)\bbbr] \\
	&\qquad + \int_{\mu=0}^{1}\bbbbl[\fvars{B_{02}}{}{\mu}\int_{\nu=0}^{1}\bbbl(\sum_{i=0}^{2}\fvarss{\bs{\Phi}_i}{\mu}{\nu}\fvarss{D_{22}^{i}}{\mu}{\nu}\fvars{\mbf{u}_2}{\nu}{}\bbbr)\bbbbr] \\
	&\hspace*{2.5cm}= \int_{\nu=0}^{1}\bbl(\fvarss{R_{02}}{}{\nu}\fvars{\mbf{u}_2}{\nu}{}\bbr) = \mcl{P}\smallbmat{()&()&R_{02}\\()&()&()\\()&()&()}\mbf{u}_2.
	\end{align*}
	Similarly, such states get mapped to functions in $L_2^{n_1}[x]$ as
	\begin{align*}
	&\fvars{\left(\mcl{P}\smallbmat{()&()&()\\B_{10}&B_{11}&B_{12}\\()&()&()} \mcl{P}\smallbmat{D_{00}&()&()\\D_{10}&()&()\\D_{20}&()&()}u_0\right)}{x}{} = \fvarss{B_{10}}{x}{}D_{00}u_0 \\
	& + \int_{\eta=0}^{1}\! \bbbl(\sum_{i=0}^{2}\! \fvarss{\bs{\Phi}_i}{x}{\eta}\fvarss{B_{11}}{x}{\eta}\fvarss{D_{10}}{\eta}{}u_0\bbbr)\! 
	+\! \int_{\mu=0}^{1}\! \bbl(\fvarss{B_{12}}{x}{\mu}\fvarss{D_{20}}{\mu}{}u_0\bbr) \\
	&\hspace*{3.5cm}= \fvarss{R_{10}}{x}{} u_0 = \fvars{\left(\mcl{P}\smallbmat{()&()&()\\R_{10}&()&()\\()&()&()}u_0\right)}{x}{},
	\end{align*}
	and	
	\begin{align*}
	&\fvars{\left(\mcl{P}\smallbmat{()&()&()\\B_{10}&B_{11}&B_{12}\\()&()&()} \mcl{P}\smallbmat{()&D_{01}&()\\()&D_{11}&()\\()&D_{21}&()}\mbf{u}_1\right)}{x}{} \\
	& =  \fvars{B_{10}}{x}{}\int_{\theta=0}^{1}\bbl(\fvarss{D_{01}}{}{\theta}\fvars{\mbf{u}_1}{\theta}{}\bbr) 
	+ \fvars{\bl(\mcl{P}[B_{11}]\mcl{P}[D_{11}]\mbf{u}_1\bl)}{x}{} \\
	&\quad + \int_{\mu=0}^{1}\bbbl[\fvars{B_{12}}{x}{\mu}\int_{\theta=0}^{1}\bbl(\fvarss{D_{21}}{\mu}{\theta}\fvars{\mbf{u}_1}{\theta}{}\bbr)\bbbr] \\
	& = \int_{\theta=0}^{1}\bbl(\fvars{B_{10}}{x}{}\fvarss{D_{01}}{}{\theta}\fvars{\mbf{u}_1}{\theta}{}\bbr) +  \fvars{\bbl(\bl[\mcl{P}[B_{11}]\circ\mcl{P}[D_{11}]\br]\mbf{u}_1\bbl)}{x}{} \\
	&\quad + \int_{\theta=0}^{1}\bbbl[\int_{\mu=0}^{1}\bbl(\fvarss{B_{12}}{x}{\mu}\fvarss{D_{21}}{\mu}{\theta}\bbr)\fvars{\mbf{u}_1}{\theta}{}\bbbr] \\
	& =  \fvars{\bl(\mcl{P}[\mcl{L}_{101}(B_{10},D_{01})]\mbf{u}_1\br)}{x}{} + \fvars{\bl(\mcl{P}[\mcl{L}_{1D}(B_{11},D_{11})]\mbf{u}_1\br)}{x}{} \\
	&\quad + \fvars{\bl(\mcl{P}[\mcl{L}_{111}(B_{12},D_{21})]\mbf{u}_1\br)}{x}{} 
	= \fvars{\bl(\mcl{P}[R_{11}]\mbf{u}_1\br)}{x}{}	\\
	&\hspace*{4.6cm} = \fvars{\left(\mcl{P}\smallbmat{()&()&()\\()&R_{11}&()\\()&()&()}\mbf{u}_1\right)}{x}{},
	\end{align*}
	and finally
	\begin{align*}
	&\fvars{\left(\mcl{P}\smallbmat{()&()&()\\B_{10}&B_{11}&B_{12}\\()&()&()} \mcl{P}\smallbmat{()&()&D_{02}\\()&()&D_{12}\\()&()&D_{22}}\mbf{u}_2\right)}{x}{} \\
	&= \fvars{B_{10}}{x}{}\int_{\nu=0}^{1}\bbl(\fvarss{D_{02}}{}{\nu}\fvars{\mbf{u}_2}{\nu}{}\bbr) \\
	&\quad + \int_{\eta=0}^{1}\bbbl[\sum_{i=0}^{2}\fvarss{\bs{\Phi}_i}{x}{\eta}\fvars{B_{11}}{x}{\eta}\int_{\nu=0}^{1}\bbl(\fvarss{D_{12}}{\eta}{\nu}\fvars{\mbf{u}_2}{\nu}{}\bbr)\bbbr] \\
	&\quad + \int_{\mu=0}^{1}\bbbbl[\fvars{B_{12}}{x}{\mu}\int_{\nu=0}^{1}\bbbl(\sum_{i=0}^{2}\fvarss{\bs{\Phi}_i}{\mu}{\nu}\fvarss{D_{22}^{i}}{\mu}{\nu}\fvars{\mbf{u}_2}{\nu}{}\bbbr)\bbbbr] \\
	&\hspace*{2.0cm}= \int_{\nu=0}^{1}\bbl(\fvarss{R_{12}}{x}{\nu}\fvars{\mbf{u}_2}{\nu}{}\bbr) = \fvars{\left(\mcl{P}\smallbmat{()&()&()\\()&()&R_{12}\\()&()&()}\mbf{u}_2\right)}{x}{}.
	\end{align*}
	Similarly, we attain a mapping to $L_2^{n_1}[y]$ as
	\begin{align*}
	&\fvars{\left(\mcl{P}\smallbmat{()&()&()\\()&()&()\\B_{20}&B_{21}&B_{22}} \mcl{P}\smallbmat{D_{00}&()&()\\D_{10}&()&()\\D_{20}&()&()}u_0\right)}{y}{} = \fvarss{B_{20}}{y}{}D_{00}u_0 \\
	& +\!  \int_{\eta=0}^{1}\! \bbl(\fvarss{B_{21}}{y}{\eta}\fvarss{D_{10}}{\eta}{}u_0\bbr) +\!  \int_{\mu=0}^{1}\! \bbbl(\sum_{i=0}^{2}\fvarss{\bs{\Phi}_i}{y}{\mu}\fvarss{B_{22}}{y}{\mu}\fvarss{D_{20}}{\mu}{}u_0\bbbr) \\
	&\hspace*{3.5cm}= \fvarss{R_{20}}{y}{} u_0 = \fvars{\left(\mcl{P}\smallbmat{()&()&()\\()&()&()\\R_{20}&()&()}u_0\right)}{y}{},
	\end{align*}
	and	
	\begin{align*}
	&\fvars{\left(\mcl{P}\smallbmat{()&()&()\\()&()&()\\B_{20}&B_{21}&B_{22}} \mcl{P}\smallbmat{()&D_{01}&()\\()&D_{11}&()\\()&D_{21}&()}\mbf{u}_1\right)}{y}{} \\
	&= \fvars{B_{20}}{y}{}\int_{\theta=0}^{1}\bbl(\fvarss{D_{01}}{}{\theta}\fvars{\mbf{u}_1}{\theta}{}\bbr) \\
	&\quad + \int_{\eta=0}^{1}\bbbbl[\fvars{B_{21}}{y}{\eta}\int_{\theta=0}^{1}\bbbl(\sum_{i=0}^{2}\fvarss{\bs{\Phi}_i}{\eta}{\theta}\fvarss{D_{11}^{i}}{\eta}{\theta}\fvars{\mbf{u}_1}{\theta}{}\bbbr)\bbbbr] \\
	&\quad + \int_{\mu=0}^{1}\bbbl[\sum_{i=0}^{2}\fvarss{\bs{\Phi}_i}{y}{\mu}\fvars{B_{22}}{y}{\mu}\int_{\theta=0}^{1}\bbl(\fvarss{D_{21}}{\mu}{\theta}\fvars{\mbf{u}_1}{\theta}{}\bbr)\bbbr] \\
	&\hspace*{2.0cm}= \int_{\theta=0}^{1}\bbl(\fvarss{R_{21}}{y}{\theta}\fvars{\mbf{u}_1}{\theta}{}\bbr) = \fvars{\left(\mcl{P}\smallbmat{()&()&()\\()&()&()\\()&R_{21}&()}\mbf{u}_1\right)}{y}{},
	\end{align*}
	and finally
	\begin{align*}
	&\fvars{\left(\mcl{P}\smallbmat{()&()&()\\()&()&()\\B_{20}&B_{21}&B_{22}} \mcl{P}\smallbmat{()&()&D_{02}\\()&()&D_{12}\\()&()&D_{22}}\mbf{u}_2\right)}{y}{} \\
	& =  \fvars{B_{20}}{y}{}\int_{\nu=0}^{1}\bbl(\fvarss{D_{02}}{}{\nu}\fvars{\mbf{u}_2}{\nu}{}\bbr) 
	+ \fvars{\bl(\mcl{P}[B_{22}]\mcl{P}[D_{22}]\mbf{u}_2\bl)}{y}{} \\
	&\quad + \int_{\eta=0}^{1}\bbbl[\fvars{B_{21}}{y}{\eta}\int_{\nu=0}^{1}\bbl(\fvarss{D_{12}}{\eta}{\nu}\fvars{\mbf{u}_2}{\nu}{}\bbr)\bbbr] \\
	& = \int_{\nu=0}^{1}\bbl(\fvars{B_{20}}{y}{}\fvarss{D_{02}}{}{\nu}\fvars{\mbf{u}_2}{\nu}{}\bbr) +  \fvars{\bbl(\bl[\mcl{P}[B_{22}]\circ\mcl{P}[D_{22}]\br]\mbf{u}_2\bbl)}{y}{} \\
	&\quad + \int_{\nu=0}^{1}\bbbl[\int_{\eta=0}^{1}\bbl(\fvarss{B_{21}}{y}{\eta}\fvarss{D_{12}}{\eta}{\nu}\bbr)\fvars{\mbf{u}_2}{\nu}{}\bbbr] \\
	& =  \fvars{\bl(\mcl{P}[\mcl{L}_{101}(B_{20},D_{02})]\mbf{u}_2\br)}{y}{} + \fvars{\bl(\mcl{P}[\mcl{L}_{1D}(B_{22},D_{22})]\mbf{u}_2\br)}{y}{} \\
	&\quad + \fvars{\bl(\mcl{P}[\mcl{L}_{111}(B_{21},D_{12})]\mbf{u}_2\br)}{y}{} 
	= \fvars{\bl(\mcl{P}[R_{22}]\mbf{u}_2\br)}{x}{}	\\
	&\hspace*{4.6cm} = \fvars{\left(\mcl{P}\smallbmat{()&()&()\\()&()&()\\()&()&R_{22}}\mbf{u}_2\right)}{y}{}.
	\end{align*}
	Combining these results, we conclude $\mcl{P}[B]\circ\mcl{P}[D]=\mcl{P}[R]$.
	
\end{proof}

\subsection{An Algebra of 2D-PI Operators}\label{sec_2Dalgebra_appendix}

We now consider 2D-PI operators, acting on functions $\mbf{u}\in L_2[x,y]$. Defining a parameter space for these operators as
\begin{align}\label{eq_N2D_space_appendix}
\mcl{N}_{2D}^{n\times m}\! &:=\! \bmat{L_2^{n\times m}[x,y]   &\! \! L_2^{n\times m}[x,y,\nu] &\! \! L_2^{n\times m}[x,y,\nu]\\
	L_2^{n\times m}[x,y,\theta] &\! \! L_2^{n\times m}[x,y,\theta,\nu]   &\! \! L_2^{n\times m}[x,y,\theta,\nu]\\
	L_2^{n\times m}[x,y,\theta] &\! \! L_2^{n\times m}[x,y,\theta,\nu]   &\! \! L_2^{n\times m}[x,y,\theta,\nu]},
\end{align}
for any 
\begin{align*}
N&=\bmat{N_{00}&N_{01}&N_{02}\\N_{10}&N_{11}&N_{12}\\N_{20}&N_{21}&N_{22}}\in \mcl{N}_{2D}^{n\times m},
\end{align*}
we may write the associated PI operation using the functions $\bs{\Phi}_i$ as defined in Eqn.~\eqref{eq_Phi_appendix} as 
\begin{align*}
\fvars{\bl(\mcl{P}[N]\mbf{u}\br)}{xy}{}
&=\int_{\theta,\nu=0}^{1}\bbbl(\sum_{i,j=0}^{2}\bbl[\fvarss{\bs{\Phi}_{i}}{x}{\theta}\fvarss{\bs{\Phi}_{j}}{y}{\nu}\fvars{N_{ij}}{xy}{\theta\nu}\bbl]\fvars{\mbf{u}}{\theta\nu}{}\bbbr).
\end{align*}

\begin{lem}\label{lem_2Dalgebra_appendix}
	For any $N:=\smallbmat{N_{00}&N_{01}&N_{02}\\N_{10}&N_{11}&N_{12}\\N_{20}&N_{21}&N_{22}}\in \mcl{N}_{2D}^{n\times k}$ and $M:=\smallbmat{M_{00}&M_{01}&M_{02}\\M_{10}&M_{11}&M_{12}\\M_{20}&M_{21}&M_{22}}\in \mcl{N}_{2D}^{k\times m}$, there exists a unique
	$Q \in \mcl{N}_{2D}^{n\times m}$ such that $\mcl{P}[N]\circ\mcl{P}[M]=\mcl{P}[Q]$. Specifically, we may choose $Q=\mcl{L}_{2D}(N,M)\in\mcl{N}_{2D}^{n\times m}$, where the linear parameter map $\mcl{L}_{2D}:\mcl{N}_{2D}\times\mcl{N}_{2D}\rightarrow\mcl{N}_{2D}$ is such that
	\begin{align}\label{eq_composition_2Dto2D_1_appendix}
	\mcl{L}_{2D}(N,M)=\smallbmat{Q_{00}&Q_{01}&Q_{02}\\Q_{10}&Q_{11}&Q_{12}\\Q_{20}&Q_{21}&Q_{22}}\in\mcl{N}_{2D}^{n\times m},
	\end{align}
	where, defining functions $\bs{\Psi}_{kij}$ as in Eqn.~\eqref{eq_Psi_appendix},
	\begin{align*}
	\fvars{Q_{kr}}{x}{\theta}= 
	\int_{\eta,\mu=0}^{1}\bbbbl(\sum_{i,j,p,q=0}^{2}
	\fvarss{\bs{\Psi}_{kij}}{x}{\eta\theta} \fvarss{\bs{\Psi}_{rpq}}{y}{\mu\nu}
	\fvarss{N_{ip}}{xy}{\eta\mu}
	\fvars{M_{jq}}{\eta\mu}{\theta\nu}\bbbbr),
	\end{align*}
	for each $k\in\{0,1,2\}$.
\end{lem}
\begin{proof}
	Let $\mbf{u}\in L_2^{m}[x,y]$ be arbitrary. Then, applying the results from Proposition~\ref{prop_Psi_appendix}, we find
	\begin{align*}
	&\fvars{\bl(\mcl{P}[N]\mcl{P}[M]\mbf{u}\br)}{xy}{}	\\
	&=\int_{\eta,\mu=0}^{1}\bbbbl[\sum_{i,p=0}^{2}
	\fvarss{\bs{\Phi}_{i}}{x}{\eta}
	\fvarss{\bs{\Phi}_{p}}{y}{\mu}
	\fvars{N_{ip}}{xy}{\eta\mu}	\\
	&\hspace*{2.25cm}
	\int_{\theta,\nu=0}^{1}\bbbl(\sum_{j,q=0}^{2}
	\fvarss{\bs{\Phi}_{j}}{\eta}{\theta}
	\fvarss{\bs{\Phi}_{q}}{\mu}{\nu}
	\fvarss{M_{jq}}{\eta\mu}{\theta\nu}
	\fvars{\mbf{u}}{\theta\nu}{}\bbbr)
	\bbbbr]      \\
	&=\int_{\theta,\nu=0}^{1}\bbbbl[
	\int_{\eta,\mu=0}^{1}\\
	&\hspace*{0.7cm}
	\bbbbl(\sum_{i,j,p,q=0}^{2} 	
	\fvarss{\bs{\Phi}_{i}}{x}{\eta}
	\fvarss{\bs{\Phi}_{j}}{\eta}{\theta}
	\fvarss{\bs{\Phi}_{p}}{y}{\mu}
	\fvarss{\bs{\Phi}_{q}}{\mu}{\nu}
	\fvarss{N_{ip}}{xy}{\eta\mu}
	\fvars{M_{jq}}{\eta\mu}{\theta\nu}
	\bbbbr)\fvars{\mbf{u}}{\theta\nu}{}
	\bbbbr]    \\
	&=\int_{\theta,\nu=0}^{1}\bbbbl[
	\sum_{k,r=0}^{2}
	\fvars{\bs{\Phi}_{k}}{x}{\theta}\fvars{\bs{\Phi}_{r}}{y}{\nu} 
	\int_{\eta,\mu=0}^{1} \\
	&\hspace*{1.5cm}
	\bbbbl(\sum_{i,j,p,q=0}^{2}
	\fvarss{\bs{\Psi}_{kij}}{x}{\eta\theta}
	\fvarss{\bs{\Psi}_{rpq}}{y}{\mu\nu}
	\fvarss{N_{ip}}{xy}{\eta\mu}
	\fvars{M_{jq}}{\eta\mu}{\theta\nu}
	\bbbbr)\fvars{\mbf{u}}{\theta\nu}{}
	\bbbbr]  \\
	&=\int_{\theta,\nu=0}^{1}\bbbbl(
	\sum_{k,r=0}^{2}
	\fvarss{\bs{\Phi}_{k}}{x}{\theta}
	\fvarss{\bs{\Phi}_{r}}{y}{\nu}
	\fvarss{Q_{kr}}{xy}{\theta\nu}
	\fvars{\mbf{u}}{\theta\nu}{}
	\bbbbr)
	=\fvars{(\mcl{P}[Q]\mbf{u})}{xy}{}
	\end{align*}
\end{proof}

\subsection{An Operator from 2D to 011}\label{sec_2Dto011algebra_appendix}

Having defined algebras of operators on $RL^{n_0,n_1}[x,y]:=\R^{n_0}\times L_2^{n_1}[x]\times L_2^{n_1}[y]$ and $L_2^{n_2}[x,y]$, it remains to define operators mapping $L_2[x,y]\rightarrow RL[x,y]$ and back. For this first mapping, we first define a parameter space 
\begin{align*}
\mcl{N}^{n\times m}_{2D\rightarrow 1D}&:=L_2^{n\times m}[x,\nu]\times L_2^{n\times m}[x,\theta,\nu]\times L_2^{n\times m}[x,\theta,\nu],
\end{align*}
with associated operator $\mcl{P}[N]:L_2^{m}[x,y]\rightarrow L_2^{n}[x]$ defined such that, for arbitrary $N=\{N_0,N_1,N_2\}\in\mcl{N}^{n\times m}_{2D\rightarrow 1D}$,
\begin{align*}
\fvars{(\mcl{P}[N]\mbf{u})}{x}{}:=\int_{\theta,\mu=0}^{1}\bbbbl(\sum_{i=0}^{2}\bbl[\fvarss{\bs{\Phi}_i}{x}{\theta}\fvarss{N_i}{x}{\theta\nu}\fvars{\mbf{u}}{\theta\nu}{}\bbr]\bbbbl).
\end{align*}
This operator maps functions on two variables to functions on a single variable, allowing us to build a mapping from $L_2[x,y]$ to $RL[x,y]$. In particular, letting
\begin{align}\label{eq_N2Dto011_space_appendix}
\mcl{N}_{2D\rightarrow 011}\smallbmat{n_0\\n_1\\m_2}&:=\bmat{L_2^{n_0\times m_2}[\theta,\nu]\\\mcl{N}_{2D\rightarrow 1D}^{n_1\times m_2}\\\mcl{N}_{2D\rightarrow 1D}^{n_1\times m_2}},	
\end{align}
an associated operator $\mcl{P}[D]:L_2^{m_2}[x,y]\rightarrow RL^{n_0,n_1}[x,y]$ may be defined such that, for any $D=\smallbmat{D_0\\D_1\\D_2}\in\mcl{N}_{2D\rightarrow 011}\smallbmat{n_0\\n_1\\m_2}$ and $\mbf{u}\in L_2^{m_2}[x,y]$,
\begin{align*}
\fvars{(\mcl{P}[D]\mbf{u})}{xy}{}:=\bmat{\int_{\theta,\nu=0}^{1}\bbl(\fvarss{D_0}{}{\theta\nu}\fvars{\mbf{u}}{\theta\nu}{}\bbr)\\
	\fvars{(\mcl{P}[D_1]\mbf{u})}{x}{}\\
	\fvars{(\mcl{P}[D_2]\mbf{u})}{y}{}}.
\end{align*}
This operator allows us to map functions $\mbf{u}\in L_2[x,y]$ to functions $\mbf{v}=\smallbmat{v_0\\\mbf{v}_1\\\mbf{v}_2}\in\smallbmat{\R\\L_2[x]\\L_2[y]}$, which is necessary to map state variables living on the interior of a 2D domain to state variables living on the boundary. An important property of this operator is also that its composition with 011- and 2D-PI operators returns another 2D$\rightarrow$011-PI operator, as described in the following lemmas.
\begin{lem}\label{lem_011x2Dto1Dalgebra_appendix}
	For any $B=\smallbmat{B_{00}&B_{01}&B_{02}\\B_{10}&B_{11}&B_{12}\\B_{20}&B_{21}&B_{22}}\in\mcl{N}_{011}\smallbmat{n_0&p_0\\n_1&p_1}$ and $D=\smallbmat{D_{0}\\D_{1}\\D_{2}}\in\mcl{N}_{2D\rightarrow 011}\smallbmat{p_0\\p_1\\m_2}$, where
	\begin{align*} &B_{11}=\{B_{11}^{0},B_{11}^{1},B_{11}^{2}\}\in\mcl{N}_{1D},\\ 
	&B_{22}=\{B_{22}^{0},B_{22}^{1},B_{22}^{2}\}\in\mcl{N}_{1D},\\ 
	&D_{1}=\{D_{1}^{0},D_{1}^{1},D_{1}^{2}\}\in\mcl{N}_{2D\rightarrow 1D},\\ 
	&D_{2}=\{D_{2}^{0},D_{2}^{1},D_{2}^{2}\}\in\mcl{N}_{2D\rightarrow 1D},
	\end{align*}
	there exists a unique $S\in\mcl{N}_{2D\rightarrow 011}\smallbmat{n_0\\n_1\\m_2}$ such that $\mcl{P}[B]\circ\mcl{P}[D]=\mcl{P}[S]$. Specifically, we may choose $S=\mcl{L}_{2D\rightarrow 011}^{1}(B,D)\in\mcl{N}_{2D\rightarrow 011}\smallbmat{n_0\\n_1\\m_2}$, where the linear parameter map $\mcl{L}_{2D\rightarrow 011}^{1}:\mcl{N}_{011}\times\mcl{N}_{2D\rightarrow 011}\rightarrow\mcl{N}_{2D\rightarrow 011}$ is defined such that
	\begin{align}\label{eq_composition_2Dto011_1_appendix}
	\mcl{L}_{2D\rightarrow 011}^{1}(B,D)=\bmat{S_{0}\\S_{1}\\S_{2}}\in\mcl{N}_{2D\rightarrow 011}\smallbmat{n_0\\n_1\\m_2},
	\end{align}
	with
	\begin{align*}
	\fvars{S_{0}}{}{\theta\nu}&=B_{00}\fvars{D_{0}}{}{\theta\nu} + \int_{\eta=0}^{1}\bbbbl(\fvars{B_{01}}{}{\eta}\sum_{i=0}^{2}\bbl[\fvarss{\bs{\Phi}_i}{\eta}{\theta}\fvars{D_{1}^{i}}{\eta}{\theta\nu}\bbr]\bbbbr)\\
	&\qquad + \int_{\mu=0}^{1}\bbbbl(\fvars{B_{02}}{}{\mu}\sum_{i=0}^{2}\bbl[\fvarss{\bs{\Phi}_i}{\mu}{\nu}\fvars{D_{2}^{i}}{\mu}{\theta\nu}\bbr]\bbbbr) \\ 
	S_{1}&=\mcl{L}_{0}(B_{10},D_{0})+\mcl{L}_{1}(B_{11},D_{1})+\mcl{L}_{2}(B_{12},D_{2})	\\
	S_{2}&=\mcl{L}_{0}(B_{20},D_{0})+\mcl{L}_{2}(B_{21},D_{1})+\mcl{L}_{1}(B_{22},D_{2}),	
	\end{align*}
	where $\mcl{L}_{0}:L_2[x]\times L_2[\theta,\nu]\rightarrow \mcl{N}_{2D\rightarrow 1D}$ is defined such that, for arbitrary $N\in L_2[x]$, $M\in L_2[\theta,\nu]$, $\mcl{L}_{0}(N,M)=\{Q_0,Q_1,Q_2\}$ with
	\begin{align*}
	\fvars{Q_{0}}{x}{\theta\nu}&=0	&
	\fvars{Q_{1}}{x}{\theta\nu}&=\fvars{Q_{2}}{x}{\theta\nu}
	=\fvarss{N}{x}{}\fvars{M}{}{\theta\nu},
	\end{align*}
	$\mcl{L}_{1}:\mcl{N}_{1D}\times \mcl{N}_{2D\rightarrow 1D}\rightarrow \mcl{N}_{2D\rightarrow 1D}$ is defined such that, for arbitrary $N=\{N_0,N_1,N_2\}\in \mcl{N}_{1D}$, $M=\{M_0,M_1,M_2\}\in \mcl{N}_{2D\rightarrow 1D}$, $\mcl{L}_{1}(N,M)=\{Q_0,Q_1,Q_2\}$ with
	\begin{align*}
	\fvars{Q_{k}}{x}{\theta\nu}&=\int_{\eta=0}^{1}\bbbl(\sum_{i,j=0}^{2}\fvarss{\bs{\Psi}_{kij}}{x}{\eta\theta}\fvarss{N_i}{x}{\eta}\fvars{M_j}{\eta}{\theta\nu}\bbbr)
	\end{align*}
	for $k\in\{0,1,2\}$, and $\mcl{L}_{2}:L_2[x,\nu]\times \mcl{N}_{2D\rightarrow 1D}\rightarrow \mcl{N}_{2D\rightarrow 1D}$ is defined such that, for arbitrary $N\in L_2[x,\nu]$, $M=\{M_0,M_1,M_2\}\in \mcl{N}_{2D\rightarrow 1D}$, $\mcl{L}_{2}(N,M)=\{Q_0,Q_1,Q_2\}$, where
	\begin{align*}
	\fvars{Q_{0}}{x}{\theta\nu}&=0	&
	\fvars{Q_{1}}{x}{\theta\nu}&=\fvars{Q_{2}}{x}{\theta\nu}=\int_{\mu=0}^{1}\bbbl(\sum_{i=0}^{2}\fvarss{\bs{\Phi}_{i}}{\mu}{\nu}\fvarss{N}{x}{\mu}\fvars{M_i}{\mu}{\theta\nu}\bbbr).
	\end{align*}
	
\end{lem}

\begin{proof}
	To prove this lemma, we will exploit the linear structure of 011-PI operators, allowing us to express
	\begin{align*}
	&\mcl{P}\smallbmat{B_{00}&B_{01}&B_{02}\\B_{10}&B_{11}&B_{12}\\B_{20}&B_{21}&B_{22}}\mcl{P}\smallbmat{D_0\\D_1\\D_2} = 
	\mcl{P}\smallbmat{B_{00}&B_{01}&B_{02}\\0&0&0\\0&0&0}\mcl{P}\smallbmat{D_0\\D_1\\D_2}\\
	&\quad+ \mcl{P}\smallbmat{0&0&0\\B_{10}&B_{11}&B_{12}\\0&0&0}\mcl{P}\smallbmat{D_0\\D_1\\D_2} + \mcl{P}\smallbmat{0&0&0\\0&0&0\\B_{20}&B_{21}&B_{22}}\mcl{P}\smallbmat{D_0\\D_1\\D_2}.
	\end{align*}
	Considering each of these terms separately, we may invoke the definitions of the different operators to find that, for arbitrary $\mbf{u}\in L_2^{m_2}[x,y]$,
	\begin{align*}
	&\mcl{P}\smallbmat{B_{00}&B_{01}&B_{02}\\()&()&()\\()&()&()} \mcl{P}\smallbmat{D_{0}\\D_{1}\\D_{2}}\mbf{u} = B_{00}\int_{\theta,\nu=0}^{1}\bbl(\fvarss{D_{0}}{}{\theta\nu}\fvars{\mbf{u}}{\theta\nu}{}\bbr) \\
	&\qquad + \int_{\eta=0}^{1}\fvars{B_{01}}{}{\eta}\bbbbl(\int_{\theta,\mu=0}^{1}\bbbl[\sum_{i=0}^{2}\bbl(\fvarss{\bs{\Phi}_i}{\eta}{\theta}\fvarss{D_{1}^{i}}{\eta}{\theta\nu}\fvars{\mbf{u}}{\theta\nu}{}\bbr)\bbbl]\bbbbr)\\
	&\qquad + \int_{\mu=0}^{1}\fvars{B_{02}}{}{\mu}\bbbbl(\int_{\theta,\mu=0}^{1}\bbbl[\sum_{i=0}^{2}\bbl(\fvarss{\bs{\Phi}_i}{\mu}{\nu}\fvarss{D_{2}^{i}}{\mu}{\theta\nu}\fvars{\mbf{u}}{\theta\nu}{}\bbr)\bbbl]\bbbbr)\\
	&\hspace*{2.75cm}=\int_{\theta,\nu=0}^{1}\bbl(\fvarss{S_{0}}{}{\theta\nu}\fvars{\mbf{u}}{\theta\nu}{}\bbr)
	=\mcl{P}\smallbmat{S_{0}\\()\\()}\mbf{u}.
	\end{align*}
	In addition, using Corollary~\ref{cor_integral_Phi12_appendix} and Proposition~\ref{prop_Psi_appendix}, it follows that
	\begin{align*}
	&\fvars{\left(\mcl{P}\smallbmat{()&()&()\\B_{10}&B_{11}&B_{12}\\()&()&()} \mcl{P}\smallbmat{D_{0}\\D_{1}\\D_{2}}\mbf{u}\right)}{x}{}
	\! =\! \fvars{B_{10}}{x}{}\! \int_{\theta,\nu=0}^{1}\! \bbl(\fvarss{D_{0}}{}{\theta\nu}\fvars{\mbf{u}}{\theta\nu}{}\bbr) \\
	&+\!  \int_{\eta=0}^{1}\! \bbbbl(\sum_{i=0}^{2}\! \bbbl[\fvarss{\mbf{\Phi}_{i}}{x}{\eta}\fvars{B_{11}^{i}}{x}{\eta}\int_{\theta,\nu=0}^{1}\! \bbbl(\sum_{j=0}^{2}\bbl[\fvarss{\bs{\Phi}_j}{\eta}{\theta}\fvarss{D_{1}^{j}}{\eta}{\theta\nu}\fvars{\mbf{u}}{\theta\nu}{}\bbr]\bbbl)\bbbr]\bbbbr)\\
	&\hspace*{1.6cm} +\!  \int_{\mu=0}^{1}\! \fvars{B_{12}}{x}{\mu}\bbbbl(\int_{\theta,\nu=0}^{1}\! \bbbbl[\sum_{i=0}^{2}\bbl(\fvarss{\bs{\Phi}_i}{\mu}{\nu}\fvarss{D_{2}^{i}}{\mu}{\theta\nu}\fvars{\mbf{u}}{\theta\nu}{}\bbr)\bbbbl]\bbbbr)\\
	&=\int_{\theta,\nu=0}^{1}\bbbl(\sum_{k=1}^{2}\bbl[\fvarss{\bs{\Phi}_k}{x}{\theta}\fvarss{B_{10}}{x}{}\fvars{D_{0}}{}{\theta\nu}\bbr]\fvars{\mbf{u}}{\theta\nu}{}\bbr) \\
	&+\!  \int_{\theta,\nu=0}^{1}\! \bbbbl(\sum_{k=0}^{2}\fvars{\mbf{\Phi}_{k}}{x}{\theta}\int_{\eta=0}^{1}\! \bbbbl[\sum_{i,j=0}^{2}\fvarss{\mbf{\Psi}_{kij}}{x}{\eta\theta}\fvars{B_{11}^{i}}{x}{\eta}\fvars{D_{1}^{j}}{\eta}{\theta\nu}\bbbbr]\fvars{\mbf{u}}{\theta\nu}{}\bbbbr)\\
	&\hspace*{0.55cm}+\! \int_{\theta,\nu=0}^{1}\! \bbbbl(\sum_{k=1}^{2}\fvars{\bs{\Phi}_k}{x}{\theta}\int_{\mu=0}^{1}\! \bbbbl[\sum_{i=0}^{2}\fvarss{\bs{\Phi}_i}{\mu}{\nu}\fvarss{B_{12}}{x}{\mu}\fvars{D_{2}^{i}}{\mu}{\theta\nu}\bbbbl]\fvars{\mbf{u}}{\theta\nu}{}\bbbbr)\\
	&=\fvars{\bl(\mcl{P}[\mcl{L}_{0}(B_{10},D_{0})]\mbf{u}\br)}{x}{} +\fvars{\bl(\mcl{P}[\mcl{L}_{1}(B_{11},D_{1})]\mbf{u}\br)}{x}{}	\\ &\hspace*{2.5cm}+\fvars{\bl(\mcl{P}[\mcl{L}_{2}(B_{12},D_{2})]\mbf{u}\br)}{x}{} =\fvars{\left(\mcl{P}\smallbmat{()\\S_{1}\\()}\mbf{u}\right)}{x}.
	\end{align*}	
	Finally, by the same approach,
	\begin{align*}
	&\fvars{\left(\mcl{P}\smallbmat{()&()&()\\()&()&()\\B_{20}&B_{21}&B_{22}} \mcl{P}\smallbmat{D_{0}\\D_{1}\\D_{2}}\mbf{u}\right)}{y}{}
	\! =\!  \fvars{B_{20}}{y}{}\! \int_{\theta,\nu=0}^{1}\! \bbl(\fvarss{D_{0}}{}{\theta\nu}\fvars{\mbf{u}}{\theta\nu}{}\bbr) \\
	&+\!  \int_{\mu=0}^{1}\! \bbbbl(\sum_{i=0}^{2}\bbbl[\fvarss{\mbf{\Phi}_{i}}{y}{\mu}\fvars{B_{22}^{i}}{y}{\mu}\int_{\theta,\nu=0}^{1}\! \bbbl(\sum_{j=0}^{2}\bbl[\fvarss{\bs{\Phi}_j}{\mu}{\nu}\fvarss{D_{2}^{j}}{\mu}{\theta\nu}\fvars{\mbf{u}}{\theta\nu}{}\bbr]\bbbl)\bbbr]\bbbbr)\\
	&\hspace*{1.6cm} +\!  \int_{\eta=0}^{1}\! \fvars{B_{21}}{y}{\eta}\bbbbl(\int_{\theta,\nu=0}^{1}\! \bbbbl[\sum_{i=0}^{2}\bbl(\fvarss{\bs{\Phi}_i}{\eta}{\theta}\fvarss{D_{1}^{i}}{\eta}{\theta\nu}\fvars{\mbf{u}}{\theta\nu}{}\bbr)\bbbbl]\bbbbr)\\
	&=\int_{\theta,\nu=0}^{1}\bbbl(\sum_{k=1}^{2}\bbl[\fvarss{\bs{\Phi}_k}{y}{\nu}\fvarss{B_{20}}{y}{}\fvars{D_{0}}{}{\theta\nu}\bbr]\fvars{\mbf{u}}{\theta\nu}{}\bbr) \\
	&+\! \int_{\theta,\nu=0}^{1}\! \bbbbl(\sum_{k=0}^{2}\fvars{\mbf{\Phi}_{k}}{y}{\nu}\int_{\mu=0}^{1}\! \bbbbl[\sum_{i,j=0}^{2}\fvarss{\mbf{\Psi}_{kij}}{y}{\mu\nu}\fvars{B_{22}^{i}}{y}{\mu}\fvars{D_{2}^{j}}{\mu}{\theta\nu}\bbbbr]\fvars{\mbf{u}}{\theta\nu}{}\bbbbr)\\
	&\hspace*{0.55cm}+\!  \int_{\theta,\nu=0}^{1}\! \bbbbl(\sum_{k=1}^{2}\fvars{\bs{\Phi}_k}{y}{\nu}\int_{\eta=0}^{1}\! \bbbl[\sum_{i=0}^{2}\fvarss{\bs{\Phi}_i}{\eta}{\theta}\fvarss{B_{21}}{y}{\eta}\fvars{D_{1}^{i}}{\eta}{\theta\nu}\bbbl]\fvars{\mbf{u}}{\theta\nu}{}\bbbbr)\\
	&\quad=\fvars{\bl(\mcl{P}[\mcl{L}_{0}(B_{20},D_{0})]\mbf{u}\br)}{x}{} +	\fvars{\bl(\mcl{P}[\mcl{L}_{1}(B_{22},D_{2})]\mbf{u}\br)}{x}{}\\ &\hspace*{2.5cm}+\fvars{\bl(\mcl{P}[\mcl{L}_{2}(B_{21},D_{1})]\mbf{u}\br)}{x}{} =\fvars{\left(\mcl{P}\smallbmat{()\\()\\S_{2}}\mbf{u}\right)}{y}{}.
	\end{align*}
	Combining the results, we conclude that $\mcl{P}[B]\circ\mcl{P}[D]=\mcl{P}[S]$.
	
\end{proof}

\begin{lem}\label{lem_2Dto1Dx2Dalgebra_appendix}
	For any $N=\smallbmat{N_{00}&N_{01}&N_{02}\\N_{10}&N_{11}&N_{12}\\N_{20}&N_{21}&N_{22}}\in\mcl{N}_{2D}^{p_2\times m_2}$ and $D=\smallbmat{D_{0}\\D_{1}\\D_{2}}\in\mcl{N}_{2D\rightarrow 011}\smallbmat{n_0\\n_1\\p_2}$, where
	\begin{align*}
	&D_{1}=\{D_{1}^{0},D_{1}^{1},D_{1}^{2}\}\in\mcl{N}_{2D\rightarrow 1D},\\ 
	&D_{2}=\{D_{2}^{0},D_{2}^{1},D_{2}^{2}\}\in\mcl{N}_{2D\rightarrow 1D},
	\end{align*}
	there exists a unique $S\in\mcl{N}_{2D\rightarrow 011}\smallbmat{n_0\\n_1\\m_2}$ such that $\mcl{P}[D]\circ\mcl{P}[N]=\mcl{P}[S]$. Specifically, we may choose $S=\mcl{L}_{2D\rightarrow 011}^{2}(D,N)\in\mcl{N}_{2D\rightarrow 011}\smallbmat{n_0\\n_1\\m_2}$, where the linear parameter map $\mcl{L}_{2D\rightarrow 011}^{2}:\mcl{N}_{2D\rightarrow 011}\times\mcl{N}_{2D}\rightarrow\mcl{N}_{2D\rightarrow 011}$ is defined such that
	\begin{align}\label{eq_composition_2Dto011_2_appendix}
	\mcl{L}_{2D\rightarrow 011}^{2}(D,N)=\bmat{S_{0}\\S_{1}\\S_{2}}\in\mcl{N}_{2D\rightarrow 011}\smallbmat{n_0\\n_1\\m_2},
	\end{align}
	where
	\begin{align*}
	&\fvars{S_{0}}{}{\theta\nu}=\int_{\eta,\mu=0}^{1}\bbbbl(\sum_{j,q=0}^{2}\fvarss{\bs{\Phi}_j}{\eta}{\theta}\fvarss{\bs{\Phi}_q}{\mu}{\nu}\fvarss{D_0}{}{\eta\mu}\fvars{N_{jq}}{\eta\mu}{\theta\nu}\bbbbr)
	\end{align*}
	and where
	\begin{align*}
	&S_{1}=\{S_{1}^{0},S_{1}^{1},S_{1}^{2}\},	&
	&S_{2}=\{S_{2}^{0},S_{2}^{1},S_{2}^{2}\},	
	\end{align*}
	with
	\begin{align*}
	\fvars{S_{1}^{k}}{x}{\theta\nu}=\int_{\eta,\mu=0}^{1}\bbbbl(\sum_{i,j,q=0}^{2}\fvarss{\bs{\Psi}_{kij}}{x}{\eta\theta}\fvarss{\bs{\Phi}_q}{\mu}{\nu}\fvarss{D_1^{i}}{x}{\eta\mu}\fvars{N_{jq}}{\eta\mu}{\theta\nu}\bbbbr)	\\
	\fvars{S_{2}^{r}}{y}{\theta\nu}=\int_{\eta,\mu=0}^{1}\bbbbl(\sum_{j,p,q=0}^{2}\fvarss{\bs{\Phi}_j}{\eta}{\theta}\fvarss{\bs{\Psi}_{rpq}}{y}{\mu\nu}\fvarss{D_2^{p}}{y}{\eta\mu}\fvars{N_{jq}}{\eta\mu}{\theta\nu}\bbbbr),	\\
	\end{align*}
	for $k,r\in\{0,1,2\}$.
	
\end{lem}

\begin{proof}
	
	To prove this result, we once again note that, by linearity of the PI operators,
	\begin{align*}
	\mcl{P}[D]\mcl{P}[N]= \mcl{P}\smallbmat{D_0\\0\\0}\mcl{P}[N]+ \mcl{P}\smallbmat{0\\D_1\\0}\mcl{P}[N]+ \mcl{P}\smallbmat{0\\0\\D_2}\mcl{P}[N].
	\end{align*}
	Considering each of these terms separately, we find that, for arbitrary $\mbf{u}\in L_2[x,y]$,
	\begin{align*}
	&\mcl{P}\smallbmat{D_0\\()\\()}\mcl{P}[N]\mbf{u}\\
	&=\int_{\eta,\mu=0}^{1}\bbbbl(\fvars{D_0}{}{\eta\mu}\int_{\theta,\nu=0}^{1}\bbbbl[\sum_{j,q=0}^{2}\fvarss{\bs{\Phi}_j}{\eta}{\theta}\fvarss{\bs{\Phi}_q}{\mu}{\nu}\fvarss{N_{jq}}{\eta\mu}{\theta\nu}\fvars{\mbf{u}}{\theta\nu}{}\bbbbr]\bbbbr)\\
	&=\int_{\theta,\nu=0}^{1}\bbbbl(\int_{\eta,\mu=0}^{1}
	\bbbbl[\sum_{j,q=0}^{2}\fvarss{\bs{\Phi}_j}{\eta}{\theta}\fvarss{\bs{\Phi}_q}{\mu}{\nu}\fvarss{D_0}{}{\eta\mu}\fvars{N_{jq}}{\eta\mu}{\theta\nu}\bbbbr]\fvars{\mbf{u}}{\theta\nu}{}\bbbbr)\\
	&=\int_{\theta,\nu=0}^{1}\bbl(\fvarss{S_0}{}{\theta\nu}\fvars{\mbf{u}}{\theta\nu}{}\bbr)=\mcl{P}\smallbmat{S_0\\()\\()}\mbf{u}.
	\end{align*}	
	Similarly, using Proposition~\ref{prop_Psi_appendix}, we find
	\begin{align*}
	&\fvars{\left(\mcl{P}\smallbmat{()\\D_1\\()}\mcl{P}[N]\mbf{u}\right)}{x}{}\\
	&=\int_{\eta,\mu=0}^{1}\bbbbl(\sum_{i=0}^{2}\fvarss{\bs{\Phi}_i}{x}{\eta}\fvars{D_1^{i}}{x}{\eta\mu}\int_{\theta,\nu=0}^{1}\\
	&\hspace*{4.0cm}\bbbbl[\sum_{j,q=0}^{2}\fvarss{\bs{\Phi}_j}{\eta}{\theta}\fvarss{\bs{\Phi}_q}{\mu}{\nu}\fvarss{N_{jq}}{\eta\mu}{\theta\nu}\fvars{\mbf{u}}{\theta\nu}{}\bbbbr]\bbbbr)\\
	&=\int_{\theta,\nu=0}^{1}\bbbbl(\sum_{k=0}^{2}\fvars{\bs{\Phi}_k}{x}{\theta}\int_{\eta,\mu=0}^{1}\\
	&\hspace*{2.5cm}\bbbbl[\sum_{i,j,q=0}^{2}\fvarss{\bs{\Psi}_{kij}}{x}{\eta\theta}\fvarss{\bs{\Phi}_q}{\mu}{\nu}\fvarss{D_1^{i}}{x}{\eta\mu}\fvars{N_{jq}}{\eta\mu}{\theta\nu}\bbbbr]\fvars{\mbf{u}}{\theta\nu}{}\bbbbr)\\
	&=\int_{\theta,\nu=0}^{1}\bbbbl(\sum_{k=0}^{2}\fvarss{\bs{\Phi}_k}{x}{\theta\nu}\fvarss{S_1^k}{x}{\theta\nu}\fvars{\mbf{u}}{\theta\nu}{}\bbbbr)=\fvars{\left(\mcl{P}\smallbmat{()\\S_1\\()}\mbf{u}\right)}{x}{},
	\end{align*}
	and
	\begin{align*}
	&\fvars{\left(\mcl{P}\smallbmat{()\\()\\D_2}\mcl{P}[N]\mbf{u}\right)}{y}{}\\
	&=\int_{\eta,\mu=0}^{1}\bbbbl(\sum_{p=0}^{2}\fvarss{\bs{\Phi}_p}{y}{\mu}\fvars{D_2^{p}}{y}{\eta\mu}\int_{\theta,\nu=0}^{1}\\
	&\hspace*{4.0cm}\bbbbl[\sum_{j,q=0}^{2}\fvarss{\bs{\Phi}_j}{\eta}{\theta}\fvarss{\bs{\Phi}_q}{\mu}{\nu}\fvarss{N_{jq}}{\eta\mu}{\theta\nu}\fvars{\mbf{u}}{\theta\nu}{}\bbbbr]\bbbbr)\\
	&=\int_{\theta,\nu=0}^{1}\bbbbl(\sum_{r=0}^{2}\fvars{\bs{\Phi}_r}{y}{\nu}\int_{\eta,\mu=0}^{1}\\
	&\hspace*{2.5cm}\bbbbl[\sum_{j,p,q=0}^{2}\fvarss{\bs{\Phi}_j}{\eta}{\theta}\fvarss{\bs{\Psi}_{rpq}}{y}{\mu\nu}\fvarss{D_2^{p}}{y}{\eta\mu}\fvars{N_{jq}}{\eta\mu}{\theta\nu}\bbbbr]\fvars{\mbf{u}}{\theta\nu}{}\bbbbr)\\
	&=\int_{\theta,\nu=0}^{1}\bbbbl(\sum_{r=0}^{2}\fvarss{\bs{\Phi}_r}{x}{\theta\nu}\fvarss{S_2^r}{y}{\theta\nu}\fvars{\mbf{u}}{\theta\nu}{}\bbbbr)=\fvars{\left(\mcl{P}\smallbmat{()\\()\\S_2}\mbf{u}\right)}{y}{}.
	\end{align*}
	Combining the results, we conclude that $\mcl{P}[D]\circ\mcl{P}[N]=\mcl{P}[S]$.
	
\end{proof}

\subsection{An Operator from 011 to 2D}\label{sec_011to2Dalgebra_appendix}

In addition to the operator mapping $L_2^{n_2}[x,y]$ to $RL^{n_0,n_1}[x,y]:=\R^{n_0}\times L_2^{n_1}[x]\times L_2^{n_1}[y]$ defined in the previous section, we also define an operator performing the inverse of this mapping. For this, we first define a parameter space
\begin{align*}
\mcl{N}^{n\times m}_{1D\rightarrow 2D}&:=L_2^{n\times m}[x,y]\times L_2^{n\times m}[x,y,\theta]\times L_2^{n\times m}[x,y,\theta]	
\end{align*}
with associated operator $\mcl{P}[N]:L_2^{m}[x]\rightarrow L_2^{n}[x,y]$ defined such that, for arbitrary $N=\{N_0,N_1,N_2\}\in\mcl{N}^{n\times m}_{1D\rightarrow 2D}$ and $\mbf{u}\in L_2^{m}[x]$,
\begin{align*}
\fvars{(\mcl{P}[N]\mbf{u})}{xy}{}:=\int_{\theta=0}^{1}\bbbbl(\sum_{i=0}^{2}\bbl[\fvarss{\bs{\Phi}_i}{x}{\theta}\fvarss{N_i}{xy}{\theta}\fvars{\mbf{u}}{\theta}{}\bbr]\bbbbl).
\end{align*}
Building upon this, we define yet another space
\begin{align}\label{eq_N011to2D_space_appendix}
\mcl{N}_{011\rightarrow 2D}\smallbmat{m_0\\m_1\\n_2}&:=\bmat{L_2^{n_2\times m_0}[x,y]\\\mcl{N}_{1D\rightarrow 2D}^{n_2\times m_1}\\\mcl{N}_{1D\rightarrow 2D}^{n_2\times m_1}},
\end{align}
with associated PI operator $\mcl{P}[E]:RL^{m_0,m_1}\rightarrow L_2^{n_2}[x,y]$ such that, for arbitrary $E=\smallbmat{E_0\\E_1\\E_2}\in\mcl{N}_{011\rightarrow 2D}\smallbmat{m_0\\m_1\\n_2}$ and $\mbf{u}=\smallbmat{u_0\\\mbf{u}_1\\\mbf{u}_2}\in RL^{m_0,m_1}[x,y]$,
\begin{align*}
\fvars{(\mcl{P}[E]\mbf{u})}{xy}{}:=\fvarss{E_0}{xy}{}u_0+\fvars{(\mcl{P}[E_1]\mbf{u}_1)}{xy}{}+\fvars{(\mcl{P}[E_2]\mbf{u})}{xy}{}.
\end{align*}
This operator maps functions in $\mbf{u}=\smallbmat{u_0\\\mbf{u}_1\\\mbf{u}_2}\in\smallbmat{\R\\L_2[x]\\L_2[y]}$ to functions $\mbf{v}\in L_2[x,y]$, allowing us to map state variables living on the boundary of a 2D domain to state variables living on its interior. As was the case for 2D$\rightarrow$011-PI operators, the composition of 011$\rightarrow$2D-PI operators with 011- and 2D-PI operators too can be expressed as a PI operator, as described in the following lemmas.

\begin{lem}\label{lem_1Dto2Dx011algebra_appendix}
	For any $E=\smallbmat{E_{0}\\E_{1}\\E_{2}}\in\mcl{N}_{011\rightarrow 2D}\smallbmat{p_0\\p_1\\n_2}$ and $B=\smallbmat{B_{00}&B_{01}&B_{02}\\B_{10}&B_{11}&B_{12}\\B_{20}&B_{21}&B_{22}}\in\mcl{N}_{011}\smallbmat{p_0&m_0\\p_1&m_1}$, where
	\begin{align*} &B_{11}=\{B_{11}^{0},B_{11}^{1},B_{11}^{2}\}\in\mcl{N}_{1D}^{p_1\times m_1},\\ 
	&B_{22}=\{B_{22}^{0},B_{22}^{1},B_{22}^{2}\}\in\mcl{N}_{1D}^{p_1\times m_1},\\ 
	&E_{1}=\{E_{1}^{0},E_{1}^{1},E_{1}^{2}\}\in\mcl{N}_{1D\rightarrow 2D}^{n_2\times m_1},\\ 
	&E_{2}=\{E_{2}^{0},E_{2}^{1},E_{2}^{2}\}\in\mcl{N}_{1D\rightarrow 2D}^{n_2\times m_1},
	\end{align*}
	there exists a unique $T\in\mcl{N}_{011\rightarrow 2D}\smallbmat{m_0\\m_1\\n_2}$ such that $\mcl{P}[E]\circ\mcl{P}[B]=\mcl{P}[T]$. Specifically, we may choose $T=\mcl{L}_{011\rightarrow 2D}^{1}(E,B)\in\mcl{N}_{011\rightarrow 2D}\smallbmat{m_0\\m_1\\n_2}$, where the linear parameter map $\mcl{L}_{011\rightarrow 2D}^{1}:\mcl{N}_{011\rightarrow 2D}\times\mcl{N}_{011}\rightarrow\mcl{N}_{011\rightarrow 2D}$ is defined such that
	\begin{align}\label{eq_composition_011to2D_1_appendix}
	\mcl{L}_{011\rightarrow 2D}^{1}(E,B)=\bmat{T_{0}\\T_{1}\\T_{2}}\in\mcl{N}_{011\rightarrow 2D}\smallbmat{m_0\\m_1\\n_2},
	\end{align}
	where
	\begin{align*}
	&\fvars{T_{0}}{xy}{}=\fvars{E_{0}}{xy}{}B_{00} + \int_{\eta=0}^{1}\bbbbl(\sum_{i=0}^{2}\bbl[\fvarss{\bs{\Phi}_i}{x}{\eta}\fvars{E_{1}^{i}}{xy}{\eta}\fvars{B_{10}}{\eta}{}\bbr]\bbbbr)\\
	&\qquad\quad + \int_{\mu=0}^{1}\bbbbl(\sum_{i=0}^{2}\bbl[\fvarss{\bs{\Phi}_i}{y}{\mu}\fvars{E_{2}^{i}}{xy}{\mu}\fvars{B_{20}}{\mu}{}\bbr]\bbbbr) \\ 
	&T_{1}=\mcl{L}_{0}(E_{0},B_{01})+\mcl{L}_{1}(E_{1},B_{11})+\mcl{L}_{2}(E_{2},B_{21})	\\
	&T_{2}=\mcl{L}_{0}(E_{0},B_{02})+\mcl{L}_{2}(E_{1},B_{12})+\mcl{L}_{1}(E_{2},B_{22}),	
	\end{align*}
	and where $\mcl{L}_{0}:L_2[\theta,\nu]\times L_2[x]\rightarrow \mcl{N}_{1D\rightarrow 2D}$ is defined such that, for arbitrary $N\in L_2[x,y]$, $M\in L_2[\theta]$, $\mcl{L}_{0}(N,M)=\{Q_0,Q_1,Q_2\}$ with
	\begin{align*}
	\fvars{Q_{0}}{xy}{\theta}&=0	&
	\fvars{Q_{1}}{xy}{\theta}&=\fvars{Q_{2}}{xy}{\theta}
	=\fvarss{N}{xy}{}\fvars{M}{}{\theta},
	\end{align*}
	and $\mcl{L}_{1}:\mcl{N}_{1D\rightarrow 2D}\times \mcl{N}_{1D}\rightarrow \mcl{N}_{1D\rightarrow 2D}$ is defined such that, for arbitrary $N=\{N_0,N_1,N_2\}\in \mcl{N}_{1D\rightarrow 2D}$, $M=\{M_0,M_1,M_2\}\in \mcl{N}_{1D}$, $\mcl{L}_{1}(N,M)=\{Q_0,Q_1,Q_2\}$ with
	\begin{align*}
	\fvars{Q_{k}}{xy}{\theta}&=\int_{\eta=0}^{1}\bbbl(\sum_{i,j=0}^{2}\fvarss{\bs{\Psi}_{kij}}{x}{\eta\theta}\fvarss{N_i}{xy}{\eta}\fvars{M_j}{\eta}{\theta}\bbbr)
	\end{align*}
	for $k\in\{0,1,2\}$, and $\mcl{L}_{2}:\mcl{N}_{1D\rightarrow 2D}\times L_2[x,\nu]\rightarrow \mcl{N}_{1D\rightarrow 2D}$ is defined such that, for arbitrary $N=\{N_0,N_1,N_2\}\in\mcl{N}_{1D\rightarrow 2D}$, $M\in L_2[y,\theta]$, $\mcl{L}_{2}(N,M)=\{Q_0,Q_1,Q_2\}$, where
	\begin{align*}
	\fvars{Q_{0}}{xy}{\theta}&=0	&
	\fvars{Q_{1}}{xy}{\theta}&=\fvars{Q_{2}}{xy}{\theta}=\int_{\mu=0}^{1}\bbbl(\sum_{i=0}^{2}\fvarss{\bs{\Phi}_{i}}{y}{\mu}\fvarss{N_i}{xy}{\mu}\fvars{M}{\mu}{\theta}\bbbr).
	\end{align*}
	
\end{lem}

\begin{proof}
	To prove this lemma, we will exploit the linear structure of 011-PI operators, allowing us to express
	\begin{align*}
	&\mcl{P}\smallbmat{E_0\\E_1\\E_2}\mcl{P}\smallbmat{B_{00}&B_{01}&B_{02}\\B_{10}&B_{11}&B_{12}\\B_{20}&B_{21}&B_{22}} = 
	\mcl{P}\smallbmat{E_0\\E_1\\E_2}\mcl{P}\smallbmat{B_{00}&0&0\\B_{10}&0&0\\B_{20}&0&0}\\
	&\qquad\qquad+ \mcl{P}\smallbmat{E_0\\E_1\\E_2}\mcl{P}\smallbmat{0&B_{01}&0\\0&B_{11}&0\\0&B_{21}&0}+ \mcl{P}\smallbmat{E_0\\E_1\\E_2}\mcl{P}\smallbmat{0&0&B_{02}\\0&0&B_{12}\\0&0&B_{22}}.
	\end{align*}
	Considering each of these terms separately, we find that for an arbitrary function $\mbf{u}=\smallbmat{u_0\\\mbf{u}_1\\\mbf{u}_2}\in RL^{m_0,m_1}$,
	\begin{align*}
	&\fvars{\left(\mcl{P}\smallbmat{E_0\\E_1\\E_2}\mcl{P}\smallbmat{B_{00}&()&()\\B_{10}&()&()\\B_{20}&()&()}u_0\right)}{xy}{} = \fvarss{E_{0}}{xy}{}B_{00}u_0 \\
	&\qquad + \int_{\eta=0}^{1}\bbbbl(\sum_{i=0}^{2}\bbl[\fvarss{\bs{\Phi}_i}{x}{\eta}\fvarss{E_{1}^{i}}{xy}{\eta}\fvarss{B_{10}}{\eta}{}u_0\bbr]\bbbbr)\\
	&\qquad + \int_{\mu=0}^{1}\bbbbl(\sum_{i=0}^{2}\bbl[\fvarss{\bs{\Phi}_i}{y}{\mu}\fvarss{E_{2}^{i}}{xy}{\mu}\fvars{B_{20}}{\mu}{}u_0\bbr]\bbbbr)\\
	&\hspace*{2.75cm}=\bbl(\fvarss{T_{0}}{xy}{}u_0\bbr)
	=\fvars{\left(\mcl{P}\smallbmat{T_{0}\\()\\()}u_0\right)}{xy}{}.
	\end{align*}
	Furthermore, using Corollary~\ref{cor_integral_Phi12_appendix} and Proposition~\ref{prop_Psi_appendix}, we find that
	\begin{align*}
	&\fvars{\left(\mcl{P}\smallbmat{E_0\\E_1\\E_2}\mcl{P}\smallbmat{0&B_{01}&0\\0&B_{11}&0\\0&B_{21}&0}\mbf{u}_1\right)}{xy}{} = \int_{\theta=0}^{1}\bbl(\fvarss{E_{0}}{xy}{}\fvarss{B_{01}}{}{\theta}\fvars{\mbf{u}_1}{\theta}{}\bbr) \\
	&+ \int_{\eta=0}^{1}\bbbbl(\sum_{i=0}^{2}\bbbl[\fvarss{\mbf{\Phi}_{i}}{x}{\eta}\fvars{E_{1}^{i}}{xy}{\eta}\int_{\theta=0}^{1}\bbbl(\sum_{j=0}^{2}\bbl[\fvarss{\bs{\Phi}_j}{\eta}{\theta}\fvarss{B_{11}^{j}}{\eta}{\theta}\fvars{\mbf{u}_1}{\theta}{}\bbr]\bbbl)\bbbr]\bbbbr)\\
	&+ \int_{\mu=0}^{1}\bbbbl(\sum_{i=0}^{2}\bbbl[\fvarss{\bs{\Phi}_i}{y}{\mu}\fvars{E_{2}^{i}}{xy}{\mu}\int_{\theta=0}^{1}\bbl(\fvarss{B_{21}}{\mu}{\theta}\fvars{\mbf{u}_1}{\theta}{}\bbr)\bbbr]\bbbbl)\\
	&=\int_{\theta=0}^{1}\bbbl(\sum_{k=1}^{2}\fvars{\bs{\Phi}_k}{x}{\theta}\bbl[\fvarss{E_{0}}{xy}{}\fvars{B_{01}}{}{\theta}\bbr]\fvars{\mbf{u}_1}{\theta}{}\bbr) \\
	&+ \int_{\theta=0}^{1}\bbbbl(\sum_{k=0}^{2}\fvars{\mbf{\Phi}_{k}}{x}{\theta}\int_{\eta=0}^{1}\bbbbl[\sum_{i,j=0}^{2}\fvarss{\mbf{\Psi}_{kij}}{x}{\eta\theta}\fvars{E_{1}^{i}}{xy}{\eta}\fvars{B_{11}^{j}}{\eta}{\theta}\bbbbr]\fvars{\mbf{u}_1}{\theta}{}\bbbbr)\\
	&+ \int_{\theta=0}^{1}\bbbbl(\sum_{k=1}^{2}\fvars{\bs{\Phi}_k}{x}{\theta}\int_{\mu=0}^{1}\bbbbl[\sum_{i=0}^{2}\fvarss{\bs{\Phi}_i}{y}{\mu}\fvarss{E_{2}}{xy}{\mu}\fvars{B_{21}^{i}}{\mu}{\theta}\bbbbl]\fvars{\mbf{u}_1}{\theta}{}\bbbbr)\\
	&=\fvars{\bl(\mcl{P}[\mcl{L}_{0}(E_{0},B_{01})]\mbf{u}_1\br)}{xy}{} +\fvars{\bl(\mcl{P}[\mcl{L}_{1}(E_{1},B_{11})]\mbf{u}_1\br)}{xy}{}	\\ &\hspace*{1.75cm}+\fvars{\bl(\mcl{P}[\mcl{L}_{2}(E_{2},B_{21})]\mbf{u}_1\br)}{xy}{} =\fvars{\left(\mcl{P}\smallbmat{()\\T_{1}\\()}\mbf{u}_1\right)}{xy}{}.
	\end{align*}
	Finally, performing the same steps
	\begin{align*}
	&\fvars{\left(\mcl{P}\smallbmat{E_0\\E_1\\E_2}\mcl{P}\smallbmat{0&0&B_{02}\\0&0&B_{12}\\0&0&B_{22}}\mbf{u}_2\right)}{xy}{} = \int_{\nu=0}^{1}\bbl(\fvarss{E_{0}}{xy}{}\fvarss{B_{02}}{}{\nu}\fvars{\mbf{u}_2}{\nu}{}\bbr) \\
	&+ \int_{\eta=0}^{1}\bbbbl(\sum_{i=0}^{2}\bbbl[\fvarss{\bs{\Phi}_i}{x}{\eta}\fvars{E_{1}^{i}}{xy}{\eta}\int_{\nu=0}^{1}\bbl(\fvarss{B_{12}}{\eta}{\nu}\fvars{\mbf{u}_2}{\nu}{}\bbr)\bbbr]\bbbbl)\\
	&+ \int_{\mu=0}^{1}\bbbbl(\sum_{i=0}^{2}\bbbl[\fvarss{\mbf{\Phi}_{i}}{y}{\mu}\fvars{E_{2}^{i}}{xy}{\mu}\int_{\nu=0}^{1}\bbbl(\sum_{j=0}^{2}\bbl[\fvarss{\bs{\Phi}_j}{\mu}{\nu}\fvarss{B_{22}^{j}}{\mu}{\nu}\fvars{\mbf{u}_2}{\nu}{}\bbr]\bbbl)\bbbr]\bbbbr)\\
	&=\int_{\nu=0}^{1}\bbbl(\sum_{k=1}^{2}\fvars{\bs{\Phi}_k}{y}{\nu}\bbl[\fvarss{E_{0}}{xy}{}\fvars{B_{02}}{}{\nu}\bbr]\fvars{\mbf{u}_2}{\nu}{}\bbr) \\
	&+ \int_{\nu=0}^{1}\bbbbl(\sum_{k=1}^{2}\fvars{\bs{\Phi}_k}{y}{\nu}\int_{\eta=0}^{1}\bbbbl[\sum_{i=0}^{2}\fvarss{\bs{\Phi}_i}{x}{\eta}\fvarss{E_{1}}{xy}{\eta}\fvars{B_{12}^{i}}{\eta}{\nu}\bbbbl]\fvars{\mbf{u}_2}{\nu}{}\bbbbr)\\
	&+ \int_{\nu=0}^{1}\bbbbl(\sum_{k=0}^{2}\fvars{\mbf{\Phi}_{k}}{y}{\nu}\int_{\mu=0}^{1}\bbbbl[\sum_{i,j=0}^{2}\fvarss{\mbf{\Psi}_{kij}}{y}{\mu\nu}\fvars{E_{2}^{i}}{xy}{\mu}\fvars{B_{22}^{j}}{\mu}{\nu}\bbbbr]\fvars{\mbf{u}_2}{\nu}{}\bbbbr)\\
	&=\fvars{\bl(\mcl{P}[\mcl{L}_{0}(E_{0},B_{02})]\mbf{u}_2\br)}{xy}{} +\fvars{\bl(\mcl{P}[\mcl{L}_{2}(E_{1},B_{12})]\mbf{u}_2\br)}{xy}{}	\\ &\hspace*{1.75cm}+\fvars{\bl(\mcl{P}[\mcl{L}_{1}(E_{2},B_{22})]\mbf{u}_2\br)}{xy}{} =\fvars{\left(\mcl{P}\smallbmat{()\\()\\T_{2}}\mbf{u}_2\right)}{xy}{}.
	\end{align*}
	Combining the results, we conclude that $\mcl{P}[E]\circ\mcl{P}[B]=\mcl{P}[T]$.	
	
\end{proof}

\begin{lem}\label{lem_2Dx1Dto2Dalgebra_appendix}
	For any $N=\smallbmat{N_{00}&N_{01}&N_{02}\\N_{10}&N_{11}&N_{12}\\N_{20}&N_{21}&N_{22}}\in\mcl{N}_{2D}^{n_2\times p_2}$ and $E=\smallbmat{E_{0}\\E_{1}\\E_{2}}\in\mcl{N}_{011\rightarrow 2D}\smallbmat{m_0\\m_1\\p_2}$, where
	\begin{align*}
	&E_{1}=\{E_{1}^{0},E_{1}^{1},E_{1}^{2}\}\in\mcl{N}_{1D\rightarrow 2D}^{p_2\times m_1},\\ 
	&E_{2}=\{E_{2}^{0},E_{2}^{1},E_{2}^{2}\}\in\mcl{N}_{1D\rightarrow 2D}^{p_2\times m_1},
	\end{align*}
	there exists a unique $T\in\mcl{N}_{011\rightarrow 2D}\smallbmat{n_0\\n_1\\m_2}$ such that $\mcl{P}[N]\circ\mcl{P}[E]=\mcl{P}[T]$. Specifically, we may choose $T=\mcl{L}_{011\rightarrow 2D}^{2}(N,E)\in\mcl{N}_{011\rightarrow 2D}\smallbmat{m_0\\m_1\\n_2}$, where the linear parameter map $\mcl{L}_{011\rightarrow 2D}^{2}:\mcl{N}_{2D}\times\mcl{N}_{011\rightarrow 2D}\rightarrow\mcl{N}_{011\rightarrow 2D}$ is defined such that
	\begin{align}\label{eq_composition_011to2D_2_appendix}
	\mcl{L}_{011\rightarrow 2D}^{2}(N,E)=\bmat{T_{0}\\T_{1}\\T_{2}}\in\mcl{N}_{011\rightarrow 2D}\smallbmat{m_0\\m_1\\n_2},
	\end{align}
	where
	\begin{align*}
	&\fvars{T_{0}}{xy}{}=\int_{\eta,\mu=0}^{1}\bbbbl(\sum_{i,p=0}^{2}\fvarss{\bs{\Phi}_i}{x}{\eta}\fvarss{\bs{\Phi}_p}{y}{\mu}\fvarss{N_{ip}}{xy}{\eta\mu}\fvarss{E_0}{\eta\mu}{}\bbbbr)
	\end{align*}
	and
	\begin{align*}
	&T_{1}=\{T_{1}^{0},T_{1}^{1},T_{1}^{2}\},	&
	&T_{2}=\{T_{2}^{0},T_{2}^{1},T_{2}^{2}\},	
	\end{align*}
	with
	\begin{align*}
	\fvars{T_{1}^{k}}{xy}{\theta}=\int_{\eta,\mu=0}^{1}\bbbbl(\sum_{i,j,p=0}^{2}\fvarss{\bs{\Psi}_{kij}}{x}{\eta\theta}\fvarss{\bs{\Phi}_p}{y}{\mu}\fvarss{N_{ip}}{xy}{\eta\mu}\fvarss{E_1^j}{\eta\mu}{\theta}\bbbbr)	\\
	\fvars{T_{2}^{r}}{xy}{\nu}=\int_{\eta,\mu=0}^{1}\bbbbl(\sum_{i,p,q=0}^{2}\fvarss{\bs{\Phi}_i}{x}{\eta}\fvarss{\bs{\Psi}_{rpq}}{y}{\mu\nu}\fvarss{N_{ip}}{xy}{\eta\mu}\fvarss{E_2^q}{\eta\mu}{\nu}\bbbbr),	\\
	\end{align*}
	for $k,r\in\{0,1,2\}$.
	
\end{lem}

\begin{proof}
	
	Applying Proposition~\ref{prop_Psi_appendix}, and the definitions of the operators $\mcl{P}[E]$ and $\mcl{P}[N]$, it follows that, for arbitrary $\mbf{u}=\smallbmat{u_0\\\mbf{u}_1\\\mbf{u}_2}\in RL^{m_0,m_1}$,
	\begin{align*}
	&\fvars{\left(\mcl{P}[N]\mcl{P}[E]\mbf{u}\right)}{xy}{}\\
	&= \int_{\eta,\mu=0}^{1}\bbbbl(\sum_{i,p=0}^{2}\fvarss{\bs{\Phi}_i}{x}{\eta}\fvarss{\bs{\Phi}_p}{y}{\mu}\fvarss{N_{ip}}{xy}{\eta\mu}\fvarss{E_0}{\eta\mu}{}u_0\bbbbr)\\
	&+\int_{\eta,\mu=0}^{1}\bbbbl(\sum_{i,p=0}^{2}\fvarss{\bs{\Phi}_i}{x}{\eta}\fvarss{\bs{\Phi}_p}{y}{\mu}\fvars{N_{ip}}{xy}{\eta\mu}\\
	&\hspace*{3.0cm} \int_{\theta=0}^{1}\bbbbl[\sum_{j=0}^{2}\fvarss{\bs{\Phi}_j}{\eta}{\theta}\fvarss{E_1^j}{\eta\mu}{\theta}\fvars{\mbf{u}_1}{\theta}{}\bbbbr]\bbbbr)\\
	&+\int_{\eta,\mu=0}^{1}\bbbbl(\sum_{i,p=0}^{2}\fvarss{\bs{\Phi}_i}{x}{\eta}\fvarss{\bs{\Phi}_p}{y}{\mu}\fvars{N_{ip}}{xy}{\eta\mu}\\
	&\hspace*{3.0cm} \int_{\nu=0}^{1}\bbbbl[\sum_{q=0}^{2}\fvarss{\bs{\Phi}_q}{\mu}{\nu}\fvarss{E_2^q}{\eta\mu}{\nu}\fvars{\mbf{u}_2}{\nu}{}\bbbbr]\bbbbr)\\
	&= \int_{\eta,\mu=0}^{1}\bbbbl(\sum_{i,p=0}^{2}\fvarss{\bs{\Phi}_i}{x}{\eta}\fvarss{\bs{\Phi}_p}{y}{\mu}\fvarss{N_{ip}}{xy}{\eta\mu}\fvarss{E_0}{\eta\mu}{}\bbbbr)u_0\\
	&+\int_{\theta=0}^{1}
	\bbbbl(\sum_{k=0}^{2}\fvars{\bs{\Phi}_k}{x}{\theta}\\
	&\qquad \int_{\eta,\mu=0}^{1}\bbbbl[\sum_{i,j,p=0}^{2}\fvarss{\bs{\Psi}_{kij}}{x}{\eta\theta}\fvarss{\bs{\Phi}_p}{y}{\mu}\fvarss{N_{ip}}{xy}{\eta\mu}\fvarss{E_1^j}{\eta\mu}{\theta}\bbbbr]\fvars{\mbf{u}_1}{\theta}{}\bbbbr)\\
	&+\int_{\nu=0}^{1}
	\bbbbl(\sum_{r=0}^{2}\fvars{\bs{\Phi}_r}{y}{\nu}\\
	&\qquad \int_{\eta,\mu=0}^{1}\bbbbl[\sum_{i,p,q=0}^{2}\fvarss{\bs{\Phi}_i}{x}{\eta}\fvarss{\bs{\Psi}_{rpq}}{y}{\mu\nu}\fvarss{N_{ip}}{xy}{\eta\mu}\fvarss{E_2^q}{\eta\mu}{\nu}\bbbbr]\fvars{\mbf{u}_2}{\nu}{}\bbbbr)\\
	&=\fvarss{T_0}{xy}{}u_0 + \int_{\theta=0}^{1}\bbbbl(\sum_{k=0}^{2}\fvarss{\bs{\Phi}_k}{x}{\theta}\fvarss{T_1^k}{xy}{\theta}\fvars{\mbf{u}_1}{\theta}{}\bbbbr)\\
	&\hspace*{1.0cm} + \int_{\nu=0}^{1}\bbbbl(\sum_{r=0}^{2}\fvarss{\bs{\Phi}_r}{y}{\nu}\fvarss{T_2^r}{xy}{\nu}\fvars{\mbf{u}_2}{\nu}{}\bbbbr)\\
	&=\fvars{\left(\mcl{P}\smallbmat{T_0\\()\\()}u_0\right)}{xy}{} + \fvars{\left(\mcl{P}\smallbmat{()\\T_1\\()}\mbf{u}_1\right)}{xy}{} + \fvars{\left(\mcl{P}\smallbmat{()\\()\\T_2}\mbf{u}_2\right)}{xy}{}\\
	&=\fvars{\bl(\mcl{P}[T]\mbf{u}\br)}{xy}{},
	\end{align*}
	as desired.
	
\end{proof}

In addition to these compositions with 011- and 2D-PI operators, the compositions of 011$\rightarrow$2D-PI and 2D$\rightarrow$011-PI can also be expressed as PI operators, as described in the following lemmas.

\begin{lem}\label{lem_2Dto1Dx1Dto2D algebra_appendix}
	For any $D=\smallbmat{D_{0}\\D_{1}\\D_{2}}\in\mcl{N}_{2D\rightarrow011}\smallbmat{n_0\\n_1\\p_2}$ and  $E=\smallbmat{E_{0}\\E_{1}\\E_{2}}\in\mcl{N}_{011\rightarrow 2D}\smallbmat{m_0\\m_1\\p_2}$, where
	\begin{align*} &D_{1}=\{D_{1}^{0},D_{1}^{1},D_{1}^{2}\}\in\mcl{N}_{2D\rightarrow 1D}^{n_1\times p_2},\\ 
	&D_{2}=\{D_{2}^{0},D_{2}^{1},D_{2}^{2}\}\in\mcl{N}_{2D\rightarrow 1D}^{n_1\times p_2},\\ 
	&E_{1}=\{E_{1}^{0},E_{1}^{1},E_{1}^{2}\}\in\mcl{N}_{1D\rightarrow 2D}^{p_2\times m_1},\\ 
	&E_{2}=\{E_{2}^{0},E_{2}^{1},E_{2}^{2}\}\in\mcl{N}_{1D\rightarrow 2D}^{p_2\times m_1},
	\end{align*}
	there exists a unique $R\in\mcl{N}_{011}\smallbmat{n_0&m_0\\n_1&m_1}$ such that $\mcl{P}[D]\circ\mcl{P}[E]=\mcl{P}[R]$. Specifically, we may choose $R=\mcl{L}_{011\rightarrow 011}(D,E)\in\mcl{N}_{011}\smallbmat{n_0&m_0\\n_1&m_1}$, where the linear parameter map $\mcl{L}_{011\rightarrow 011}:\mcl{N}_{2D\rightarrow 011}\times\mcl{N}_{011\rightarrow 2D}\rightarrow\mcl{N}_{011}$ is defined such that
	\begin{align}\label{eq_composition_011to011_2_appendix}
	\mcl{L}_{011\rightarrow 011}(D,E)=\bmat{R_{00}&R_{01}&R_{02}\\R_{10}&R_{11}&R_{12}\\R_{20}&R_{21}&R_{22}}\in\mcl{N}_{011}\smallbmat{n_0&m_0\\n_1&m_1},
	\end{align}
	where
	\begin{align*}
	&R_{00}=\int_{\eta,\nu=0}^{1}\bbl(\fvarss{D_{0}}{}{\eta\mu}\fvarss{E_{0}}{\eta\mu}{}\bbr)\\	
	&\fvars{R_{10}}{x}{}=\int_{\eta,\mu=0}^{1}\bbbbl(\sum_{i=0}^{k}\fvarss{\bs{\Phi}_i}{x}{\eta}\fvarss{D_1^{i}}{x}{\eta\mu}\fvarss{E_0}{\eta\mu}{}\bbbbr)\\
	&\fvars{R_{20}}{y}{}=\int_{\eta,\mu=0}^{1}\bbbbl(\sum_{i=0}^{k}\fvarss{\bs{\Phi}_i}{y}{\mu}\fvarss{D_2^{i}}{y}{\eta\mu}\fvarss{E_0}{\eta\mu}{}\bbbbr)\\
	&\fvars{R_{01}}{}{\theta}=\int_{\eta,\mu=0}^{1}\bbbbl(\sum_{i=0}^{2}\fvarss{\bs{\Phi}_i}{\eta}{\theta}\fvarss{D_0}{}{\eta\mu}\fvars{E_1^{i}}{\eta\mu}{\theta}\bbbbr)\\
	&\fvars{R_{21}}{y}{\theta}=\int_{\eta,\mu=0}^{1}\bbbbl(\sum_{i,j=0}^{2}\fvarss{\bs{\Phi}_i}{y}{\mu}\fvarss{\bs{\Phi}_j}{\eta}{\theta}\fvarss{D_2^{i}}{y}{\eta\mu}\fvarss{E_1^{j}}{\eta\mu}{\theta}\bbbbr)\\
	&\fvars{R_{02}}{}{\nu}=\int_{\eta,\mu=0}^{1}\bbbbl(\sum_{i=0}^{2}\fvarss{\bs{\Phi}_i}{\mu}{\nu}\fvarss{D_0}{}{\eta\mu}\fvars{E_2^{i}}{\eta\mu}{\nu}\bbbbr)	\\
	&\fvars{R_{12}}{x}{\nu}=\int_{\eta,\mu=0}^{1}\bbbbl(\sum_{i,j=0}^{2}\fvarss{\bs{\Phi}_i}{x}{\eta}\fvarss{\bs{\Phi}_j}{\mu}{\nu}\fvarss{D_1^{i}}{x}{\eta\mu}\fvarss{E_2^{j}}{\eta\mu}{\nu}\bbbbr)
	\end{align*}
	and where
	\begin{align*}
	 R_{11}&=\{R_{11}^{0},R_{11}^{1},R_{11}^{2}\},	&
	 R_{22}&=\{R_{22}^{0},R_{22}^{1},R_{22}^{2}\},
	\end{align*}
	with
	\begin{align*}
	&\fvars{R_{11}^{k}}{x}{\theta}=\int_{\eta,\mu=0}^{1}\bbbbl(\sum_{i,j=0}^{2}\fvarss{\bs{\Psi}_{kij}}{x}{\eta\theta}\fvarss{D_1^{i}}{x}{\eta\mu}\fvars{E_1^{j}}{\eta\mu}{\theta}\bbbbr) \\
	&\fvars{R_{22}^{k}}{y}{\nu}=\int_{\eta,\mu=0}^{1}\bbbbl(\sum_{i,j=0}^{2}\fvarss{\bs{\Psi}_{kij}}{y}{\mu\nu}\fvarss{D_2^{i}}{y}{\eta\mu}\fvars{E_2^{j}}{\eta\mu}{\nu}\bbbbr),
	\end{align*}
	for $k\in\{0,1,2\}$.
	
\end{lem}

\begin{proof}
	To prove this result, we exploit the linear structure of the PI operators, allowing us to decompose
	\begin{align*}
	\mcl{P}\smallbmat{D_0\\D_1\\D_2}\mcl{P}\smallbmat{E_0\\E_1\\E_2}&= \mcl{P}\smallbmat{D_0\\0\\0}\mcl{P}\smallbmat{E_0\\0\\0}+ \mcl{P}\smallbmat{0\\D_1\\0}\mcl{P}\smallbmat{E_0\\0\\0}\\
	&+ \hdots+ \mcl{P}\smallbmat{0\\D_1\\0}\mcl{P}\smallbmat{0\\0\\E_2}+ \mcl{P}\smallbmat{0\\0\\D_2}\mcl{P}\smallbmat{0\\0\\E_2}
	\end{align*}
	Focusing first on the terms involving $E_0$, we find that, for arbitrary $u_0\in\R^{m_0}$,
	\begin{align*}
	\mcl{P}\smallbmat{D_0\\()\\()}\mcl{P}\smallbmat{E_0\\()\\()}u_0&= \int_{\eta,\mu=0}^{1}\bbbl(\fvarss{D_0}{}{\eta\mu}\fvarss{E_0}{\eta\mu}{}u_0\bbbr)\\
	&=R_{00}u_0= \mcl{P}\smallbmat{R_{00}&()&()\\()&()&()\\()&()&()}u_0,
	\end{align*}
	and
	\begin{align*}
	\fvars{\left(\mcl{P}\smallbmat{()\\D_1\\()}\mcl{P}\smallbmat{E_0\\()\\()}u_0\right)}{x}{}\! 
	&=\! \int_{\eta,\mu=0}^{1}\! \bbbbl(\sum_{i=0}^{k}\! \fvarss{\bs{\Phi}_i}{x}{\eta}\fvarss{D_1^{i}}{x}{\eta\mu}\fvarss{E_0}{\eta\mu}{}\! u_0\bbbbr)\\
	&=\fvarss{R_{10}}{x}{}u_0=\fvars{\left(\mcl{P}\smallbmat{()&()&()\\R_{10}&()&()\\()&()&()}u_0\right)}{x}{},
	\end{align*}
	and finally,
	\begin{align*}
	\fvars{\left(\mcl{P}\smallbmat{()\\()\\D_2}\mcl{P}\smallbmat{E_0\\()\\()}u_0\right)}{y}{}\! 
	&=\! \int_{\eta,\mu=0}^{1}\! \bbbbl(\sum_{i=0}^{k}\! \fvarss{\bs{\Phi}_i}{y}{\mu}\fvarss{D_2^{i}}{y}{\eta\mu}\fvarss{E_0}{\eta\mu}{}\! u_0\bbbbr)\\
	&=\fvarss{R_{20}}{y}{}u_0=\fvars{\left(\mcl{P}\smallbmat{()&()&()\\()&()&()\\R_{20}&()&()}u_0\right)}{y}{}.
	\end{align*}
	Similarly, for the terms involving $E_1$, we find that, for arbitrary $\mbf{u}_1\in L_2^{m_1}[x]$,
	\begin{align*}
	&\mcl{P}\smallbmat{D_0\\()\\()}\mcl{P}\smallbmat{()\\E_1\\()}\mbf{u}_1\\
	&\quad= \int_{\eta,\mu=0}^{1}\bbbbl(\fvars{D_0}{}{\eta\mu}\int_{\theta=0}^{1}\bbbbl[\sum_{i=0}^{2}\fvarss{\bs{\Phi}_i}{\eta}{\theta}\fvarss{E_1^{i}}{\eta\mu}{\theta}\fvars{\mbf{u}_1}{\theta}{}\bbbbr]\bbbbr)\\
	&\qquad=\int_{\theta=0}^{1}\bbbbl(\int_{\eta,\mu=0}^{1}\bbbbl[\sum_{i=0}^{2}\fvarss{\bs{\Phi}_i}{\eta}{\theta}\fvarss{D_0}{}{\eta\mu}\fvars{E_1^{i}}{\eta\mu}{\theta}\bbbbr]\fvars{\mbf{u}_1}{\theta}{}\bbbbr)\\
	&\qquad\quad=\int_{\theta=0}^{1}\bbl(\fvarss{R_{01}}{}{\theta}\fvars{\mbf{u}_1}{\theta}{}\bbr)=\mcl{P}\smallbmat{()&R_{01}&()\\()&()&()\\()&()&()}\mbf{u}_1,
	\end{align*}
	and
	\begin{align*}
	&\fvars{\left(\mcl{P}\smallbmat{()\\()\\D_2}\mcl{P}\smallbmat{()\\E_1\\()}\mbf{u}_1\right)}{y}{}\\
	&= \int_{\eta,\mu=0}^{1}\! \bbbbl(\sum_{i=0}^{2}\fvarss{\bs{\Phi}_i}{y}{\mu}\fvars{D_2^{i}}{y}{\eta\mu}\int_{\theta=0}^{1}\! \bbbbl[\sum_{j=0}^{2}\fvarss{\bs{\Phi}_j}{\eta}{\theta}\fvarss{E_1^{j}}{\eta\mu}{\theta}\fvars{\mbf{u}_1}{\theta}{}\bbbbr]\bbbbr)\\
	&= \int_{\theta=0}^{1}\bbbbl(\int_{\eta,\mu=0}^{1}\bbbbl[\sum_{i,j=0}^{2}\fvarss{\bs{\Phi}_i}{y}{\mu}\fvarss{\bs{\Phi}_j}{\eta}{\theta}\fvarss{D_2^{i}}{y}{\eta\mu}\fvarss{E_1^{j}}{\eta\mu}{\theta}\bbbbr]\fvars{\mbf{u}_1}{\theta}{}\bbbbr)\\
	&=\int_{\theta=0}^{1}\bbl(\fvarss{R_{21}}{y}{\theta}\fvars{\mbf{u}_1}{\theta}{}\bbr)=\fvars{\left(\mcl{P}\smallbmat{()&()&()\\()&()&()\\()&R_{21}&()}\mbf{u}_1\right)}{y}{},
	\end{align*}
	and finally, using Proposition~\ref{prop_Psi_appendix},
	\begin{align*}
	&\fvars{\left(\mcl{P}\smallbmat{()\\D_1\\()}\mcl{P}\smallbmat{()\\E_1\\()}\mbf{u}_1\right)}{x}{}\\
	&= \int_{\eta,\mu=0}^{1}\! \bbbbl(\sum_{i=0}^{2}\fvarss{\bs{\Phi}_i}{x}{\eta}\fvars{D_1^{i}}{x}{\eta\mu}\int_{\theta=0}^{1}\! \bbbbl[\sum_{j=0}^{2}\fvarss{\bs{\Phi}_j}{\eta}{\theta}\fvarss{E_1^{j}}{\eta\mu}{\theta}\fvars{\mbf{u}_1}{\theta}{}\bbbbr]\bbbbr)\\
	&=\int_{\theta=0}^{1}\! \bbbbl(\sum_{k=0}^{2}\! \fvars{\bs{\Phi}_k}{x}{\theta}\int_{\eta,\mu=0}^{1}\! \bbbbl[\sum_{i,j=0}^{2}\! \fvarss{\bs{\Psi}_{kij}}{x}{\eta\theta}\fvarss{D_1^{i}}{x}{\eta\mu}\fvars{E_1^{j}}{\eta\mu}{\theta}\bbbbr]\fvars{\mbf{u}_1}{\theta}{}\bbbbr)\\
	&=\fvars{\left(\mcl{P}\smallbmat{()&()&()\\()&R_{11}&()\\()&()&()}\mbf{u}_1\right)}{x}{}.
	\end{align*}
	Finally, for the terms involving $E_2$, for arbitrary $\mbf{u}_2\in L_2^{m_1}[y]$,
	\begin{align*}
	&\mcl{P}\smallbmat{D_0\\()\\()}\mcl{P}\smallbmat{()\\()\\E_2}\mbf{u}_2\\
	&\quad= \int_{\eta,\mu=0}^{1}\bbbbl(\fvars{D_0}{}{\eta\mu}\int_{\nu=0}^{1}\bbbbl[\sum_{i=0}^{2}\fvarss{\bs{\Phi}_i}{\mu}{\nu}\fvarss{E_2^{i}}{\eta\mu}{\nu}\fvars{\mbf{u}_2}{\nu}{}\bbbbr]\bbbbr)\\
	&\qquad=\int_{\nu=0}^{1}\bbbbl(\int_{\eta,\mu=0}^{1}\bbbbl[\sum_{i=0}^{2}\fvarss{\bs{\Phi}_i}{\mu}{\nu}\fvarss{D_0}{}{\eta\mu}\fvars{E_2^{i}}{\eta\mu}{\nu}\bbbbr]\fvars{\mbf{u}_2}{\nu}{}\bbbbr)\\
	&\qquad\quad=\int_{\nu=0}^{1}\bbl(\fvarss{R_{02}}{}{\nu}\fvars{\mbf{u}_2}{\nu}{}\bbr)=\mcl{P}\smallbmat{()&()&R_{02}\\()&()&()\\()&()&()}\mbf{u}_2,
	\end{align*}
	and
	\begin{align*}
	&\fvars{\left(\mcl{P}\smallbmat{()\\D_1\\()}\mcl{P}\smallbmat{()\\()\\E_2}\mbf{u}_2\right)}{x}{}\\
	&= \int_{\eta,\mu=0}^{1}\! \bbbbl(\sum_{i=0}^{2}\fvarss{\bs{\Phi}_i}{x}{\eta}\fvars{D_1^{i}}{x}{\eta\mu}\int_{\nu=0}^{1}\! \bbbbl[\sum_{j=0}^{2}\fvarss{\bs{\Phi}_j}{\mu}{\nu}\fvarss{E_2^{j}}{\eta\mu}{\nu}\fvars{\mbf{u}_2}{\nu}{}\bbbbr]\bbbbr)\\
	&= \int_{\nu=0}^{1}\bbbbl(\int_{\eta,\mu=0}^{1}\bbbbl[\sum_{i,j=0}^{2}\fvarss{\bs{\Phi}_i}{x}{\eta}\fvarss{\bs{\Phi}_j}{\mu}{\nu}\fvarss{D_1^{i}}{x}{\eta\mu}\fvarss{E_2^{j}}{\eta\mu}{\nu}\bbbbr]\fvars{\mbf{u}_2}{\nu}{}\bbbbr)\\
	&=\int_{\nu=0}^{1} \bbl(\fvarss{R_{12}}{x}{\nu}\fvars{\mbf{u}_2}{\nu}{}\bbr)=\fvars{\left(\mcl{P}\smallbmat{()&()&()\\()&()&R_{12}\\()&()&()}\mbf{u}_2\right)}{x}{},
	\end{align*}
	and, once more using Proposition~\ref{prop_Psi_appendix},
	\begin{align*}
	&\fvars{\left(\mcl{P}\smallbmat{()\\()\\D_2}\mcl{P}\smallbmat{()\\()\\E_2}\mbf{u}_2\right)}{y}{}\\
	&= \int_{\eta,\mu=0}^{1}\! \bbbbl(\sum_{i=0}^{2}\fvarss{\bs{\Phi}_i}{y}{\mu}\fvars{D_2^{i}}{y}{\eta\mu}\int_{\nu=0}^{1}\! \bbbbl[\sum_{j=0}^{2}\fvarss{\bs{\Phi}_j}{\mu}{\nu}\fvarss{E_2^{j}}{\eta\mu}{\nu}\fvars{\mbf{u}_2}{\nu}{}\bbbbr]\bbbbr)\\
	&=\int_{\nu=0}^{1}\! \bbbbl(\sum_{k=0}^{2}\! \fvars{\bs{\Phi}_k}{y}{\nu}\int_{\eta,\mu=0}^{1}\! \bbbbl[\sum_{i,j=0}^{2}\! \fvarss{\bs{\Psi}_{kij}}{y}{\mu\nu}\fvarss{D_2^{i}}{y}{\eta\mu} \fvars{E_2^{j}}{\eta\mu}{\nu}\bbbbr]\fvars{\mbf{u}_2}{\nu}{}\bbbbr)\\
	&=\fvars{\left(\mcl{P}\smallbmat{()&()&()\\()&()&()\\()&()&R_{22}}\mbf{u}_2\right)}{y}{}.
	\end{align*}
	Combining the results, we conclude that $\mcl{P}[D]\circ\mcl{P}[E]=\mcl{P}[R]$.
	
\end{proof}

\begin{table*}[!ht]
	\centering
	\begin{tabular}{l|l|l}
		Parameter space & Explicit notation & Associated PI operation	\\\hline
		$\mcl{N}_{1D}^{n\times m}$ 
		& $L_2^{n\times m}[x,\theta] \times L_2^{n\times m}[x] \times L_2^{n\times m}[x,\theta] $  
		& $\! \! \fvars{\bl(\mcl{P}[N]\mbf{u}\br)}{x}{}=\int_{\theta=0}^{1}\bbbl(\sum_{i=0}^{2}\bbl[\fvarss{\bs{\Phi}_{i}}{x}{\theta}\fvars{N_i}{x}{\theta}\bbl]\fvars{\mbf{u}}{\theta}{}\bbbr)$
		\\\hline
		$\mcl{N}_{011}\smallbmat{n_0&m_0\\n_1&m_1}$ 
		& $\! \! \bmat{\R^{n_0\times m_0}   &L_2^{n_0\times m_1}[\theta] &L_2^{n_0\times m_1}[\nu]\\ L_2^{n_1\times m_0}[x] &\mcl{N}_{1D}^{n_1\times m_1}   &L_2^{n_1\times m_1}[x,\nu]\\			L_2^{n_1\times m_0}[y] &L_2^{n_1\times m_1}[y,\theta]   &\mcl{N}_{1D}^{n_1\times m_1}}$	 
		&	$\! \! \fvars{\bl(\mcl{P}[N]\mbf{u}\br)}{xy}{}
		=\! \bmat{N_{00}u_0 
			&\! \! \int_{\theta=0}^{1}\bbl(\fvarss{N_{01}}{}{\theta}\fvars{\mbf{u}_1}{\theta}{}\bbr)
			&\! \! \int_{\nu=0}^{1}\bbl(\fvarss{N_{02}}{}{\nu}\fvars{\mbf{u}_2}{\nu}{}\bbr)
			\\\fvarss{N_{10}}{x}{}u_0
			&\! \! \fvars{\bl(\mcl{P}[N_{11}]\mbf{u}_1\br)}{x}{}
			&\! \! \int_{\nu=0}^{1}\bbl(\fvarss{N_{12}}{x}{\nu}\fvars{\mbf{u}_2}{\nu}{}\bbr)
			\\\fvarss{N_{20}}{y}{}u_0
			&\! \! \int_{\theta=0}^{1}\bbl(\fvarss{N_{21}}{y}{\theta}\fvars{\mbf{u}_1}{\theta}{}\bbr)
			&\! \! \fvars{\bl(\mcl{P}[N_{22}]\mbf{u}_2\br)}{y}{}}$
		\\\hline
		$\mcl{N}_{2D}^{n\times m}$ 
		& $\! \! \bmat{L_{2}^{n \times m}[x,y]&\!  L_{2}^{n \times m}[x,y,\nu]    &\!  L_{2}^{n \times m}[x,y,\nu]\\ L_{2}^{n \times m}[x,y,\theta]   &\!  L_2^{n\times m}[x,y,\theta,\nu]&\!  L_2^{n\times m}[x,y,\theta,\nu]\\ L_{2}^{n \times m}[x,y,\theta]&\!  L_2^{n\times m}[x,y,\theta,\nu]&\! L_2^{n\times m}[x,y,\theta,\nu]}$ 
		& $\! \! \fvars{\bl(\mcl{P}[N]\mbf{u}\br)}{xy}{}=\int_{\theta,\nu=0}^{1}\bbbl(\sum_{i,p=0}^{2}\bbl[\fvarss{\bs{\Phi}_{i}}{x}{\theta}\fvarss{\bs{\Phi}_{p}}{y}{\nu}\fvars{N_{ip}}{xy}{\theta\nu}\bbl]\fvars{\mbf{u}}{\theta\nu}{}\bbbr)$
		\\\hline
		$\mcl N_{2D\rightarrow 1D}^{n\times m}$ 
		& $L_2^{n \times m}[\theta,\nu] \times L_2^{n \times m}[x,\theta,\nu] \times L_2^{n \times m}[x,\theta,\nu]$
		& $\! \! \fvars{\bl(\mcl{P}[N]\mbf{u}\br)}{x}{}=\int_{\theta,\nu=0}^{1}\bbbl(\sum_{i=0}^{2}\bbl[\fvarss{\bs{\Phi}_{i}}{x}{\theta}\fvars{N_i}{x}{\theta\nu}\bbl]\fvars{\mbf{u}}{\theta\nu}{}\bbbr)$
		\\\hline
		$\mcl N_{2D\rightarrow 011}\smallbmat{n_0\\n_1\\m_2}$
		& $\! \! \bmat{L_2^{n_0\times m_2}[\theta,\nu]\\\mcl{N}^{n_1\times m_2}_{2D\rightarrow 1D}\\\mcl{N}^{n_1\times m_2}_{2D\rightarrow 1D}}$
		&	$\! \! \fvars{\bl(\mcl{P}[N]\mbf{u}\br)}{xy}{}
		=\bmat{
			\int_{\theta,\nu=0}^{1}\bbl(\fvarss{N_{0}}{}{\theta\nu}\fvars{\mbf{u}}{\theta\nu}{}\bbr)
			\\\fvars{\bl(\mcl{P}[N_{1}]\mbf{u}\br)}{x}{}
			\\\fvars{\bl(\mcl{P}[N_{2}]\mbf{u}\br)}{y}{}}$
		\\\hline
		$\mcl N_{1D\rightarrow 2D}^{n\times m}$ 
		& $L_2^{n \times m}[x,y] \times L_2^{n \times m}[x,y,\theta] \times L_2^{n \times m}[x,y,\theta]$ 
		& $\! \! \fvars{\bl(\mcl{P}[N]\mbf{u}\br)}{xy}{}=\int_{\theta=0}^{1}\bbbl(\sum_{i=0}^{2}\bbl[\fvarss{\bs{\Phi}_{i}}{x}{\theta}\fvars{N_i}{xy}{\theta}\bbl]\fvars{\mbf{u}}{\theta}{}\bbbr)$
		\\\hline
		$\mcl N_{011\rightarrow 2D}\smallbmat{m_0\\m_1\\n_2}$
		& $\! \! \bmat{L_2^{n_2\times m_0}[x,y]\\\mcl{N}^{n_2\times m_1}_{1D\rightarrow 2D}\\\mcl{N}^{n_2\times m_1}_{1D\rightarrow 2D}}$
		& $\! \! \fvars{\bl(\mcl{P}[N]\mbf{u}\br)}{xy}{}=\fvars{N_0}{xy}{}u_0 + \fvars{\bl(\mcl{P}[N_1]\mbf{u}_1\br)}{xy}{} + \fvars{\bl(\mcl{P}[N_2]\mbf{u}_2\br)}{xy}{}$
		\\\hline
		$\mcl{N}_{0112}{\smallbmat{n_0&m_0\\n_1&m_1\\n_2&m_2}}$ 
		& $\! \! \bmat{\mcl{N}_{011}{\smallbmat{n_0&m_0\\n_1&m_1}} & \mcl{N}_{2D\rightarrow 011}\smallbmat{n_0\\n_1\\m_2} \\			\mcl{N}_{011\rightarrow 2D}\smallbmat{m_0\\m_1\\n_2} & \mcl{N}_{2D}^{n_2\times m_2}}$
		&	$\! \! \fvars{\bl(\mcl{P}[N]\mbf{u}\br)}{xy}{}
		=\bmat{\fvars{\bl(\mcl{P}[N_{11}]\mbf{u}_1\br)}{xy}{}
			&\fvars{\bl(\mcl{P}[N_{12}]\mbf{u}_2\br)}{xy}{}
			\\\fvars{\bl(\mcl{P}[N_{21}]\mbf{u}_1\br)}{xy}{}
			&\fvars{\bl(\mcl{P}[N_{22}]\mbf{u}_2\br)}{xy}{}}$
	\end{tabular}
	\caption{Parameter spaces for different PI operators, as presented in Appendix~\ref{sec_PI_algebras_appendix}}
	\label{tab_PI_param_spaces_appendix}
\end{table*}

\begin{lem}\label{lem_1Dto2Dx2Dto1D algebra_appendix}
	For any $E=\smallbmat{E_{0}\\E_{1}\\E_{2}}\in\mcl{N}_{011\rightarrow 2D}\smallbmat{p_0\\p_1\\n_2}$ and $D=\smallbmat{D_{0}\\D_{1}\\D_{2}}\in\mcl{N}_{2D\rightarrow011}\smallbmat{p_0\\p_1\\m_2}$, where
	\begin{align*} 
	&E_{1}=\{E_{1}^{0},E_{1}^{1},E_{1}^{2}\}\in\mcl{N}_{1D\rightarrow 2D}^{n_2\times p_1},\\ 
	&E_{2}=\{E_{2}^{0},E_{2}^{1},E_{2}^{2}\}\in\mcl{N}_{1D\rightarrow 2D}^{n_2\times p_1},\\
	&D_{1}=\{D_{1}^{0},D_{1}^{1},D_{1}^{2}\}\in\mcl{N}_{2D\rightarrow 1D}^{p_1\times m_2},\\ 
	&D_{2}=\{D_{2}^{0},D_{2}^{1},D_{2}^{2}\}\in\mcl{N}_{2D\rightarrow 1D}^{p_1\times m_2},
	\end{align*}
	there exists a unique $Q\in\mcl{N}_{2D}^{n_2\times m_2}$ such that $\mcl{P}[D]\circ\mcl{P}[E]=\mcl{P}[Q]$. Specifically, we may choose $Q=\mcl{L}_{2D\rightarrow 2D}(E,D)\in\mcl{N}_{2D}^{n_2\times m_2}$, where the linear parameter map $\mcl{L}_{2D\rightarrow 2D}:\mcl{N}_{011\rightarrow 2D}\times\mcl{N}_{2D\rightarrow 011}\rightarrow\mcl{N}_{2D}$ is defined such that
	\begin{align}\label{eq_composition_2Dto2D_2_appendix}
	\mcl{L}_{2D\rightarrow 2D}(E,D)=\bmat{Q_{00}&Q_{01}&Q_{02}\\Q_{10}&Q_{11}&Q_{12}\\Q_{20}&Q_{21}&Q_{22}}\in\mcl{N}_{2D}^{n_2\times m_2},
	\end{align}
	where $\fvars{Q_{00}}{xy}{\theta\nu}=0$, and	\\
	\vspace{-1.0cm}
	\begin{align*}
	\fvars{Q_{10}}{xy}{\theta\nu}&=\int_{\mu=0}^{1} \bbbbl(\sum_{p,q=0}^{2}\fvarss{\bs{\Psi}_{0pq}}{y}{\mu\nu}\fvarss{E_2^{p}}{xy}{\mu}\fvarss{D_2^{q}}{\mu}{\theta\nu}\bbbbr)\\
	\fvars{Q_{20}}{xy}{\theta\nu}&=\int_{\mu=0}^{1} \bbbbl(\sum_{p,q=0}^{2}\fvarss{\bs{\Psi}_{0pq}}{y}{\mu\nu}\fvarss{E_2^{p}}{xy}{\mu}\fvarss{D_2^{q}}{\mu}{\theta\nu}\bbbbr)\\
	\fvars{Q_{01}}{xy}{\theta\nu}&=\int_{\eta=0}^{1} \bbbbl(\sum_{i,j=0}^{2}\fvarss{\bs{\Psi}_{0ij}}{x}{\eta\theta}\fvarss{E_1^{i}}{xy}{\eta}\fvarss{D_1^{j}}{\eta}{\theta\nu}\bbbbr)\\
	\fvars{Q_{02}}{xy}{\theta\nu}&=\int_{\eta=0}^{1} \bbbbl(\sum_{i,j=0}^{2}\fvarss{\bs{\Psi}_{0ij}}{x}{\eta\theta}\fvarss{E_1^{i}}{xy}{\eta}\fvarss{D_1^{j}}{\eta}{\theta\nu}\bbbbr)\\
	\fvars{Q_{11}}{xy}{\theta\nu}&=\fvarss{E_0}{xy}{}\fvarss{D_0}{}{\theta\nu}\! 
	+\! \int_{\eta=0}^{1}\!  \bbbbl(\sum_{i,j=0}^{2}\! \fvarss{\bs{\Psi}_{1ij}}{x}{\eta\theta}\fvarss{E_1^{i}}{xy}{\eta}\fvarss{D_1^{j}}{\eta}{\theta\nu}\bbbbr)\\
	&\qquad+\int_{\mu=0}^{1}\!  \bbbbl(\sum_{p,q=0}^{2}\! \fvarss{\bs{\Psi}_{1pq}}{y}{\mu\nu}\fvarss{E_2^{p}}{xy}{\mu}\fvarss{D_2^{q}}{\mu}{\theta\nu}\bbbbr)\\
	\fvars{Q_{21}}{xy}{\theta\nu}&=\fvarss{E_0}{xy}{}\fvarss{D_0}{}{\theta\nu}\! 
	+\! \int_{\eta=0}^{1}\!  \bbbbl(\sum_{i,j=0}^{2}\! \fvarss{\bs{\Psi}_{2ij}}{x}{\eta\theta}\fvarss{E_1^{i}}{xy}{\eta}\fvarss{D_1^{j}}{\eta}{\theta\nu}\bbbbr)\\
	&\qquad+\int_{\mu=0}^{1}\!  \bbbbl(\sum_{p,q=0}^{2}\! \fvarss{\bs{\Psi}_{1pq}}{y}{\mu\nu}\fvarss{E_2^{p}}{xy}{\mu}\fvarss{D_2^{q}}{\mu}{\theta\nu}\bbbbr)\\
	\fvars{Q_{12}}{xy}{\theta\nu}&=\fvarss{E_0}{xy}{}\fvarss{D_0}{}{\theta\nu}\! 
	+\! \int_{\eta=0}^{1}\!  \bbbbl(\sum_{i,j=0}^{2}\! \fvarss{\bs{\Psi}_{1ij}}{x}{\eta\theta}\fvarss{E_1^{i}}{xy}{\eta}\fvarss{D_1^{j}}{\eta}{\theta\nu}\bbbbr)\\
	&\qquad+\int_{\mu=0}^{1}\!  \bbbbl(\sum_{p,q=0}^{2}\! \fvarss{\bs{\Psi}_{2pq}}{y}{\mu\nu}\fvarss{E_2^{p}}{xy}{\mu}\fvarss{D_2^{q}}{\mu}{\theta\nu}\bbbbr)\\
	\fvars{Q_{22}}{xy}{\theta\nu}&=\fvarss{E_0}{xy}{}\fvarss{D_0}{}{\theta\nu}\! 
	+\! \int_{\eta=0}^{1}\!  \bbbbl(\sum_{i,j=0}^{2}\! \fvarss{\bs{\Psi}_{2ij}}{x}{\eta\theta}\fvarss{E_1^{i}}{xy}{\eta}\fvarss{D_1^{j}}{\eta}{\theta\nu}\bbbbr)\\
	&\qquad+\int_{\mu=0}^{1}\!  \bbbbl(\sum_{p,q=0}^{2}\! \fvarss{\bs{\Psi}_{2pq}}{y}{\mu\nu}\fvarss{E_2^{p}}{xy}{\mu}\fvarss{D_2^{q}}{\mu}{\theta\nu}\bbbbr)
	\end{align*}
	
\end{lem}

\begin{proof}
	Applying Corollary~\ref{cor_integral_Phi12_appendix} and Proposition~\ref{prop_Psi_appendix}, and invoking the definitions of the operators $\mcl{P}[E]$ and $\mcl{P}[D]$, for arbitrary $\mbf{u}\in L_2[x,y]$,
	\begin{align*}
	&\fvars{\left(\mcl{P}[E]\mcl{P}[D]\mbf{u}\right)}{xy}{}= \fvars{E_0}{xy}{}\int_{\theta,\nu=0}^{1}\bbl(\fvarss{D_0}{}{\theta\nu}\fvars{\mbf{u}}{\theta\nu}{}\bbr)\\
	&+\int_{\eta=0}^{1}\bbbbl(\sum_{i=0}^{2}\fvarss{\bs{\Phi}_i}{x}{\eta}\fvars{E_1^{i}}{xy}{\eta} \int_{\theta,\nu=0}^{1}\bbbbl[\sum_{j=0}^{2}\fvarss{\bs{\Phi}_j}{\eta}{\theta}\fvarss{D_1^{j}}{\eta}{\theta\nu}\fvars{\mbf{u}}{\theta\nu}{}\bbbbr]\bbbbr) \\
	&+\int_{\mu=0}^{1}\bbbbl(\sum_{p=0}^{2}\fvarss{\bs{\Phi}_p}{y}{\mu}\fvars{E_2^{p}}{xy}{\mu} \int_{\theta,\nu=0}^{1}\bbbbl[\sum_{q=0}^{2}\fvarss{\bs{\Phi}_q}{\mu}{\nu}\fvarss{D_2^{q}}{\mu}{\theta\nu}\fvars{\mbf{u}}{\theta\nu}{}\bbbbr]\bbbbr)	\\
	&=\int_{\theta,\nu=0}^{1}\bbbbl(\sum_{k,r=1}^{2}\fvarss{\bs{\Phi}_k}{x}{\theta}\fvarss{\bs{\Phi}_r}{y}{\nu}\fvarss{E_0}{xy}{}\fvarss{D_0}{}{\theta\nu}\fvars{\mbf{u}}{\theta\nu}{}\bbbbr)\\
	&+\int_{\theta,\nu=0}^{1}\bbbbl(\sum_{k=0}^{2}\sum_{r=1}^{2}\fvarss{\bs{\Phi}_k}{x}{\theta}\fvarss{\bs{\Phi}_r}{y}{\nu}\\
	&\hspace*{2.5cm}\int_{\eta=0}^{1} \bbbbl[\sum_{i,j=0}^{2}\fvarss{\bs{\Psi}_{kij}}{x}{\eta\theta}\fvarss{E_1^{i}}{xy}{\eta}\fvarss{D_1^{j}}{\eta}{\theta\nu}\bbbbr]\fvars{\mbf{u}}{\theta\nu}{}\bbbbr) \\
	&+\int_{\theta,\nu=0}^{1}\bbbbl(\sum_{k=1}^{2}\sum_{r=0}^{2}\fvarss{\bs{\Phi}_k}{x}{\theta}\fvarss{\bs{\Phi}_r}{y}{\nu}\\
	&\hspace*{2.5cm}\int_{\mu=0}^{1} \bbbbl[\sum_{p,q=0}^{2}\fvarss{\bs{\Psi}_{rpq}}{y}{\mu\nu}\fvarss{E_2^{p}}{xy}{\mu}\fvarss{D_2^{q}}{\mu}{\theta\nu}\bbbbr]\fvars{\mbf{u}}{\theta\nu}{}\bbbbr)	\\
	&=\int_{\theta,\nu=0}^{1}\bbbbl(\sum_{k,r=0}^{2}\fvarss{\bs{\Phi}_k}{x}{\theta}\fvarss{\bs{\Phi}_r}{y}{\nu}\fvarss{Q_{kr}}{xy}{\theta\nu}\fvars{\mbf{u}_2}{\theta\nu}{}\bbbbr)
	=\fvars{\bl(\mcl{P}[Q]\mbf{u}\br)}{xy},
	\end{align*}
	as desired.
	
\end{proof}

\begin{table*}[!ht]
	\renewcommand{\arraystretch}{1.4}
	\centering
	\begin{tabular}{c|c|c}
		Linear Parameter Map	\smallskip	& Associated Parameter Spaces	& Defined in \\\hline
		$\mcl{L}_{011}$&	$\mcl{N}_{011}\times\mcl{N}_{011}\rightarrow\mcl{N}_{011}$	& Equation~\eqref{eq_composition_011to011_1_appendix}, Appendix~\ref{sec_011algebra_appendix}\\
		$\mcl{L}_{2D}$	&	$\mcl{N}_{2D}\times\mcl{N}_{2D}\rightarrow\mcl{N}_{2D}$	& Equation~\eqref{eq_composition_2Dto2D_1_appendix}, Appendix~\ref{sec_2Dalgebra_appendix}\\
		$\mcl{L}_{2D\rightarrow 011}^{1}$	&	$\mcl{N}_{011}\times\mcl{N}_{2D\rightarrow 011}\rightarrow\mcl{N}_{2D\rightarrow 011}$	& Equation~\eqref{eq_composition_2Dto011_1_appendix}, Appendix~\ref{sec_2Dto011algebra_appendix}\\
		$\mcl{L}_{2D\rightarrow 011}^{2}$	&	$\mcl{N}_{2D\rightarrow 011}\times\mcl{N}_{2D}\rightarrow\mcl{N}_{2D\rightarrow 011}$	& Equation~\eqref{eq_composition_2Dto011_2_appendix}, Appendix~\ref{sec_2Dto011algebra_appendix}\\
		$\mcl{L}_{011\rightarrow 2D}^{1}$	&	$\mcl{N}_{011\rightarrow 2D}\times\mcl{N}_{011}\rightarrow\mcl{N}_{011\rightarrow 2D}$	& Equation~\eqref{eq_composition_011to2D_1_appendix}, Appendix~\ref{sec_011to2Dalgebra_appendix}\\
		$\mcl{L}_{011\rightarrow 2D}^{2}$	&	$\mcl{N}_{2D}\times\mcl{N}_{011\rightarrow 2D}\rightarrow\mcl{N}_{011\rightarrow 2D}$	& Equation~\eqref{eq_composition_011to2D_2_appendix}, Appendix~\ref{sec_011to2Dalgebra_appendix}\\
		$\mcl{L}_{011\rightarrow 011}$	&	$\mcl{N}_{2D\rightarrow 011}\times\mcl{N}_{011\rightarrow 2D}\rightarrow\mcl{N}_{011}$	& Equation~\eqref{eq_composition_011to011_2_appendix}, Appendix~\ref{sec_011to2Dalgebra_appendix}\\
		$\mcl{L}_{2D\rightarrow 2D}$	&	$\mcl{N}_{011\rightarrow 2D}\times\mcl{N}_{2D\rightarrow 011}\rightarrow\mcl{N}_{2D}$	& Equation~\eqref{eq_composition_2Dto2D_2_appendix}, Appendix~\ref{sec_011to2Dalgebra_appendix}\\	
	\end{tabular}
	\caption{Linear parameter maps for compositions of PI operators, as presented in Appendix~\ref{sec_PI_algebras_appendix}}
	\label{tab_composition_maps_appendix}
\end{table*}

\subsection{An Algebra of 0112-PI Operators}\label{sec_0112algebra_appendix}

Having described PI operators mapping functions in $RL^{n_0,n_1}[x,y]:=\R\times L_2[x]\times L_2[y]$ and $L_2^{n_2}[x,y]$ and the corresponding composition rules, we can now combine our results to describe a 0112-PI operator, mapping functions in $RLL^{n_0,n_1,n_2}[x,y]:=RL^{n_0,n_1}[x,y]\times L_2^{n_2}[x,y]$. Letting
\begin{align}\label{eq_N0112_space_appendix}
 \mcl{N}_{0112}\smallbmat{n_0&m_0\\n_1&m_1\\n_2&m_2}:= \bmat{\mcl{N}_{011}\smallbmat{n_0&m_0\\n_1&m_1}& \mcl{N}_{2D\rightarrow 011}\smallbmat{n_0\\n_1\\m_2}\\ \mcl{N}_{011\rightarrow 2D}\smallbmat{m_0\\m_1\\n_2}& \mcl{N}_{2D}^{n_2\times m_2}},
\end{align}
for arbitrary
\begin{align*}
 N=\bmat{N_{11}&N_{12}\\N_{21}&N_{22}}\in\mcl{N}\smallbmat{n_0&m_0\\n_1&m_1\\n_2&m_2},
\end{align*}
we define an associated PI operator $\mcl{P}[N]:RLL^{m_0,m_1,m_2}[x,y]\rightarrow RLL^{n_0,n_1,n_2}[x,y]$ as
\begin{align*}
 \mcl{P}[N]=\bmat{\mcl{P}[N_{11}]&\mcl{P}[N_{12}]\\\mcl{P}[N_{21}]&\mcl{P}[N_{22}]},
\end{align*}
where the different operators $\mcl{P}[N_11]$,...,$\mcl{P}[N_{22}]$ are as defined in the previous sections. Using the results from these sections, it is also easy to see that the set of operators parameterized in this manner also forms an algebra.

\begin{thm}\label{thm_composition0112_appendix}
	For any $B=\smallbmat{B_{11}&B_{12}\\B_{21}&B_{22}}\in\mcl{N}_{0112}\smallbmat{n_0&p_0\\n_1&p_1\\n_2&p_2}$ and $D=\smallbmat{D_{11}&D_{12}\\D_{21}&D_{22}}\in\mcl{N}_{0112}\smallbmat{p_0&m_0\\p_1&m_1\\p_2&m_2}$, there exists a unique $R\in\mcl{N}_{0112}\smallbmat{n_0&m_0\\n_1&m_1\\n_2&m_2}$ such that $\mcl{P}[B]\circ\mcl{P}[D]=\mcl{P}[R]$. Specifically, we may choose $R=\mcl{L}_{0112}(B,D)\in\mcl{N}_{0112}\smallbmat{n_0&m_0\\n_1&m_1\\n_2&m_2}$, where the linear parameter map $\mcl{L}_{0112}:\mcl{N}_{0112}\times\mcl{N}_{0112}\rightarrow\mcl{N}_{0112}$ defined such that
	\begin{align}\label{eq_composition_0112_appendix}
	 \mcl{L}_{0112}(B,D)=\bmat{R_{11}&R_{12}\\R_{21}&R_{22}}\in\mcl{N}_{0112}\smallbmat{n_0&m_0\\n_1&m_1\\n_2&m_2},
	\end{align}
	where
	\begin{align*}
	 R_{11}&=\mcl{L}_{011}(B_{11},D_{11})+\mcl{L}_{011\rightarrow 011}(B_{12},D_{21}),\\
	 R_{21}&=\mcl{L}_{011\rightarrow 2D}^{1}(B_{21},D_{11})+\mcl{L}_{011\rightarrow 2D}^{2}(B_{22},D_{21}),\\
	 R_{12}&=\mcl{L}_{2D\rightarrow 011}^{1}(B_{11},D_{12})+\mcl{L}_{2D\rightarrow 011}^{2}(B_{12},D_{22}),\\
	 R_{22}&=\mcl{L}_{2D\rightarrow 2D}(B_{21},D_{12})+\mcl{L}_{2D}(B_{22},D_{22}),
	\end{align*}
	where the different parameter maps $\mcl{L}$ are listed in Table~\ref{tab_composition_maps_appendix}.
\end{thm}
\begin{proof}
	Exploiting the linear structure of the 0112-PI operator, and applying the results from the previous sections, it follows that
	\begin{align*}
	 &\mcl{P}[B]\circ\mcl{P}[D]= \mcl{P}\smallbmat{B_{11}&B_{12}\\B_{21}&B_{22}}\circ \mcl{P}\smallbmat{D_{11}&D_{12}\\D_{21}&D_{22}}	\\
	 &\qquad=\bmat{\mcl{P}[B_{11}]&\mcl{P}[B_{12}]\\\mcl{P}[B_{21}]&\mcl{P}[B_{22}]} \bmat{\mcl{P}[D_{11}]&\mcl{P}[D_{12}]\\\mcl{P}[D_{21}]&\mcl{P}[D_{22}]}\\
	 &\qquad=\bmat{\mcl{P}[B_{11}]\circ\mcl{P}[D_{11}]&\mcl{P}[B_{11}]\circ\mcl{P}[D_{12}]\\\mcl{P}[B_{21}]\circ\mcl{P}[D_{11}]&\mcl{P}[B_{21}]\circ\mcl{P}[D_{12}]}\\
	 &\qquad\qquad+\bmat{\mcl{P}[B_{12}]\circ\mcl{P}[D_{21}]&\mcl{P}[B_{12}]\circ\mcl{P}[D_{22}]\\\mcl{P}[B_{22}]\circ\mcl{P}[D_{21}]&\mcl{P}[B_{22}]\circ\mcl{P}[D_{22}]}\\
	 &\qquad=\bmat{\mcl{P}[R_{11}]&\mcl{P}[R_{12}]\\\mcl{P}[R_{21}]&\mcl{P}[R_{22}]}=\mcl{P}[R]
	\end{align*}	
	
\end{proof}

\clearpage

\section{Additional Proofs}
\subsection{Inverse of 011-PI Operators}\label{sec_inv_appendix}
\begin{lem}\label{lem_inv_appendix}
Suppose
\[
 Q=\bmat{Q_{00}&Q_{0x}&Q_{0y}\\Q_{x0}&Q_{xx}&Q_{xy}\\Q_{y0}&Q_{yx}&Q_{yy}}\in{N}_{011}\smallbmat{n_0&n_0\\n_1&n_1}
\]
where $Q_{xx}=\{Q_{xx}^0,Q^1_{xx},Q^1_{xx}\}\in\mcl{N}_{1D}^{n_1\times n_1}$ (separable) and $Q_{yy}=\{Q^0_{yy},Q^1_{yy},Q^1_{yy}\}\in\mcl{N}_{1D}^{n_1\times n_1}$ (separable). Suppose the matrix-valued functions can be decomposed as
\begin{align*}
 &\bmat{&Q_{0x}(x)&Q_{0y}(y)\\Q_{x0}(x)&Q_{xx}^{1}(x,\theta)&Q_{xy}(x,y)\\Q_{y0}(y)&Q_{yx}(x,y)&Q_{yy}^{1}(y,\nu)}
 =\\
&\qquad\bmat{&H_{0x}Z(x)&H_{0y}Z(y)\\
	Z^T(x)H_{x0}&Z^T(x)\Gamma_{xx}Z(\theta)&Z^T(x)\Gamma_{xy}Z(y)\\
	Z^T(y)H_{y0}&Z^T(y)\Gamma_{yx}Z(x)&Z^T(y)\Gamma_{yy}Z(\nu)}
\end{align*}
for some $q\in \N$ and $Z \in L_2^{q \times n_1}$ and
\[
 \bmat{&H_{0x}&H_{0y}\\H_{x0}&\Gamma_{xx}&\Gamma_{xy}\\H_{y0}&\Gamma_{yx}&\Gamma_{yy}}\in
 \bmat{&\mb{R}^{n_0\times q}&\mb{R}^{n_0\times q}\\
 \mb{R}^{q\times n_0}&\mb{R}^{q\times q}&\mb{R}^{q\times q}    \\
 \mb{R}^{q\times n_0}&\mb{R}^{q\times q}&\mb{R}^{q\times q}}
\]
Finally, define the linear parameter map $\mcl{L}_{\text{inv}}:\mcl{N}_{011}\smallbmat{n_0&n_0\\n_1&n_1}\rightarrow \mcl{N}_{011}\smallbmat{n_0&n_0\\n_1&n_1}$ as in Equation.~\eqref{eq_inv_011_appendix} in Figure~\ref{fig_inverse_operator_appendix}.
Then, there exists a set of parameters $\hat{Q}=\mcl{L}_{\text{inv}}(Q)\in\mcl{N}_{011}\smallbmat{n_0&n_0\\n_1&n_1}$ such that
\begin{align*}
 (\mcl{P}[\hat Q]\circ\mcl{P}[Q]) \mbf{u}=(\mcl{P}[Q]\circ\mcl{P}[\hat Q]) \mbf{u}=\mbf{u},
\end{align*}
for any $\mbf{u}\in\R^{n_0}\times L_2^{n_1}[x]\times L_2^{n_1}[y]$.\\
\end{lem}

\begin{figure*}[!t]
\hrulefill
\footnotesize

\begin{flalign}\label{eq_inv_011_appendix}
&\qquad\mcl{L}_{\text{inv}}(Q)=
\bmat{\hat{Q}_{00}&\hat{Q}_{0x}&\hat{Q}_{0y}\\\hat{Q}_{x0}&\hat{Q}_{xx}&\hat{Q}_{xy}\\\hat{Q}_{y0}&\hat{Q}_{yx}&\hat{Q}_{yy}}\in\mcl{N}_{011}\smallbmat{n_0&n_0\\n_1&n_1},
&	&\text{with}
&\mat{\hat{Q}_{xx}=\{\hat{Q}_{xx}^{0},\hat{Q}_{xx}^{1},\hat{Q}_{xx}^{1}\}\in\mcl{N}_{1D}^{n_1\times n_1},\\
\hat{Q}_{yy}=\{\hat{Q}_{yy}^{0},\hat{Q}_{yy}^{1},\hat{Q}_{yy}^{1}\}\in\mcl{N}_{1D}^{n_1\times n_1}},	&	&
\end{flalign}
and with
\begin{flalign*}
&\hspace*{0.2cm}\hat{Q}_{00}=\bbl(I_{n_0}-\hat{H}_{0x}K_{xx}H_{x0}-\hat{H}_{0y}K_{yy}H_{y0}\bbr)Q_{00}^{-1},	\\
&\begin{array}{l}
\hat{Q}_{xx}^{0}(x)=[Q_{xx}^{0}(x)]^{-1},		\\
 \\
\hat{Q}_{yy}^{0}(y)=[Q_{yy}^{0}(y)]^{-1},
\end{array}
&
&\hspace*{-2.0cm}\bmat{&\hat{Q}_{0x}(x)&\hat{Q}_{0y}(y)\\\hat{Q}_{x0}(x)&\hat{Q}_{xx}^{1}(x,\theta)&\hat{Q}_{xy}(x,y)\\\hat{Q}_{y0}(y)&\hat{Q}_{yx}(x,y)&\hat{Q}_{yy}^{1}(y,\nu)}
=
\bmat{&\hat{H}_{0x}\hat{Z}_{0x}(x)&\hat{H}_{0y}\hat{Z}_{0y}(y)\\
	\hat{Z}_{x0}^T(x)\hat{H}_{x0}&\hat{Z}_{x0}^T(x)\hat{\Gamma}_{xx}\hat{Z}_{0x}(\theta)&\hat{Z}_{x0}^T(x)\hat{\Gamma}_{xy}\hat{Z}_{0y}(y)\\
	\hat{Z}_{y0}^T(y)\hat{H}_{y0}&\hat{Z}_{y0}^T(y)\hat{\Gamma}_{yx}\hat{Z}_{0x}(x)&\hat{Z}_{y0}^T(y)\hat{\Gamma}_{yy}\hat{Z}_{0y}(\nu)},
\end{flalign*}
where
\begin{flalign*}
&\quad \hat{Z}_{x0}(x)=Z(x)[\hat{Q}_{xx}^{0}(x)]^T, &
&\hat{Z}_{0x}(x)=Z(x)\hat{Q}_{xx}^{0}(x), 		&
&\hat{Z}_{y0}(y)=Z(y)[\hat{Q}_{yy}^{0}(y)]^T, &
&\hat{Z}_{0y}(y)=Z(y)\hat{Q}_{yy}^{0}(y),	&
\end{flalign*}
and where
\begin{align*}
\hat{\Gamma}_{yx}&=-\bl(\Pi_{yx}-\Pi_{yy}\text{E}_{yx}\br)
\bl(\Sigma_{x}
-K_{xx}\Pi_{xy}\text{E}_{yx}\br)^{-1},   &
\hat{\Gamma}_{xy}&=-\bl(\Pi_{xy}
-\Pi_{xx}\text{E}_{xy}\br)
\bl(\Sigma_{y}-K_{yy}\Pi_{yx}\text{E}_{xy}\br)^{-1}, \\
\hat{\Gamma}_{yy}&=-\bbl(\Pi_{yy}+\hat{\Gamma}_{yx}K_{xx}\Pi_{xy}\bbr)
\Sigma_{y}^{-1}, 	&
\hat{\Gamma}_{xx}&=-\bbl(\Pi_{xx}
+\hat{\Gamma}_{xy}K_{yy}\Pi_{yx}\bbr)\Sigma_{x}^{-1},   \\
\hat{H}_{y0}&=-\bbl(H_{y0}+\hat{\Gamma}_{yy}K_{yy}H_{y0} + \hat{\Gamma}_{yx}K_{xx}H_{x0}\bbr)
Q_{00}^{-1},      &
\hat{H}_{x0}&=
-\bbl(H_{x0} + \hat{\Gamma}_{xx}K_{xx} H_{x0} + \hat{\Gamma}_{xy} K_{yy} H_{y0}\bbr)Q_{00}^{-1},  \\
\hat{H}_{0y}&=-\bl(Q_{00}^{-1}H_{0y}-Q_{00}^{-1}H_{0x} \text{E}_{xy}\br)
\bl(\Sigma_{y}-K_{yy}\Pi_{yx} \text{E}_{xy}\br)^{-1},   &
\hat{H}_{0x}&=-\bbl(Q_{00}^{-1}H_{0x}
+\hat{H}_{0y}K_{yy}\Pi_{yx}\bbr)\Sigma_{x}^{-1},
\end{align*}
with
\begin{align*}
&\begin{array}{l}
\Pi_{xx}=\Gamma_{xx}-H_{x0}Q_{00}^{-1}H_{0x},   \\
\Pi_{xy}=\Gamma_{xy}-H_{x0}Q_{00}^{-1}H_{0y},   \\
\Pi_{yx}=\Gamma_{yx}-H_{y0}Q_{00}^{-1}H_{0x},   \\
\Pi_{yy}=\Gamma_{yy}-H_{y0}Q_{00}^{-1}H_{0y},   
\end{array}		&
&\begin{array}{l}
\Sigma_x=I_{q}+K_{xx}\Pi_{xx},    \\
\Sigma_y=I_{q}+K_{yy}\Pi_{yy},    \\
\text{E}_{xy}=\Sigma_{x}^{-1}K_{xx}\Pi_{xy},  \\
\text{E}_{yx}=\Sigma_{y}^{-1}K_{yy}\Pi_{yx},
\end{array}
&
&\begin{array}{l}
K_{xx}=\int_{a}^{b} Z(x)\hat{Q}_{xx}^0(x)Z^T(x)dx,    \\
 \\
K_{yy}=\int_{c}^{d} Z(y)\hat{Q}_{yy}^0(y)Z^T(y)dy.
\end{array}
\end{align*}
	
\hrulefill
\caption{Parameters $\hat{Q}$ describing the inverse PI operator $\mcl{P}[\hat{Q}]=\mcl{P}[Q]^{-1}$ in Lemma~\ref{lem_inv_appendix}}
\label{fig_inverse_operator_appendix}
\end{figure*}

\begin{proof}

Let $Q,\hat{Q}\in\mcl{N}\smallbmat{n_0&n_0\\n_1&n_1}$ be as defined. Then, by the composition rules of PI operators, we have $\mcl{P}[R]=\mcl{P}[\hat{Q}]\circ\mcl{P}[Q]$, where 
\begin{align*}
 R=\bmat{R_{00}&R_{0x}&R_{0y}\\R_{x0}&R_{xx}&R_{xy}\\R_{y0}&R_{yx}&R_{yy}}= \mcl{L}_{011}(\hat{Q},Q)\in\mcl{N}\smallbmat{n_0&n_0\\n_1&n_1},
\end{align*}
with $R_{xx}=\{R_{xx}^{0},R_{xx}^{1},R_{xx}^{1}\}\in\mcl{N}_{1D}^{n_1\times n_1}$ and $R_{yy}=\{R_{yy}^{0},R_{yy}^{1},R_{yy}^{1}\}\in\mcl{N}_{1D}^{n_1\times n_1}$, and where
\begin{align*}
 &R_{00}\! =\! \hat{Q}_{00}Q_{00}\! +\! \int_{a}^{b}\!  \hat{Q}_{0x}(x)Q_{x0}(x)dx  
 \! +\! \int_{c}^{d}\!  \hat{Q}_{0y}(y)Q_{y0}(y)dy    \\
 &R_{xx}^{0}(x) =\hat{Q}_{xx}^{0}(x)Q_{xx}^{0}(x)  \\
 &R_{yy}^{0}(y) =\hat{Q}_{yy}^{0}(y)Q_{yy}^{0}(y)    
\end{align*}
and
\begin{align*}
 &R_{0x}(x)=\hat{Q}_{00} Q_{0x}(x) + \hat{Q}_{0x}(x)Q_{xx}^{0}(x) \\
 &\quad+ \int_{a}^{b} \hat{Q}_{0x}(\theta) Q_{xx}^{1}(\theta,x) d\theta + \int_{c}^{d} \hat{Q}_{0y}(y)Q_{yx}(x,y) dy \\
 &R_{0y}(y)=\hat{Q}_{00} Q_{0y}(y) + \hat{Q}_{0y}(y)Q_{yy}^{0}(y) \\
 &\quad+ \int_{c}^{d} \hat{Q}_{0y}(\nu) Q_{yy}^{1}(\nu,y) d\nu + \int_{a}^{b} \hat{Q}_{0x}(x)Q_{xy}(x,y) dx,     
\end{align*}
\begin{align*}
 &R_{x0}(x)=\hat{Q}_{x0}(x)Q_{00} + \hat{Q}_{xx}^{0}(x)Q_{x0}(x)    \\
 &\quad+\int_{a}^{b} \hat{Q}_{xx}^{1}(x,\theta) Q_{x0}(\theta) d\theta
 + \int_{c}^{d} \hat{Q}_{xy}(x,y) Q_{y0}(y) dy   \\
 &R_{xx}^{1}(x,\theta) =\hat{Q}_{x0}(x)Q_{0x}(\theta)+\hat{Q}_{xx}^{0}(x)Q_{xx}^{1}(x,\theta) \\
 &\quad+\hat{Q}_{xx}^{1}(x,\theta)Q_{xx}^{0}(\theta) 
 +\int_{a}^{b} \hat{Q}_{xx}^{1}(x,\eta)Q_{xx}^{1}(\eta,\theta)d\eta \\
 &\qquad +\int_{c}^{d} \hat{Q}_{xy}(x,y) Q_{yx}(\theta,y) dy,  \\
 &R_{xy}(x,y)=\hat{Q}_{x0}(x)Q_{0y}(y) + \hat{Q}_{xx}^{0}(x)Q_{xy}(x,y) \\
 &\quad+\int_{a}^{b} \hat{Q}_{xx}^{1}(x,\theta)Q_{xy}(\theta,y) d\theta \\
 &\qquad +\hat{Q}_{xy}(x,y)Q_{yy}^{0}(y) +\int_{c}^{d} \hat{Q}_{xy}(x,\nu)Q_{yy}^{1}(\nu,y) d\nu    \\
 &R_{y0}(y)=\hat{Q}_{y0}(y)Q_{00} + \hat{Q}_{yy}^{0}(y)Q_{y0}(y)    \\
 &\quad+\int_{c}^{d} \hat{Q}_{yy}^{1}(y,\nu) Q_{y0}(\nu) d\nu 
 +\int_{a}^{b} \hat{Q}_{yx}(x,y) Q_{x0}(x) dx  \\
 &R_{yx}(y,x)=\hat{Q}_{y0}(y)Q_{0x}(x) + \hat{Q}_{yy}^{0}(y)Q_{yx}(x,y) \\
 &\quad+\int_{c}^{d} \hat{Q}_{yy}^{1}(y,\nu)Q_{yx}(x,\nu)   \\
 &\qquad+\hat{Q}_{yx}(x,y)Q_{xx}^{0}(x) + \int_{a}^{b} \hat{Q}_{yx}(\theta,y)Q_{xx}^{1}(\theta,x) d\theta  \\
 &R_{yy}^{1}(y,\nu) =\hat{Q}_{y0}(y)Q_{0y}(\nu) + \hat{Q}_{yy}^{0}(y)Q_{yy}^{1}(y,\nu)  \\
 &\quad+\hat{Q}_{yy}^{1}(y,\nu)Q_{yy}^{0}(\nu) 
 +\int_{c}^{d} \hat{Q}_{yy}^{1}(y,\mu)Q_{yy}^{1}(\mu,\nu)d\mu   \\
 &\qquad +\int_{a}^{b} \hat{Q}_{yx}(x,y) Q_{xy}(x,\nu) dx.
\end{align*}
To demonstrate that the operator $\mcl{P}[\hat{Q}]$ defines an inverse of the operator $\mcl{P}[Q]$, we show that $\mcl{P}[R]$ describes an identity operation on $\R^{n_0}\times L_2^{n_1}[x]\times L_2^{n_1}[y]$. 
In particular, we show that each of the functions $R_{ij}$ is constantly equal to zero, except for $R_{00}$, $R_{xx}^{0}$ and $R_{yy}^{0}$, which are identity matrices of appropriate sizes. 
To this end, we first note that $\hat{Q}_{xx}^{0}=[Q_{xx}^{0}]^{-1}$ and $\hat{Q}_{yy}^{0}=[Q_{yy}^{0}]^{-1}$, from which it immediately follows that $R_{xx}^{0}=I_{n_1}$ and $R_{yy}^{0}=I_{n_1}$. To see that also $R_{00}=I_{n_0}$, we expand each of the terms in its definition, obtaining
\begin{align*}
 \hat{Q}_{00}Q_{00}
 =\bbl(I_{n_0}&-\hat{H}_{0x}K_{xx}H_{x0}-\hat{H}_{0y}K_{yy}H_{y0}\bbr),   \\
 \int_{a}^{b}\hat{Q}_{0x}(x)Q_{x0}(x)dx    
 &=\int_{a}^{b}\hat{H}_{0x}
 Z(x)\hat{Q}_{xx}^{0}Z^T(x)H_{x0} dx   \\
 &=\hat{H}_{0x} K_{xx}H_{x0},        \\
 \int_{c}^{d}\hat{Q}_{0y}(y)Q_{y0}(y)dy    
 &=\int_{c}^{d}\hat{H}_{0y}
 Z(y)\hat{Q}_{yy}^{0}Z^T(y)H_{y0} dy   \\
 &=\hat{H}_{0y} K_{yy}H_{y0}.        
\end{align*}
Adding these terms, we immediately find $R_{00}=I_{n_0}$. 

For the remaining functions, we also expand the different terms in their definitions. Starting with $R_{0x}(x)$, we find
\begin{align*}
 &\hat{Q}_{00}Q_{0x}
 =\bbl(I_{n_0}\! -\! \hat{H}_{0x}K_{xx}H_{x0}\! -\! \hat{H}_{0y}K_{yy}H_{y0}\bbr)Q_{00}^{-1}H_{0x}Z(x)  \\
 &\hspace*{1.2cm}=Q_{00}^{-1}H_{0x}Z(x)
 -\hat{H}_{0x}K_{xx}\left(\Gamma_{xx}-\Pi_{xx}\right)Z(x) \\
 &\hspace*{1.2cm}\qquad-\hat{H}_{0y}K_{yy}\left(\Gamma_{yx}-\Pi_{yx}\right)Z(x), \\ 
 &\hat{Q}_{0x}(x)Q_{xx}^{0}(x)
 =\hat{H}_{0x}Z(x)\hat{Q}_{xx}^{0}(x)Q_{xx}^{0}(x)
 =\hat{H}_{0x}Z(x),  \\
 &\int_{a}^{b}\hat{Q}_{0x}(\theta)Q_{xx}^{1}(\theta,x)d\theta =  \\
 &\hspace*{1.5cm} \int_{a}^{b}\hat{H}_{0x}Z(\theta)\hat{Q}_{xx}^{0}(\theta)Z^T(\theta)\Gamma_{xx}Z(x)d\theta   \\
 &\hspace{5.0cm}=\hat{H}_{0x}K_{xx}\Gamma_{xx}Z(x), \\
 &\int_{c}^{d}\hat{Q}_{0y}(y)Q_{yx}(x,y)dy =   \\
 &\hspace*{1.5cm} \int_{c}^{d}\hat{H}_{0y}Z(y)\hat{Q}_{yy}^{0}(y)Z^T(y)\Gamma_{yx}Z(x)dy \\  
 &\hspace*{5.0cm}=\hat{H}_{0y}K_{yy}\Gamma_{yx}Z(x), 
\end{align*}
from which it follows that
\begin{align*}
 R_{0x}(x) &= Q_{00}^{-1}H_{0x}Z(x) + \hat{H}_{0x}Z(x) \\
 &\qquad +\hat{H}_{0x}K_{xx}\Pi_{xx} Z(x)
 +\hat{H}_{0y}K_{yy}\Pi_{yx} Z(x)   \\
 &=Q_{00}^{-1}H_{0x}Z(x) + \hat{H}_{0x}Z(x)   \\
 &\qquad +\hat{H}_{0x}\left(\Sigma_x - I_{q}\right) Z(x)
 +\hat{H}_{0y}K_{yy}\Pi_{yx} Z(x)   \\
 &=Q_{00}^{-1}H_{0x}Z(x)
 -\left(Q_{00}^{-1}H_{0x}+\hat{H}_{0y}K_{yy}\Pi_{yx}\right) Z(x)    \\
 &\qquad+\hat{H}_{0y}K_{yy}\Pi_{yx} Z(x)   
 =0.
\end{align*}
Expanding the terms in the expression of $R_{0y}$ in the same way, we obtain
\begin{align*}
&\hat{Q}_{00}Q_{0y}
=Q_{00}^{-1}H_{0y}Z(y)
-\hat{H}_{0x}K_{xx}\left(\Gamma_{xy}-\Pi_{xy}\right)Z(y) \\
&\hspace*{1.2cm}\qquad-\hat{H}_{0y}K_{yy}\left(\Gamma_{yy}-\Pi_{yy}\right)Z(y), \\ 
&\hat{Q}_{0y}(x)Q_{yy}^{0}(y)
=\hat{H}_{0y}Z(y),  \\
&\int_{c}^{d}\hat{Q}_{0y}(\nu)Q_{yy}^{1}(\nu,y)d\nu 
=\hat{H}_{0y}K_{yy}\Gamma_{yy}Z(y), \\
&\int_{a}^{b}\hat{Q}_{0x}(x)Q_{xy}(x,y)dx  
=\hat{H}_{0x}K_{xx}\Gamma_{xy}Z(y), 
\end{align*}
suggesting also
\begin{align*}
 R_{0y}(y)&= Q_{00}^{-1}H_{0y}Z(y) + \hat{H}_{0y}Z(y) \\
 &\qquad +\hat{H}_{0x}K_{xx}\Pi_{xy}Z(y)
 +\hat{H}_{0y}K_{yy}\Pi_{yy}Z(y)    \\
 &= Q_{00}^{-1}H_{0y}Z(y) + \hat{H}_{0y}Z(y) \\
 &\qquad +\hat{H}_{0x}\Sigma_{x}\text{E}_{xy}Z(y)
 +\hat{H}_{0y}\left(\Sigma_y - I_q\right)Z(y)   \\
 &= Q_{00}^{-1}H_{0y}Z(y) -\left(Q_{00}^{-1}H_{0x}+\hat{H}_{0y}K_{yy}\Pi_{yx}\right)\text{E}_{xy}Z(y)   \\
 &\qquad +\hat{H}_{0y}\Sigma_y Z(y)  \\
 &=\left(Q_{00}^{-1}H_{0y}-Q_{00}^{-1}H_{0x}\text{E}_{xy}\right)Z(y) \\
 &\qquad +\hat{H}_{0y}\left(\Sigma_y - K_{yy}\Pi_{yx}\text{E}_{xy}\right)Z(y)    \\
 &=\left(Q_{00}^{-1}H_{0y}-Q_{00}^{-1}H_{0x}\text{E}_{xy}\right)Z(y)    \\
 &\qquad -\left(Q_{00}^{-1}H_{0y}-Q_{00}^{-1}H_{0x}\text{E}_{xy}\right)Z(y)
 =0.
\end{align*}
Next, we consider the expression for $R_{x0}(x)$, for which
\begin{align*}
 &\hat{Q}_{x0}(x)Q_{00}
 =\hat{Q}_{xx}^{0}(x)Z^T(x)\hat{H}_{x0}Q_{00}  \\
 &\quad=-\hat{Q}_{xx}^{0}Z^T(x)\bbl(H_{x0} + \hat{\Gamma}_{xx}K_{xx} H_{x0} + \hat{\Gamma}_{xy} K_{yy} H_{y0}\bbr),   \\
 &\hat{Q}_{xx}^{0}(x)Q_{x0}(x)
 =\hat{Q}_{xx}^{0}(x)Z^T(x)H_{x0},   \\
 &\int_{a}^{b}\hat{Q}_{xx}^{1}(x,\theta)Q_{x0}(\theta)d\theta   =   \\
 &\hspace*{1.0cm}\int_{a}^{b}\hat{Q}_{xx}^{0}(x)Z^T(x)\hat{\Gamma}_{xx}Z(\theta)\hat{Q}_{xx}^{0}(\theta)Z^T(\theta)H_{x0}d\theta  \\
 &\hspace*{4.0cm}= \hat{Q}_{xx}^{0}(x)Z^T(x)\hat{\Gamma}_{xx}K_{xx}H_{x0}, \\
 &\int_{c}^{d}\hat{Q}_{xy}(x,y)Q_{y0}(y)dy = \\
 &\hspace*{1.0cm}\int_{c}^{d}\hat{Q}_{xx}^{0}(x)Z^T(x)\hat{\Gamma}_{xy}Z(y)\hat{Q}_{yy}^{0}(y)Z^T(y)H_{y0}dy  \\
 &\hspace*{4.0cm} =\hat{Q}_{xx}^{0}(x)Z^T(x)\hat{\Gamma}_{xy}K_{yy}H_{y0}.
\end{align*}
Adding these terms, it is clear that also $R_{x0}(x)=0$.\\
Similarly, for $R_{xx}^{1}$,
\begin{align*}
 &\hat{Q}_{x0}(x)Q_{0x}(\theta)
 =\hat{Q}_{xx}^{0}(x)Z^T(x)\hat{H}_{x0} H_{0x}Z(\theta), \\
 &\hat{Q}_{xx}^{0}(x)Q_{xx}^{1}(x,\theta)
 =\hat{Q}_{xx}^{0}(x)Z^T(x)\Gamma_{xx} Z(\theta),    \\
 &\hat{Q}_{xx}^{1}(x,\theta)Q_{xx}^{0}(\theta)
 =\hat{Q}_{xx}^{0}(x)Z^T(x)\hat{\Gamma}_{xx} Z(\theta), \\
 &\int_{a}^{b}\hat{Q}_{xx}^{1}(x,\eta)Q_{xx}^{1}(\eta,\theta)d\eta =\\
 &\quad \int_{a}^{b} \hat{Q}_{xx}^{0}(x)Z^T(x)\hat{\Gamma}_{xx} Z(\eta)\hat{Q}_{xx}^{0}(\eta)Z^T(\eta)\Gamma_{xx} Z(\theta)d\eta \\
 &\hspace*{3.8cm} =\hat{Q}_{xx}^{0}(x)Z^T(x)\hat{\Gamma}_{xx} K_{xx}\Gamma_{xx} Z(\theta),  \\
 &\int_{c}^{d}\hat{Q}_{xy}(x,y)Q_{yx}(\theta,y)dy =\\
 &\quad \int_{c}^{d} \hat{Q}_{xx}^{0}(x)Z^T(x)\hat{\Gamma}_{xy} Z(y)\hat{Q}_{yy}^{0}(y)Z^T(y)\Gamma_{yx} Z(\theta)dy \\
 &\hspace*{3.8cm} =\hat{Q}_{xx}^{0}(x)Z^T(x)\hat{\Gamma}_{xy} K_{yy}\Gamma_{yx} Z(\theta).
\end{align*}
We note that each of these terms may be described as a constant matrix, premultiplied by the function $\hat{Q}_{xx}^{0}(x)Z^T(x)$, and postmultiplied by the function $Z(\theta)$. Hence, we may also express $R_{xx}^{1}(x,\theta)=\hat{Q}_{xx}^{0}(x)Z^T(x) A_{xx} Z(\theta)$ for some matrix $A_{xx}$. In particular, adding the different terms, we find this matrix to be given by
\begin{align*}
 A_{xx}
 &=\hat{H}_{x0} H_{0x}
 +\Gamma_{xx} + \hat{\Gamma}_{xx}    
 +\hat{\Gamma}_{xx} K_{xx}\Gamma_{xx}
 +\hat{\Gamma}_{xy} K_{yy}\Gamma_{yx}   \\
 &=
 -\bbl[H_{x0} + \hat{\Gamma}_{xx}K_{xx} H_{x0} + \hat{\Gamma}_{xy} K_{yy} H_{y0}\bbr]Q_{00}^{-1}H_{0x} \\
 &\qquad +\Gamma_{xx} + \hat{\Gamma}_{xx}    
 +\hat{\Gamma}_{xx} K_{xx}\Gamma_{xx}
 +\hat{\Gamma}_{xy} K_{yy}\Gamma_{yx} \\
 &=
 -\bbl[\Gamma_{xx}-\Pi_{xx}\bbr]
 -\hat{\Gamma}_{xx}K_{xx}\bbl[\Gamma_{xx}-\Pi_{xx}\bbr] \\
 &\qquad -\hat{\Gamma}_{xy}K_{yy}\bbl[\Gamma_{yx}-\Pi_{yx}\bbr] 
  +\Gamma_{xx} + \hat{\Gamma}_{xx}     \\
 &\qquad\quad +\hat{\Gamma}_{xx} K_{xx}\Gamma_{xx}
 +\hat{\Gamma}_{xy} K_{yy}\Gamma_{yx} \\
 &=
 \Pi_{xx}
 +\hat{\Gamma}_{xx}K_{xx}\Pi_{xx} 
 +\hat{\Gamma}_{xy}K_{yy}\Pi_{yx} 
  + \hat{\Gamma}_{xx}\\
 &=
 \Pi_{xx}
 +\hat{\Gamma}_{xx}\bbl[\Sigma_x - I_q\bbr] 
 +\hat{\Gamma}_{xy}K_{yy}\Pi_{yx} 
  + \hat{\Gamma}_{xx} \\
 &=
 \Pi_{xx}
 -\bbl[\Pi_{xx}+\hat{\Gamma}_{xy}K_{yy}\Pi_{yx}\bbr] 
  +\hat{\Gamma}_{xy}K_{yy}\Pi_{yx} 
 =0,
\end{align*}
proving that also $R_{xx}^{1}=0$. Finally, for $R_{xy}$, we find
\begin{align*}
 &\hat{Q}_{x0}(x)Q_{0y}(y)
 =\hat{Q}_{xx}^{0}(x)Z^T(x)\hat{H}_{x0} H_{0y}Z(y), \\
 &\hat{Q}_{xx}^{0}(x)Q_{xy}^{1}(x,y)
 =\hat{Q}_{xx}^{0}(x)Z^T(x)\Gamma_{xy} Z(y),    \\
 &\int_{a}^{b}\hat{Q}_{xx}^{1}(x,\theta)Q_{xy}(\theta,y)d\theta =\\
 &\quad \int_{a}^{b} \hat{Q}_{xx}^{0}(x)Z^T(x)\hat{\Gamma}_{xx} Z(\theta)\hat{Q}_{xx}^{0}(\theta)Z^T(\theta)\Gamma_{xy} Z(y)d\theta \\
 &\hspace*{3.8cm} =\hat{Q}_{xx}^{0}(x)Z^T(x)\hat{\Gamma}_{xx} K_{xx}\Gamma_{xy} Z(y),  \\
 &\hat{Q}_{xy}(x,y)Q_{yy}^{0}(y)
 =\hat{Q}_{xx}^{0}(x)Z^T(x)\hat{\Gamma}_{xy} Z(y)\hat{Q}_{yy}^{0}(y)Q_{yy}^{0}(y)    \\
 &\hspace*{2.5cm} =\hat{Q}_{xx}^{0}(x)Z^T(x)\hat{\Gamma}_{xy} Z(y), \\
 &\int_{c}^{d}\hat{Q}_{xy}(x,\nu)Q_{yy}^{1}(\nu,y)d\nu =\\
 &\quad \int_{c}^{d} \hat{Q}_{xx}^{0}(x)Z^T(x)\hat{\Gamma}_{xy} Z(\nu)\hat{Q}_{yy}^{0}(\nu)Z^T(\nu)\Gamma_{yy} Z(y)d\nu \\
 &\hspace*{3.8cm} =\hat{Q}_{xx}^{0}(x)Z^T(x)\hat{\Gamma}_{xy} K_{yy}\Gamma_{yy} Z(y).
\end{align*}
Studying these terms, it is clear that (similar to $R_{xx}^{1}$) we may express $R_{xy}(x,y)=\hat{Q}_{xx}^{0}(x)Z^T(x) A_{xy} Z(y)$, where 
\begin{align*}
 A_{xy}
 &=
 +\Gamma_{xy} + \hat{\Gamma}_{xy}    
 +\hat{\Gamma}_{xx} K_{xx}\Gamma_{xy}
 +\hat{\Gamma}_{xy} K_{yy}\Gamma_{yy}  \\
 &=
 -\bbl[H_{x0} + \hat{\Gamma}_{xx}K_{xx} H_{x0} + \hat{\Gamma}_{xy} K_{yy} H_{y0}\bbr]Q_{00}^{-1}H_{0y} \\
 &\qquad +\Gamma_{xy} + \hat{\Gamma}_{xy}    
 +\hat{\Gamma}_{xx} K_{xx}\Gamma_{xy}
 +\hat{\Gamma}_{xy} K_{yy}\Gamma_{yy}  \\
 &=
 -\bbl[\Gamma_{xy}-\Pi_{xy}\bbr]
 -\hat{\Gamma}_{xx}K_{xx}\bbl[\Gamma_{xy}-\Pi_{xy}\bbr] \\
 &\qquad -\hat{\Gamma}_{xy}K_{yy}\bbl[\Gamma_{yy}-\Pi_{yy}\bbr] 
  +\Gamma_{xy} + \hat{\Gamma}_{xy}     \\
 &\qquad\quad +\hat{\Gamma}_{xx} K_{xx}\Gamma_{xy}
 +\hat{\Gamma}_{xy} K_{yy}\Gamma_{yy} \\
 &=
 \Pi_{xy}
 +\hat{\Gamma}_{xx}K_{xx}\Pi_{xy} 
 +\hat{\Gamma}_{xy}K_{yy}\Pi_{yy} 
  + \hat{\Gamma}_{xy} \\
 &=
 \Pi_{xy}
 +\hat{\Gamma}_{xx}\Sigma_x\text{E}_{xy} 
 +\hat{\Gamma}_{xy}\bbl[\Sigma_y - I_q\bbr]
  + \hat{\Gamma}_{xy} \\
 &=
 \Pi_{xy}
 -\bbl[\Pi_{xx}
 +\hat{\Gamma}_{xy}K_{yy}\Pi_{yx}\bbr]\text{E}_{xy} 
 +\hat{\Gamma}_{xy} \Sigma_y \\
 &=
 \Pi_{xy} - \Pi_{xx}\text{E}_{xy}   
 +\hat{\Gamma}_{xy}\bbl[\Sigma_y - K_{yy}\Pi_{yx}\text{E}_{xy}\bbr] \\
 &=
 \bbl[\Pi_{xy} - \Pi_{xx}\text{E}_{xy}\bbr]
 -\bbl[\Pi_{xy} - \Pi_{xx}\text{E}_{xy}\bbr] = 0,
\end{align*}
from which it follows that also $R_{xy}=0$.\\
Performing the same steps as for $R_{x0}$, $R_{xx}^{1}$ and $R_{xy}$, we can also show that $R_{y0}$, $R_{yy}^{1}$ and $R_{yx}$ are constantly equal to zero. This leaves only $R_{00}=I_{n_0}$, $R_{xx}^{0}=I_{n_1}$ and $R_{yy}^{0}=I_{n_1}$ as nonzero functions defining $\mcl{P}[R]=\mcl{P}[\hat{Q}]\circ\mcl{P}[Q]$. By definition of the 011-PI operator, it immediately follows that $\mcl{P}[R]\mbf{u}=\mbf{u}$ for any $\mbf{u}\in\R^{n_0}\times L_2^{n_1}[x]\times L_2^{n_1}[y]$, proving the desired result.\\

\end{proof}

\newpage

\subsection{Adjoint of 2D-PI operator}\label{sec_adjoint_appendix}
\begin{lem}\label{lem_adjoint_appendix}
Suppose $ N \in \mcl N_{2D}^{n\times m}$ and define $ \hat{N} \in \mcl N_{2D}^{m\times n}$ such that
\begin{align}\label{adj_operator_appendix}
  &\hat{N}(x,y,\theta,\nu)  \nonumber\\
  &=\bmat{\hat{N}_{00}(x,y)&\hat{N}_{01}(x,y,\nu)&\hat{N}_{02}(x,y,\nu)\\ \hat{N}_{10}(x,y,\theta)&\hat{N}_{11}(x,y,\theta,\nu)&\hat{N}_{12}(x,y,\theta,\nu)\\ \hat{N}_{20}(x,y,\theta)&\hat{N}_{21}(x,y,\theta,\nu)&\hat{N}_{22}(x,y,\theta,\nu)} \nonumber\\
  &=\bmat{N_{00}^T(x,y)&N^T_{02}(x,\nu,y)&N^T_{01}(x,\nu,y)\\ N^T_{20}(\theta,y,x)&N^T_{22}(\theta,\nu,x,y)&N^T_{21}(\theta,\nu,x,y)\\ N^T_{10}(\theta,y,x)&N^T_{12}(\theta,\nu,x,y)&N^T_{11}(\theta,\nu,x,y)}.
 \end{align}
Then for any $\mbf u\in L_2^{m}[x,y]$ and $\mbf{v}\in L_2^{n}[x,y]$,
\[
\ip{\mbf{v}}{\mcl{P}[N]\mbf{u}}_{L_2}=\ip{\mcl{P}[\hat{N}]\mbf{v}}{\mbf{u}}_{L_2}.
\]

\end{lem}

\begin{proof}
 Let $N\in\mcl{N}_{2D}^{n\times m}$ and $\hat{N}\in\mcl{N}_{2D}^{m\times n}$ be as defined. Recall the definition of the indicator function
 \begin{align*}
  \mbf{I}(x,\theta)=\begin{cases}
              1  &\text{if }x>\theta,  \\
              0  &\text{otherwise},
             \end{cases}
 \end{align*}
 and let the Dirac delta function $\bs{\delta}(x,\theta)=\bs{\delta}(\theta,x)$ be defined such that, for any $B\in L_2[x,\theta]$,
 \begin{align*}
  \int_{a}^{b}\bs{\delta}(x,\theta)B(x,\theta)d\theta = B(x,x).
 \end{align*}
 Then, as described in Appendix~\ref{sec_2Dalgebra_appendix}, defining
 \begin{align*}
  \bs{\Phi}_0(x,\theta)&=\bs{\delta}(x,\theta)	\\
  \bs{\Phi}_1(x,\theta)&=\mbf{I}(x-\theta)		&
  \bs{\Phi}_2(x,\theta)&=\mbf{I}(\theta-x)
 \end{align*}
 we may describe the 2D-PI operation $\mcl{P}[N]\mbf{u}$ as
 \begin{align*}
 &\bl(\mcl{P}[N]\mbf{u}\br)(x,y)=\\
 &~\int_{a}^{b}\int_{c}^{d}\bbbbl(\sum_{i,j=0}^{2}\bs{\Phi}_i(x,\theta)\bs{\Phi}_j(y,\nu)N(x,y,\theta,\nu)\mbf{u}(\theta,\nu)\bbbbr)d\nu d\theta
 \end{align*}
 Noting that $\bs{\Phi}_0(x,\theta)=\bs{\Phi}_0(\theta,x)$, and $\bs{\Phi}_1(x,\theta)=\bs{\Phi}_2(\theta,x)$, it immediately follows that
 \begin{align*}
  &\ip{\mbf{v}}{\mcl{P}[N]\mbf{u}}_{L_2}
  =\int_{a}^{b}\int_{c}^{d}\bbbl[\mbf{v}^T(x,y)\bbl(\mcl{P}[N]\mbf{u}\bbr)(x,y)\bbbr] dy dx   \\
  &=\int_{a}^{b}\int_{c}^{d}\bbbbl[\int_{a}^{b}\int_{c}^{d}\bbbbl(\sum_{i,j=0}^{2}\bs{\Phi}_i(x,\theta)\bs{\Phi}_j(y,\nu)\\
  &\hspace*{2.5cm}\mbf{v}^T(x,y)N(x,y,\theta,\nu)\mbf{u}(\theta,\nu)\bbbbr)d\nu d\theta\bbbbr]dy dx \\
  &=\int_{a}^{b}\int_{c}^{d}\bbbbl[\int_{a}^{b}\int_{c}^{d}\bbbbl(\sum_{i,j=0}^{2}\bs{\Phi}_i(\theta,x)\bs{\Phi}_j(\nu,y)\\
  &\hspace*{2.25cm}\bbl[N^T(\theta,\nu,x,y)\mbf{v}(\theta,\nu)\bbr]^T\mbf{u}(x,y)\bbbbr)d\nu d\theta\bbbbr]dy dx \\
  &=\int_{a}^{b}\int_{c}^{d}\bbbl[\bbl(\mcl{P}[\hat{N}]\mbf{v}\bbr)^T(x,y)\mbf{u}(x,y)\bbbr] dy dx   
  =\ip{\mcl{P}[\hat{N}]\mbf{v}}{\mbf{u}}_{L_2},
 \end{align*}
 as desired.
\end{proof}

\subsection{Map From Fundamental to PDE State}\label{sec_Tmap_appendix}

For the following theorem, recall the definition 
\begin{align}\label{eq_setX_appendix}
 X(\mcl{B}):=
 &\left\{\bmat{\mbf{u}_0\\\mbf{u}_{1}\\\mbf{u}_{2}}\in \bmat{L_2^{n_0}\\ H_1^{n_1}\\ H_2^{n_2}}\mid \mcl{B}\Lambda_{\text{bf}}\mbf{u}=0\right\},
\end{align}
for the space $X$ of solutions to a standardized PDE. In particular, recall that any such solution must satisfy the boundary conditions  $\mcl{B}\Lambda_{\text{bf}}\mbf{u}=0$, where $\Lambda_{\text{bf}}\mbf{u}$ describes the solution along the boundary, and $\mcl{B}=\mcl{P}[B]$ for some $B\in\mcl{N}_{011}$.

\begin{figure*}[!t]
\footnotesize
\hrulefill
\begin{align}\label{eq_Tmat_appendix}
    T_{11}(x,y,\theta,\nu) &= K_{33}^{11}(x,y,\theta,\nu)
    + T_{21}(x,y,\theta,\nu) + T_{12}(x,y,\theta,\nu)
    -T_{22}(x,y,\theta,\nu), \nonumber\\
    T_{21}(x,y,\theta,\nu) &= - K_{32}^{1}(x,y,\nu)G_{23}^{0}(\theta,\nu) + T_{22}(x,y,\theta,\nu), \nonumber\\
    T_{12}(x,y,\theta,\nu) &= - K_{31}^{1}(x,y,\theta)G_{13}^{0}(\theta,\nu) + T_{22}(x,y,\theta,\nu), \nonumber\\
    T_{22}(x,y,\theta,\nu) &= -K_{30}(x,y)G_{03}(\theta,\nu)
    -\int_{a}^{x}K_{31}^{1}(x,y,\eta)G_{13}^{1}(\eta,\nu,\theta)d\eta
    -\int_{c}^{y}K_{32}^{1}(x,y,\mu)G_{23}^{1}(\theta,\mu,\nu)d\mu, 
    \hspace{3.5cm}
\end{align}
where
\begin{align}\label{eq_K_mat_appendix}
 &K_{33}^{11}(x,y,\theta,\nu)=
    \bmat{
    0&0&0\\
    0&I_{n_1}&0\\
    0&0&(x-\theta)(y-\nu)
    },  &
 &K_{32}^{1}(x,y,\nu)=
    \bmat{
    0&0&0\\
    I_{n_1}&0&0\\
    0&(y-\nu)&(x-a)(y-\nu)
    },\nonumber\\
 &K_{30}(x,y)=
    \bmat{
    0&0&0&0&0\\
    I_{n_1}&0&0&0&0\\
    0&I_{n_2}&(x-a)&(y-c)&(y-c)(x-a)
    },   &
 &K_{31}^{1}(x,y,\theta)=
    \bmat{
    0&0&0\\
    I_{n_1}&0&0\\
    0&(x-\theta)&(y-c)(x-\theta)
    },
    \hspace{2.0cm}
\end{align}
and
\begin{align}\label{eq_G_mat_appendix}
    &\enspace G_{0}(x,y)=\hat{E}_{00}F_{0}(x,y) + \hat{E}_{01}(x)F_{1}^{0}(x,y) + \int_{a}^{b}\hat{E}_{01}(\theta)F_{1}^{1}(\theta,y,x)d\theta
    + \hat{E}_{02}(y)F_{2}^{0}(x,y) + \int_{c}^{d}\hat{E}_{02}(\nu)F_{2}^{1}(x,\nu,y),    \nonumber\\
    &
    \begin{array}{l}
    G_{1}^{0}(x,y)=\hat{E}_{11}^{0}(x)F_{1}^{0}(x,y),
    \\
    G_{1}^{1}(x,y,\theta)=\hat{E}_{10}(x)F_{0}(\theta,y) + \hat{E}_{11}^{0}(x)F_{1}^{1}(x,y,\theta) 
    \\
    \hspace*{1.5cm}+ \hat{E}_{11}^{1}(x,\theta)F_{1}^{0}(\theta,y)  
    +\int_{a}^{b}\hat{E}_{11}^{1}(x,\eta)F_{1}^{1}(\eta,y,\theta)d\eta
    \\
    \hspace*{2.0cm}+\hat{E}_{12}(x,y)F_{2}^{0}(\theta,y) + \int_{c}^{d}\hat{E}_{12}(x,\nu)F_{2}^{1}(\theta,\nu,y)d\nu,
    \end{array}
    \hspace{0.75cm}
    \begin{array}{l}
    G_{2}^{0}(x,y)=\hat{E}_{22}^{0}(y)F_{2}^{0}(x,y), 
    \\
    G_{2}^{1}(x,y,\nu)=\hat{E}_{20}(y)F_{0}(x,\nu) + \hat{E}_{22}^{0}(y)F_{2}^{1}(x,y,\nu) 
    \\
    \hspace*{1.5cm}+ \hat{E}_{22}^{1}(y,\nu)F_{2}^{0}(x,\nu)  
    +\int_{c}^{d}\hat{E}_{22}^{1}(y,\mu)F_{2}^{1}(x,\mu,\nu)d\mu
    \\
    \hspace*{2.0cm}+\hat{E}_{21}(x,y)F_{1}^{0}(x,\nu) + \int_{a}^{b}\hat{E}_{21}(\theta,y)F_{1}^{1}(\theta,\nu,x)d\theta,
    \end{array}
    \hspace{1.1cm}
\end{align}
with
\begin{align}\label{eq_F_mat_appendix}
    &F_{0}(x,y)=B_{00}H_{03}(x,y) + B_{01}(x)H_{13}^{0}(x,y) + B_{02}(y)H_{23}^{0}(x,y),    \nonumber\\
    &F_{1}^{1}(x,y,\theta)=B_{10}(x)H_{03}(\theta,y) + B_{11}^{1}(x,\theta)H_{13}^{0}(\theta,y)
    +B_{12}(x,y)H_{23}^{0}(\theta,y),   &
    &F_{1}^{0}(x,y)=B_{11}^{0}(x)H_{13}^{0}(x,y), \nonumber\\
    &F_{2}^{1}(x,y,\nu)=B_{20}(y)H_{03}(x,\nu) + B_{22}^{1}(y,\nu)H_{23}^{0}(x,\nu)
    +B_{21}(x,y)H_{13}^{0}(x,\nu),   &
    &F_{2}^{0}(x,y)=B_{22}^{0}(y)H_{23}^{0}(x,y), \hspace{3.0cm}
\end{align}
and $\smallbmat{\hat{E}_{00}&\hat{E}_{01}&\hat{E}_{02}\\\hat{E}_{10}&\hat{E}_{11}&\hat{E}_{12}\\\hat{E}_{20}&\hat{E}_{21}&\hat{E}_{22}}=\mcl{L}_{\text{inv}}\left(\smallbmat{E_{00}&E_{01}&E_{02}\\E_{10}&E_{11}&E_{12}\\E_{20}&E_{21}&E_{22}}\right)\in\mcl{N}_{011}$, where $\mcl{L}_{\text{inv}}:\mcl{N}_{011}\rightarrow\mcl{N}_{011}$ is defined as in Equation~\eqref{eq_inv_011_appendix}, and $E_{11}=\{E_{11}^{0},E_{11}^{1},E_{11}^{1}\}\in\mcl{N}_{1D}$ and $E_{22}=\{E_{22}^{0},E_{22}^{1},E_{22}^{1}\}\in\mcl{N}_{1D}$, with 
\begin{align}\label{eq_E_mat_appendix}
    &E_{00}=B_{00}H_{00} + \int_{a}^{b}B_{01}(x)H_{10}(x)dx + \int_{c}^{d}B_{02}(y)H_{20}(y)dy, \nonumber\\
    &E_{01}(x)=B_{00}H_{01}(x) + B_{01}(x)H_{11}^{0}(x),  &
    &E_{02}(y)=B_{00}H_{02}(y) + B_{02}(y)H_{22}^{0}(y),  \nonumber\\
    &E_{10}(x)=B_{10}(x)H_{00}, &
    &E_{20}(y)=B_{20}(y)H_{00}, \nonumber\\
    &E_{11}^{0}(x)=B_{11}^{0}H_{11}^{0},   &
    &E_{22}^{0}(y)=B_{22}^{0}H_{22}^{0},   \nonumber\\
    &E_{11}^{1}(x,\theta)= B_{10}(x)H_{01}(\theta) + B_{11}^{1}(x,\theta)H_{11}^{0}, &
    &E_{22}^{1}(y,\nu)= B_{20}(y)H_{02}(\nu) + B_{22}^{1}(y,\nu)H_{22}^{0}, \nonumber\\
    &E_{12}(x,y)=B_{10}(x)H_{02}(y) + B_{12}(x,y)H_{22}^{0}(y), &
    &E_{21}(x,y)=B_{20}(y)H_{01}(x) + B_{21}(x,y)H_{11}^{0}(x), 
    \hspace{2.3cm}
\end{align}
where,
{\tiny
\begin{align}\label{eq_H_mat_appendix}
 &H_{00} =
    \bmat{I_{n_1}&0&0&0&0\\
    I_{n_1}&0&0&0&0\\
    I_{n_1}&0&0&0&0\\
    I_{n_1}&0&0&0&0\\
    0&I_{n_2}&0&0&0\\
    0&I_{n_2}&(b-a)&0&0\\
    0&I_{n_2}&0&(d-c)&0\\
    0&I_{n_2}&(b-a)&(d-c)&(d-c)(b-a)\\
    0&0&I_{n_2}&0&0\\
    0&0&I_{n_2}&0&0\\
    0&0&I_{n_2}&0&(d-c)\\
    0&0&I_{n_2}&0&(d-c)\\
    0&0&0&I_{n_2}&0\\
    0&0&0&I_{n_2}&(b-a)\\
    0&0&0&I_{n_2}&0\\
    0&0&0&I_{n_2}&(b-a)\\
    0&0&0&0&I_{n_2}\\
    0&0&0&0&I_{n_2}\\
    0&0&0&0&I_{n_2}\\
    0&0&0&0&I_{n_2}},    &
    &H_{01}(x) =
    \bmat{
    0&0&0\\I_{n_1}&0&0\\0&0&0\\I_{n_1}&0&0\\
    0&0&0\\ 0&(b-x)&0\\ 0&0&0\\
    0&(b-x)&(d-c)(b-x)\\
    0&0&0\\ 0&I_{n_2}&0 \\0&0&0\\ 0&I_{n_2}&(d-c)\\
    0&0&0\\ 0&0&(b-x) \\0&0&0\\ 0&0&(b-x)\\
    0&0&0\\ 0&0&I_{n_2} \\0&0&0\\ 0&0&I_{n_2}
    },   &
    &H_{02}(y) =
    \bmat{
    0&0&0\\0&0&0\\I_{n_1}&0&0\\I_{n_1}&0&0\\
    0&0&0\\ 0&0&0\\ 0&(d-y)&0\\
    0&(d-y)&(b-a)(d-y)\\
    0&0&0\\ 0&0&0 \\0&0&(d-y)\\ 0&0&(d-y)\\
    0&0&0\\ 0&0&0 \\0&I_{n_2}&0\\ 0&I_{n_2}&(b-a)\\
    0&0&0\\ 0&0&0 \\0&0&I_{n_2}\\ 0&0&I_{n_2}
    },   \nonumber\\
    &\begin{array}{l}
    H_{11}^{0}=
    \bmat{
    I_{n_1}&0&0\\I_{n_1}&0&0\\
    0&I_{n_2}&0\\0&I_{n_2}&(d-c)\\0&0&I_{n_2}\\0&0&I_{n_2}
    },\\
    \\
    H_{22}^{0}=
    \bmat{
    I_{n_1}&0&0\\I_{n_1}&0&0\\
    0&I_{n_2}&0\\0&I_{n_2}&(b-a)\\0&0&I_{n_2}\\0&0&I_{n_2}
    }
    \end{array}
    &
    &\begin{array}{l}
    H_{13}^{0}(y)=
    \bmat{
    0&0&0\\0&I_{n_1}&0\\
    0&0&0\\ 0&0&(d-y)\\0&0&0\\0&0&I_{n_2}
    },\\
    \\
    H_{23}^{0}(x)=
    \bmat{
    0&0&0\\0&I_{n_1}&0\\
    0&0&0\\ 0&0&(b-x)\\0&0&0\\0&0&I_{n_2}
    },
    \end{array}
    &
    &H_{03}(x,y)=
    \bmat{
    0&0&0\\0&0&0\\0&0&0\\0&I_{n_1}&0\\
    0&0&0\\0&0&0\\0&0&0\\ 0&0&(d-y)(b-x)\\0&0&0\\0&0&0\\0&0&0\\ 0&0&(d-y)\\0&0&0\\0&0&0\\0&0&0\\ 0&0&(b-x)\\
    0&0&0\\0&0&0\\0&0&0\\0&0&I_{n_2}
    }.
\end{align}
}
\hrulefill
\caption{Parameters $T$ describing PI operator $\mcl{T}=\mcl{P}[T]$ mapping the fundamental state back to the PDE state in Theorem~\ref{thm_Tmap_appendix}}
\label{fig_Tmap_matrices_appendix}
\end{figure*}

\begin{thm}\label{thm_Tmap_appendix}
Let
\[
 B=\bmat{B_{00}&B_{01}&B_{02}\\B_{10}&B_{11}&B_{12}\\B_{20}&B_{21}&B_{22}}\in\mcl{N}_{011}\smallbmat{n_1+4n_2& 4n_1+16n_2\\n_1+2n_2&2n_1+4n_2}
\]
with $B_{11}=\{B_{11}^{0},B_{11}^{1},B_{11}^{1}\}\in\mcl{N}_{1D}$ and $B_{22}=\{B_{22}^{0},B_{22}^{1},B_{22}^{1}\}\in\mcl{N}_{1D}$ be given, and such that $B_{11}^{0},B_{22}^{0}\in\R^{n_1+2n_2\times 2n_1+4n_2}$. Let
\[
 T=\bmat{T_{00}&0&0\\0&T_{11}&T_{12}\\0&T_{21}&T_{22}}\in\mcl{N}_{2D}
\]
where
\[
 T_{00}=\bmat{I_{n_0}&0&0\\0&0_{n_1}&0\\0&0&0_{n_2}}
\]
and $T_{11},T_{12},T_{21},T_{22}$ are as defined in Equations~\eqref{eq_Tmat_appendix} in Figure~\ref{fig_Tmap_matrices_appendix}. Then, if $\mcl{T}=\mcl{P}[T]$, for any $\mbf{u}\in X(\mcl{B})$ and $\hat{\mbf{u}} \in L_2^{n_0+n_1+n_2}[x,y]$, we have
\begin{align*}
 \mbf{u}=\mcl{T}\mcl{D}\mbf{u} \qquad \text{and} \qquad   \hat{\mbf{u}}=\mcl{D}\mcl{T}\hat{\mbf{u}},
\end{align*}
where $\mcl{D}=\bmat{I_{n_0}&&\\&\partial_x\partial_y&\\&&\partial_x^2\partial_y^2}$.\\
\end{thm}

\begin{proof}
 We will first proof the first identity, $\mbf{u}=\mcl{T}\mcl{D}\mbf{u}$. To this end, suppose $\mbf{u}\in X(\mcl{B})$, and define $\hat{\mbf{u}}=\mcl{D}\mbf{u}\in L_2^{n_0+n_1+n_2}[x,y]$. Furthermore,
 let $K_{ij}$ and $H_{ij}$ (for appropriate $i,j\in\{0,1,2,3\}$) be as defined in Equations~\eqref{eq_K_mat_appendix} and~\eqref{eq_H_mat_appendix}, and let
 \begin{align*}
  H_1&=\bmat{
  	H_{00}~H_{01}~H_{02}\\
  	0\ ~~H_{11}\ ~~0\\
  	~~0\ ~~\ 0 ~~~~H_{22}}\in\mcl{N}_{011} &
  &H_2=\bmat{H_{03}\\H_{13}\\H_{23}}\in\mcl{N}_{2D\rightarrow 011}  \\
  K_1&=\bmat{K_{30}\\K_{31}\\K_{32}}\in\mcl{N}_{011\rightarrow 2D}    &
  &\hspace{-0.5cm}K_2=\bmat{T_{00}~~0~~~0\\~0\ ~K_{22}^{11} ~0\\~0\ ~~~0 ~ ~~0}\in\mcl{N}_{2D},
 \end{align*}
 where
 \begin{align*}
	 &H_{11}\! =\! \{H_{11}^{0},0,0\}\! \in\mcl{N}_{1D}	&	 &\! \! H_{22}\! =\! \{H_{22}^{0},0,0\}\! \! \in\mcl{N}_{1D}	\\
	 &H_{13}\! =\! \{H_{13}^{0},0,0\}\! \in\mcl{N}_{2D\rightarrow 1D} 	&
	 &\! \! H_{23}\! =\! \{H_{23}^{0},0,0\}\! \in\mcl{N}_{2D\rightarrow 1D}  \\
	 &K_{31}\! =\! \{0,K_{31}^{1},0\}\! \in\mcl{N}_{1D\rightarrow 2D} 	&
	 &\! \! K_{32}\! =\! \{0,K_{32}^{1},0\}\! \in\mcl{N}_{1D\rightarrow 2D} .
 \end{align*}
 Then, by Lemma~\ref{lem_uhat_to_u} and Corollary~\ref{cor_uhat_to_BC},
 \begin{align}\label{eq_uhat_to_u_appendix}
 \Lambda_{\text{bf}} \mbf{u} &= \mcl{H}_1 \Lambda_{\text{bc}} \mbf{u}+\mcl{H}_2 \hat{\mbf{u}} \nonumber\\
 \mbf{u} &= \mcl{K}_1 \Lambda_{\text{bc}} \mbf{u}+\mcl{K}_2 \hat{\mbf{u}},
 \end{align}
 where $\mcl{H}_1=\mcl{P}[H_1]$, $\mcl{H}_2=\mcl{P}[H_2]$, $\mcl{K}_1=\mcl{P}[K_1]$, and $\mcl{K}_2=\mcl{P}[K_2]$. Enforcing the boundary conditions $\mcl{B}\Lambda_{\text{bf}}\mbf{u}=0$, we may use the composition rules of PI operators to express
 \begin{align*}
  0=\mcl{B}\Lambda_{\text{bf}}\mbf{u}
  =\mcl{B}\mcl{H}_1 \Lambda_{\text{bc}} \mbf{u}+\mcl{B}\mcl{H}_2 \hat{\mbf{u}}
  =\mcl{E}\Lambda_{\text{bc}}\mbf{u}
  +\mcl{F}\hat{\mbf{u}},
 \end{align*}
 where $\mcl{E}=\mcl{P}[E]$ and $\mcl{F}=\mcl{P}[F]$ with
 \begin{align*}
  E&=\bmat{E_{00}&E_{01}&E_{02}\\E_{10}&E_{11}&E_{12}\\E_{20}&E_{21}&E_{22}}\in\mcl{N}_{011},    &
  F&=\bmat{F_{0}\\F_{1}\\F_{2}}\in\mcl{N}_{2D\rightarrow 011},
 \end{align*}
 and
 \begin{align*}
  E_{11}&=\{E_{11}^{0},E_{11}^{1},E_{11}^{1}\}\in\mcl{N}_{1D}
  \\
  E_{22}&=\{E_{22}^{0},E_{22}^{1},E_{22}^{1}\}\in\mcl{N}_{1D} \\
  F_{1}&=\{F_{1}^{0},F_{1}^{1},F_{1}^{1}\}\in\mcl{N}_{2D\rightarrow 1D} \\
  F_{2}&=\{F_{2}^{0},F_{2}^{1},F_{2}^{1}\}\in\mcl{N}_{2D\rightarrow 1D}
 \end{align*}
 defined as in Equations~\eqref{eq_F_mat_appendix} and~\eqref{eq_E_mat_appendix}. Assuming $\mcl{B}$ is of sufficient rank, we can then invert the 011-PI operator $\mcl{E}$, allowing us to write
 \begin{align*}
  \Lambda_{\text{bc}}\mbf{u}=-\mcl{E}^{-1}\mcl{F}\hat{\mbf{u}}=-\mcl{G}\hat{\mbf{u}},
 \end{align*}
 where $\mcl{G}=\mcl{P}[G]$ with
 \begin{align*}
 G&=\bmat{G_{0}\\\{G_{1}^{0},G_{1}^{1},G_{1}^{1}\}\\\{G_{2}^{0},G_{2}^{1},G_{2}^{1}\}}\in\mcl{N}_{2D\rightarrow 011}
 \end{align*}
 defined as in Equations~\eqref{eq_G_mat_appendix}. Finally, substituting this expression into Equation~\eqref{eq_uhat_to_u_appendix}, and once more using the composition rules of PI operators, we obtain
 \begin{align*}
  \mbf{u} = \mcl{K}_1 \Lambda_{\text{bc}} \mbf{u}+\mcl{K}_2 \hat{\mbf{u}}
  &=-\mcl{K}_1 \mcl{G}\hat{\mbf{u}}+\mcl{K}_2 \hat{\mbf{u}}\\
  &=(\mcl{K}_2-\mcl{K}_1\mcl{G})\hat{\mbf{u}}
  =\mcl{T}\hat{\mbf{u}}=\mcl{T}\mcl{D}\mbf{u},
 \end{align*}
 as desired.

 Suppose now $\hat{\mbf{u}}\in L_2^{n_0+n_1+n_2}[x,y]$ and $\mbf{u}=\bmat{\mbf{u}_0\\\mbf{u}_1\\\mbf{u}_2}=\mcl{T}\hat{\mbf{u}}$. To prove the second identity, $\hat{\mbf{u}}=\mcl{D}\mcl{T}\hat{\mbf{u}}$, we split $\mbf{u}$ into its different components $\mbf{u}_0$, $\mbf{u}_1$ and $\mbf{u}_2$, using the matrices
 \begin{align*}
  J_0 &= \bmat{I_{n_0}&0&0},	&	J_1&=\bmat{I_{n_1}&0},	&	J_2&=\bmat{0&I_{n_2}},
 \end{align*}
 and
 \begin{align*}
  N=\bmat{0&I_{n_1}&0\\0&0&I_{n_2}}\in\R^{n_1+n_2\times n_0+n_1+n_2},
 \end{align*}
 so that
 \begin{align*}
  \mbf{u}_0&=J_0\mbf{u},	&
  \mbf{u}_1&=J_1N\mbf{u},	&	\mbf{u}_2&=J_2N\mbf{u}.
 \end{align*}
 By definition of the operator $\mcl{D}$, we then have to prove that
 \begin{align}
  \hat{\mbf{u}}_0&=J_0\mcl{T}\hat{\mbf{u}}=\mcl{P}[J_0 T]\hat{\mbf{u}}, \label{eq_uhat0=Du0_appendix}\\
  \hat{\mbf{u}}_1&=J_1 N\ \partial_x\partial_y\bbl(\mcl{T}\hat{\mbf{u}}\bbr)=J_1\ \partial_x\partial_y\bbl(\mcl{P}[N T]\hat{\mbf{u}}\bbr),	\label{eq_uhat1=Du1_appendix}\\
  \hat{\mbf{u}}_2&=J_2 N\ \partial_x^2\partial_y^2\bbl(\mcl{T}\hat{\mbf{u}}\bbr)=\partial_x\partial_y\bbl[J_2\ \partial_x\partial_y\bbl(\mcl{P}[N T]\hat{\mbf{u}}\bbr)\bbr]. \label{eq_uhat2=Du2_appendix}	
 \end{align} 
 To prove Relation~\eqref{eq_uhat0=Du0_appendix}, we use the definition of the parameters $T$ (Equations~\eqref{eq_Tmat_appendix}), suggesting $J_0T_{00}=J_0$, whilst $J_0T_{ij}=0$ for $i,j\in\{1,2\}$. It immediately follows that 
 \begin{align*}
 J_0\mcl{T}\hat{\mbf{u}}=\mcl{P}[J_0 T]\hat{\mbf{u}}=J_0\hat{\mbf{u}}=\hat{\mbf{u}}_0.
 \end{align*} 
 
 For the remaining relations, we first note that $NT_{00}=0$, 
 so that
 \begin{align*}
 N\mcl{T}=\mcl{P}[NT]=\mcl{P}\smallbmat{NT_{00}&0&0\\0&NT_{11}&NT_{12}\\0&NT_{21}&NT_{22}}=\mcl{P}\smallbmat{0&0&0\\0&\overline{T}_{11}&\overline{T}_{12}\\0&\overline{T}_{21}&\overline{T}_{22}},
 \end{align*}
 where we define $\overline{T}_{ij}=NT_{ij}$ for $i,j\in\{1,2\}$. Since the multiplier terms ($T_{0j}$ and $T_{i0}$) of this PI operator are all zero, by Lemmas~\ref{lem_div_operator_x} and~\ref{lem_div_operator_y}, we may write the composition of the differential operator $\partial_x\partial_y$ with this PI operator as a PI operator,
 \begin{align*}
 \partial_x\partial_y \bbl(\mcl{P}[NT]\hat{\mbf{u}}\bbr)&=\partial_x\partial_y\left(\mcl{P}\smallbmat{0&0&0\\0&\overline{T}_{11}&\overline{T}_{12}\\0&\overline{T}_{21}&\overline{T}_{22}}\hat{\mbf{u}}\right)\\
 &=\left(\mcl{P}\smallbmat{M_{00}&M_{01}&M_{02}\\M_{10}&M_{11}&M_{12}\\M_{20}&M_{21}&M_{22}}\hat{\mbf{u}}\right),
 \end{align*}
 where,
 \begin{align*}
 &M_{00}(x,y)=\overline{T}_{11}(x,y,x,y)-\overline{T}_{21}(x,y,x,y) \\
 &\hspace*{1.20cm}-\overline{T}_{12}(x,y,x,y)+\overline{T}_{22}(x,y,x,y)
 =\overline{K}_{33}^{11}(x,y,x,y)   \\
 &M_{10}(x,y,\theta)=\partial_x\bbl(
 \overline{T}_{11}(x,y,\theta,y)-\overline{T}_{12}(x,y,\theta,y)
 \bbr)\\
 &\hspace*{1.75cm}=\partial_x\bbl(\overline{K}_{33}^{11}(x,y,\theta,y)- \overline{K}_{32}^{1}(x,y,y){G}_{23}^{0}(\theta,y)\bbr)    \\
 &M_{20}(x,y,\theta)=\partial_x\bbl(
 \overline{T}_{21}(x,y,\theta,y)-\overline{T}_{22}(x,y,\theta,y)
 \bbr)\\
 &\hspace*{1.75cm}=-\partial_x\bbl(\overline{K}_{32}^{1}(x,y,y){G}_{23}^{0}(\theta,y)\bbr)  \\
 &M_{01}(x,y,\nu)=\partial_y\bbl(
 \overline{T}_{11}(x,y,x,\nu)-\overline{T}_{21}(x,y,x,\nu)
 \bbr)\\
 &\hspace*{1.75cm}=\partial_y\bbl(\overline{K}_{33}^{11}(x,y,x,\nu)- \overline{K}_{31}^{1}(x,y,x){G}_{13}^{0}(x,\nu)\bbr)    \\
 &M_{02}(x,y,\nu)=\partial_y\bbl(
 \overline{T}_{12}(x,y,x,\nu)-\overline{T}_{22}(x,y,x,\nu)
 \bbr)\\
 &\hspace*{1.75cm}=-\partial_y\bbl(\overline{K}_{31}^{1}(x,y,x){G}_{13}^{0}(x,\nu)\bbr)  \\
 &M_{11}(x,y,\theta,\nu)=\partial_x\partial_y\overline{T}_{11}(x,y,\theta,\nu)  \\
 &M_{21}(x,y,\theta,\nu)=\partial_x\partial_y\overline{T}_{21}(x,y,\theta,\nu)  \\
 &M_{12}(x,y,\theta,\nu)=\partial_x\partial_y\overline{T}_{12}(x,y,\theta,\nu)  \\
 &M_{22}(x,y,\theta,\nu)=\partial_x\partial_y\overline{T}_{22}(x,y,\theta,\nu)
 \end{align*}
 with $\overline{K}_{31}^{1}=NK_{31}^1$, $\overline{K}_{32}^{1}=NK_{32}^{1}$, and $\overline{K}_{33}^{11}=N K_{33}^{11}$.\\
 Studying Equations~\eqref{eq_K_mat_appendix} defining $K_{31}^{1}$, $K_{32}^{1}$ and $K_{33}^{11}$, it is easy to see that $\partial_{y}K_{31}^{1}(x,y,x)=\partial_{x}K_{32}^{1}(x,y,y)=\partial_{y}K_{33}^{11}(x,y,x,\nu)=\partial_{x}K_{33}^{11}(x,y,\theta,y)=0$, and therefore also
 \begin{align*}
 M_{10}(x,y,\theta)&=0,	&
 M_{20}(x,y,\theta)&=0,\\
 M_{01}(x,y,\nu)&=0,		&
 M_{02}(x,y,\nu)&=0,
 \end{align*}
 so that
 \begin{align*}
  \partial_x\partial_y \bbl(\mcl{P}[N T]\hat{\mbf{u}}\bbr)=\left(\mcl{P}\smallbmat{M_{00}&0&0\\0&M_{11}&M_{12}\\0&M_{21}&M_{22}}\hat{\mbf{u}}\right).
 \end{align*}
 For the remaining parameters $M_{ij}$, we consider the products $J_1 M_{ij}$. Once more studying Equations~\eqref{eq_K_mat_appendix}, we note that $J_1NK_{30}$, $J_1NK_{31}^{1}$, $J_1NK_{32}^{1}$ and $J_1NK_{33}$ are all constant matrices. By the definitions (Equations~\eqref{eq_Tmat_appendix}) of functions $T_{ij}$, for each $i,j\in\{1,2\}$, this implies also
 \begin{align*}
  J_1M_{ij}(x,y,\theta,\nu)&=J_1\ \partial_x\partial_y\overline{T}_{11}(x,y,\theta,\nu) \\ &=\partial_x\partial_y (J_1NT_{ij}(x,y,\theta,\nu))=0,
 \end{align*}
 Finally, for $M_{00}$, it is easy to see that
 \begin{align*}
 J_1M_{00}(x,y,x,y)&=J_1 NK_{33}^{11}(x,y,x,y)\\
 &\qquad\quad=\bmat{0&I_{n_1}&0}
 =J_1N.
 \end{align*}
 Combining the results, we obtain
 \begin{align*}
 J_1 N \partial_x\partial_y\bbl(\mcl{T}\hat{\mbf{u}}\bbr)\! 
 &=\! J_1 \partial_x\partial_y \bbl(\mcl{P}[N T]\hat{\mbf{u}}\bbr)\\
 \!  &=\! J_1\mcl{P}\! \smallbmat{M_{00}&0&0\\0&M_{11}&M_{12}\\0&M_{21}&M_{22}}\! \hat{\mbf{u}}\! 
 =\! \mcl{P}\smallbmat{J_1N&0&0\\0&0&0\\0&0&0}\hat{\mbf{u}}\! =\! \hat{\mbf{u}}_1,
 \end{align*}
 proving Relation~\eqref{eq_uhat1=Du1_appendix}. To prove the final Relation~\eqref{eq_uhat2=Du2_appendix}, we note that $J_2M_{00}(x,y,x,y)=0$, and therefore
 \begin{align*}
 J_2\ \partial_x\partial_y \mcl{P}[N T]=\mcl{P}\smallbmat{0&0&0\\0&J_2M_{11}&J_2M_{12}\\0&J_2M_{21}&J_2M_{22}}= \mcl{P}\smallbmat{0&0&0\\0&\overline{M}_{11}&\overline{M}_{12}\\0&\overline{M}_{21}&\overline{M}_{22}},
 \end{align*}
 where we define $\overline{M}_{ij}:=J_2M_{ij}$ for $i,j\in\{1,2\}$. The resulting operator contains no multiplier terms, allowing once more the composition with the differential operator $\partial_x\partial_y$ to be taken, yielding (by Lemmas~\ref{lem_div_operator_x} and~\ref{lem_div_operator_y})
 \begin{align*}
 \partial_x\partial_y\bbl[J_2\ \partial_x\partial_y\bbl(\mcl{P}[NT]\hat{\mbf{u}}\bbr)\bbr]    
 &=\partial_x\partial_y\left(\mcl{P}\smallbmat{0&0&0\\0&\overline{M}_{11}&\overline{M}_{12}\\0&\overline{M}_{21}&\overline{M}_{22}}\hat{\mbf{u}}\right) \\
 &\quad=\left(\mcl{P}\smallbmat{W_{00}&W_{01}&W_{02}\\W_{10}&W_{11}&W_{12}\\W_{20}&W_{21}&W_{22}}\hat{\mbf{u}}\right), 
 \end{align*}
 where
 \begin{align*}
 &W_{00}(x,y)=\overline{M}_{11}(x,y,x,y)-\overline{M}_{21}(x,y,x,y) \\
 &\hspace*{2.0cm}-\overline{M}_{12}(x,y,x,y)+\overline{M}_{22}(x,y,x,y)\\
 &\hspace*{1.5cm}=[\partial_x\partial_y\hat{K}_{33}^{11}(x,y,\theta,\nu)]\bbr|_{\theta=x,\nu=y}   \\
 &W_{10}(x,y,\theta)=\partial_x\bbl(
 \overline{M}_{11}(x,y,\theta,y)-\overline{M}_{12}(x,y,\theta,y)
 \bbr)\\
 &\qquad=\partial_x^2\partial_y\bbl(\hat{K}_{33}^{11}(x,y,\theta,\nu)- \hat{K}_{32}^{1}(x,y,\nu)G_{23}^{0}(\theta,\nu)\bbr)\bbr|_{\nu=y}\\
 &W_{20}(x,y,\theta)=\partial_x\bbl(
 \overline{M}_{21}(x,y,\theta,y)-\overline{M}_{22}(x,y,\theta,y)
 \bbr)\\
 &\qquad=-\partial_x^2\partial_y\bbl(\hat{K}_{32}^{1}(x,y,\nu)G_{23}^{0}(\theta,\nu)\bbr)\bbr|_{\nu=y}  \\
 &W_{01}(x,y,\nu)=\partial_y\bbl(
 \overline{M}_{11}(x,y,x,\nu)-\overline{M}_{21}(x,y,x,\nu)
 \bbr)\\
 &\qquad=\partial_x\partial_y^2\bbl(\hat{K}_{33}^{11}(x,y,\theta,\nu)- \hat{K}_{31}^{1}(x,y,\theta)G_{13}^{0}(\theta,\nu)\bbr)\bbr|_{\theta=x}\\
 &W_{02}(x,y,\nu)=\partial_y\bbl(
 \overline{M}_{12}(x,y,x,\nu)-\overline{M}_{22}(x,y,x,\nu)
 \bbr)\\
 &\qquad =-\partial_x\partial_y^2\bbl(\hat{K}_{31}^{1}(x,y,\theta)G_{13}^{0}(\theta,\nu)\bbr)\bbr|_{\theta=x}  \\
 &W_{11}(x,y,\theta,\nu)=\partial_x\partial_y\overline{M}_{11}(x,y,\theta,\nu)  \\
 &W_{21}(x,y,\theta,\nu)=\partial_x\partial_y\overline{M}_{21}(x,y,\theta,\nu)  \\
 &W_{12}(x,y,\theta,\nu)=\partial_x\partial_y\overline{M}_{12}(x,y,\theta,\nu)  \\
 &W_{22}(x,y,\theta,\nu)=\partial_x\partial_y\overline{M}_{22}(x,y,\theta,\nu)
 \end{align*}
 with $\hat{K}_{31}^{1}:=J_2\overline{K}_{31}^1$, $\hat{K}_{32}^{1}:=J_2\overline{K}_{32}^{1}$, and $\hat{K}_{33}^{11}:=J_2\overline{K}_{33}^{11}$.
 Once again, it is clear from the definitions of functions $K_{31}^{1}$, $K_{32}^{1}$ and $K_{33}^{11}$ that the derivatives $\partial_x\partial_y^2 K_{31}^{1}$, $\partial_x\partial_y^2 K_{33}^{11}$ and $\partial_x^2\partial_y K_{32}^{1}$, $\partial_x^2\partial_y K_{33}^{11}$, evaluated at $\nu=y$ and $\theta=x$ respectively, will all be equal to zero, suggesting
 \begin{align*}
 W_{10}(x,y,\theta)&=0,	&
 W_{20}(x,y,\theta)&=0,\\
 W_{01}(x,y,\nu)&=0,	&
 W_{02}(x,y,\nu)&=0.
 \end{align*}
 In addition, studying the definitions of the functions $K_{30}$, $K_{31}^{1}$, $K_{32}^{1}$, and $K_{33}^{11}$, we  find that their derivatives $\partial_x\partial_y K_{30}$, $\partial_x\partial_y K_{31}^{1}$, $\partial_x\partial_y K_{32}^{1}$, and $\partial_x\partial_y K_{33}^{11}$ are all constant matrices. By definition of the functions $T_{11},T_{21},T_{12}$ and $T_{22}$, it follows that $\partial_x^2\partial_y^2 T_{ij}=0$ for each $i,j\in\{1,2\}$, and thus
 \begin{align*}
 W_{ij}(x,y,\theta,\nu)&=\partial_x\partial_y\overline{M}_{ij}(x,y,\theta,\nu)   \\
 &=\partial_x^2\partial_y^2[J_2 N T_{ij}(x,y,\theta,\nu)]=0.
 \end{align*}
 Finally, from the definition of $\hat{K}_{33}^{11}$, it follows that
 \begin{align*}
 W_{00}(x,y)&\! =\! \left[\partial_x\partial_y\hat{K}_{33}^{11}(x,y,\theta,\nu)\right]\bbbr|_{\tiny\mat{\theta=x\\\nu=y}}	
 \! \! \! =\! \bmat{0&\! \! 0&\! \! I_{n_2}}
 \! =\! J_2N.
 \end{align*}
 Combining these results, we obtain Relation~\eqref{eq_uhat2=Du2_appendix}
 \begin{align*}
 J_2 N \partial_x^2\partial_y^2\bbl(\! \mcl{T}\hat{\mbf{u}}\! \bbr)\! 
 &=\! \partial_x\partial_y\bbl[J_2\ \partial_x\partial_y\bbl(\mcl{P}[NT]\hat{\mbf{u}}\bbr)\bbr]\!     \\
 &=\! \left(\! \mcl{P}\! \smallbmat{W_{00}&W_{01}&W_{02}\\W_{10}&W_{11}&W_{12}\\W_{20}&W_{21}&W_{22}}\hat{\mbf{u}}\! \right)\! 
 =\! \left(\! \mcl{P}\! \smallbmat{J_{2}N&0&0\\0&0&0\\0&0&0}\hat{\mbf{u}}\! \right)\! 
 =\! \hat{\mbf{u}}_2.
 \end{align*}

\end{proof}

\subsection{A Parameterization of Positive PI Operators}\label{sec_pos_PI_appendix}

\begin{prop}\label{prop_pos_PI_appendix}
For any $Z_1,\hdots,Z_9\in L_2^{q\times n}[x,y,\theta,\nu]$ and a scalar function $g\in L_2[x,y]$ with $g(x,y)\geq 0$ for all $x,y \in [a,b] \times [c,d]$ let $\mcl L_{\text{PI}}: \R^{9q \times 9q} \rightarrow \mcl N_{2D}^{n\times n}$ be defined as
\begin{align}\label{eq_posmat_to_posPI_appendix}
 &\mcl{L}_{\text{PI}}\left(\smallbmat{P_{11}&\hdots&P_{19}\\\vdots&\ddots&\vdots\\P_{91}&\hdots&P_{99}}\right)    
 =N:=\bmat{N_{00}&N_{01}&N_{02}\\N_{10}&N_{11}&N_{12}\\N_{20}&N_{21}&N_{22}}\in\mcl{N}_{2D}^{n\times n},
\end{align}
where the functions $N_{ij}$ are as defined in Equations~\eqref{eq_pos_Nmats_appendix} in Figure~\ref{fig_positive_parameters_appendix}. Then for any $P\ge 0$, if $N =\mcl L_{\text{PI}} (P)$, we have that $\mcl{P}^*[N]=\mcl{P}[N]$, and $\ip{\mbf{u}}{\mcl{P}[N]\mbf{u}}_{L_2}\geq 0$ for any $\mbf{u}\in L_2^n[x,y]$.\\
\end{prop}

\begin{proof}
Let $P\geq 0$ be an arbitrary matrix of appropriate size, and $N=\mcl{L}_{\text{PI}}(P)\in\mcl{N}_{2D}^{n\times n}$.
 It is easy to see that, by definition of the functions $N_{ij}\in L_2$, the PI operator $\mcl{P}[N]$ defined by $N$ is self-adjoint. Furthermore, if we define a 2D-PI operator $\mcl{Z}:L_2^n[x,y]\rightarrow L_2^{9q}[x,y]$ as
 \begin{align}\label{eq_Zop_appendix}
  (\mcl{Z}\mbf{u})(x,y)
  \!=\!
	\left[\!\!
  \begin{array}{rcccl}
   \!\!\!&\!\!\!\sqrt{g}(x,y)\!\!\!&\!\!\!Z_1(x,y)\!\!\!&\!\!\!\mbf{u}(x,y)\!\!\!&\!\!\! \\
   \int_{a}^{x}\!\!\!&\!\!\!\sqrt{g}(x,y)\!\!\!&\!\!\!Z_2(x,y,\theta)\!\!\!&\!\!\!\mbf{u}(\theta,y)\!\!\!&\!\!\!d\theta    \\
   \int_{x}^{b}\!\!\!&\!\!\!\sqrt{g}(x,y)\!\!\!&\!\!\!Z_3(x,y,\theta)\!\!\!&\!\!\!\mbf{u}(\theta,y)\!\!\!&\!\!\!d\theta    \\
   \int_{c}^{y}\!\!\!&\!\!\!\sqrt{g}(x,y)\!\!\!&\!\!\!Z_4(x,y,\nu)\!\!\!&\!\!\!\mbf{u}(x,\nu)\!\!\!&\!\!\!d\nu \\
   \int_{y}^{d}\!\!\!&\!\!\!\sqrt{g}(x,y)\!\!\!&\!\!\!Z_5(x,y,\nu)\!\!\!&\!\!\!\mbf{u}(x,\nu)\!\!\!&\!\!\!d\nu \\
   \int_{a}^{x}\int_{c}^{y}\!\!\!&\!\!\!\sqrt{g}(x,y)\!\!\!&\!\!\!Z_6(x,y,\theta,\nu)\!\!\!&\!\!\!\mbf{u}(\theta,\nu)\!\!\!&\!\!\!d\nu d\theta \\
   \int_{x}^{b}\int_{c}^{y}\!\!\!&\!\!\!\sqrt{g}(x,y)\!\!\!&\!\!\!Z_7(x,y,\theta,\nu)\!\!\!&\!\!\!\mbf{u}(\theta,\nu)\!\!\!&\!\!\!d\nu d\theta \\
   \int_{a}^{x}\int_{y}^{d}\!\!\!&\!\!\!\sqrt{g}(x,y)\!\!\!&\!\!\!Z_8(x,y,\theta,\nu)\!\!\!&\!\!\!\mbf{u}(\theta,\nu)\!\!\!&\!\!\!d\nu d\theta \\
   \int_{x}^{b}\int_{y}^{d}\!\!\!&\!\!\!\sqrt{g}(x,y)\!\!\!&\!\!\!Z_9(x,y,\theta,\nu)\!\!\!&\!\!\!\mbf{u}(\theta,\nu)\!\!\!&\!\!\!d\nu d\theta
  \end{array}\!\!\right],
 \end{align}
 by the composition rules of 2D-PI operators (as discussed in Appx.~\ref{sec_2Dalgebra_appendix}), it follows that $\mcl{P}[N]=\mcl{Z}^* P\mcl{Z}$. Since $P\geq 0$, we may split $P=\left[P^{\f{1}{2}}\right]^T P^{\f{1}{2}}$ for some $P^{\f{1}{2}}\in\R^{9q\times 9q}$, and thus
 \begin{align*}
  \ip{\mbf{u}}{\mcl{P}[N]\mbf{u}}_{L_2}
  &=\ip{\mcl{Z}\mbf{u}}{P\mcl{Z}\mbf{u}}_{L_2}   \\
  &=\ip{P^{\frac{1}{2}}\mcl{Z}\mbf{u}}{P^{\frac{1}{2}}\mcl{Z}\mbf{u}}_{L_2}\geq 0
 \end{align*}
 for any $\mbf{u}\in L_2^n[x,y]$, concluding the proof.

\end{proof}

 \begin{figure*}[!t]
 	
 	\hrulefill
 	{\footnotesize
 		\begin{align}\label{eq_pos_Nmats_appendix}
 		&N_{00}(x,y)=g(x,y)[Z_{1}(x,y)]^T P_{11}Z_{1}(x,y)
 		\nonumber\\
 		& \nonumber\\
 		&N_{10}(x,y,\theta)=g(x,y)[Z_{1}(x,y)]^T P_{12}Z_{2}(x,y,\theta) +g(\theta,y)[Z_{3}(\theta,y,x)]^T P_{31}Z_{1}(\theta,y)
 		\nonumber\\
 		&\qquad+\int_{x}^{b}g(\eta,y)[Z_{2}(\eta,y,x)]^T P_{22}Z_{2}(\eta,y,\theta)d\eta
 		+\int_{\theta}^{x}g(\eta,y)[Z_{3}(\eta,y,x)]^T P_{32}Z_{2}(\eta,y,\theta)d\eta
 		+\int_{a}^{\theta}g(\eta,y)[Z_{3}(\eta,y,x)]^T P_{33}Z_{3}(\eta,y,\theta)d\eta
 		\nonumber\\
 		&N_{20}(x,y,\theta)=[N_{10}(\theta,y,x)]^T
 		\nonumber\\
 		& \nonumber\\
 		&N_{01}(x,y,\nu)=g(x,y)[Z_{1}(x,y)]^T P_{14}Z_{4}(x,y,\nu) +g(x,\nu)[Z_{5}(x,\nu,y)]^T P_{51}Z_{1}(x,\nu)
 		\nonumber\\
 		&\qquad+\int_{y}^{d}g(x,\mu)[Z_{4}(x,\mu,y)]^T P_{44}Z_{4}(x,\mu,\nu)d\mu +\int_{\nu}^{y}g(x,\mu)[Z_{5}(x,\mu,y)]^T P_{54}Z_{4}(x,\mu,\nu)d\mu
 		+\int_{c}^{\nu}g(x,\mu)[Z_{5}(x,\mu,y)]^T P_{55}Z_{5}(x,\mu,\nu)d\mu
 		\nonumber\\
 		&N_{02}(x,y,\nu)=[N_{01}(x,\nu,y)]^T
 		\nonumber\\
 		& \nonumber\\
 		&N_{11}(x,y,\theta,\nu)
 		=g(x,y)[Z_{1}(x,y)]^T P_{16}Z_{6}(x,y,\theta,\nu) +g(\theta,\nu)[Z_{9}(\theta,x,\nu,y)]^T P_{91}Z_{1}(\theta,\nu)
 		\nonumber\\
 		&\qquad+g(x,\nu)[Z_{5}^{02}(x,\nu,y)]^T P_{52}Z_{2}(x,\theta,\nu)
 		+g(\theta,y)[Z_{3}^{20}(\theta,y,x)]^T P_{34}Z_{4}(\theta,y,\nu)
 		\nonumber\\
 		&\qquad+\int_{x}^{b}g(\eta,y)[Z_{2}(\eta,y,x)]^T  P_{26}Z_{6}(\eta,y,\theta,\nu)]d\eta
 		+\int_{\theta}^{x}g(\eta,y)[Z_{3}(\eta,y,x)]^T P_{36}Z_{6}(\eta,y,\theta,\nu)d\eta
 		+\int_{a}^{\theta}g(\eta,y)[Z_{3}(\eta,y,x)]^T P_{37}Z_{7}(\eta,y,\theta,\nu)d\eta
 		\nonumber\\
 		&\qquad+\int_{x}^{b}g(\eta,\nu)[Z_{7}(\eta,\nu,x,y)]^T  P_{72}Z_{2}(\eta,\theta,\nu)]d\eta
 		+\int_{\theta}^{x}g(\eta,\nu)[Z_{9}(\eta,\nu,x,y)]^T P_{92}Z_{2}(\eta,\theta,\nu)d\eta
 		+\int_{a}^{\theta}g(\eta,\nu)[Z_{9}(\eta,\nu,x,y)]^T P_{93}Z_{3}(\eta,\theta,\nu)d\eta
 		\nonumber\\
 		&\qquad+\int_{y}^{d}g(x,\mu)[Z_{4}(x,\mu,y)]^T P_{46}Z_{6}(x,\mu,\theta,\nu)d\mu
 		+\int_{\nu}^{y}g(x,\mu)[Z_{5}(x,\mu,y)]^T P_{56}Z_{6}(x,\mu,\theta,\nu)d\mu
 		+\int_{c}^{\nu}g(x,\mu)[Z_{5}(x,\mu,y)]^T P_{58}Z_{8}(x,\mu,\theta,\nu)d\mu
 		\nonumber\\
 		&\qquad+\int_{y}^{d}g(\theta,\mu)[Z_{7}(\theta,\mu,x,y)]^T P_{74}Z_{4}(\theta,\mu,\nu)d\mu
 		+\int_{\nu}^{y}g(\theta,\mu)[Z_{9}(\theta,\mu,x,y)]^T P_{94}Z_{4}(\theta,\mu,\nu)d\mu
 		+\int_{c}^{\nu}g(\theta,\mu)[Z_{9}(\nu,x,\mu,y)]^T P_{95}Z_{5}(\theta,\mu,\nu)d\mu
 		\nonumber\\
 		&\qquad+\int_{x}^{b}\int_{y}^{d}g(\eta,\mu)[Z_{6}(\eta,\mu,x,y)]^T  P_{66}Z_{6}(\eta,\mu,\theta,\nu)]d\mu d\eta
 		+\int_{\theta}^{x}\int_{y}^{d}g(\eta,\mu)[Z_{7}(\eta,\mu,x,y)]^T  P_{76}Z_{6}(\eta,\mu,\theta,\nu)]d\mu d\eta    \nonumber\\
 		&\qquad\qquad+\int_{a}^{\theta}\int_{y}^{d}g(\eta,\mu)[Z_{7}(\eta,\mu,x,y)]^T  P_{77}Z_{7}(\eta,\mu,\theta,\nu)]d\mu d\eta
 		+\int_{x}^{b}\int_{\nu}^{y}g(\eta,\mu)[Z_{8}(\eta,\mu,x,y)]^T  P_{86}Z_{6}(\eta,\mu,\theta,\nu)]d\mu d\eta \nonumber\\
 		&\qquad\qquad\qquad+\int_{\theta}^{x}\int_{\nu}^{y}g(\eta,\mu)[Z_{9}(\eta,\mu,x,y)]^T  P_{96}Z_{6}(\eta,\mu,\theta,\nu)]d\mu d\eta
 		+\int_{a}^{\theta}\int_{\nu}^{y}g(\eta,\mu)[Z_{9}(\eta,\mu,x,y)]^T  P_{97}Z_{7}(\eta,\mu,\theta,\nu)]d\mu d\eta    \nonumber\\
 		&\qquad\qquad\qquad\qquad+\int_{x}^{b}\int_{c}^{\nu}g(\eta,\mu)[Z_{8}(\eta,\mu,x,y)]^T  P_{88}Z_{8}(\eta,\mu,\theta,\nu)]d\mu d\eta
 		+\int_{\theta}^{x}\int_{c}^{\nu}g(\eta,\mu)[Z_{9}(\eta,\mu,x,y)]^T  P_{98}Z_{8}(\eta,\mu,\theta,\nu)]d\mu d\eta    \nonumber\\
 		&\qquad\qquad\qquad\qquad\qquad+\int_{a}^{\theta}\int_{c}^{\nu}g(\eta,\mu)[Z_{9}(\eta,\mu,x,y)]^T  P_{99}Z_{9}(\eta,\mu,\theta,\nu)]d\mu d\eta
 		\nonumber\\
 		&N_{22}(x,y,\theta,\nu)=[N_{11}(\theta,x,\nu,y)]^T
 		\nonumber\\
 		& \nonumber\\
 		&N_{21}(x,y,\theta,\nu)
 		=g(x,y)[Z_{1}(x,y)]^T P_{17}Z_{7}(x,y,\theta,\nu) +g(\theta,\nu)[Z_{8}(\theta,x,\nu,y)]^T P_{81}Z_{1}(\theta,\nu)
 		\nonumber\\
 		&\qquad+g(x,\nu)[Z_{5}(x,\nu,y)]^T P_{53}Z_{3}(x,\theta,\nu)
 		+g(\theta,y)[Z_{2}(\theta,y,x)]^T P_{24}Z_{4}(\theta,y,\nu)
 		\nonumber\\
 		&\qquad+\int_{\theta}^{b}g(\eta,y)[Z_{2}(\eta,y,x)]^T  P_{26}Z_{6}(\eta,y,\theta,\nu)]d\eta
 		+\int_{x}^{\theta}g(\eta,y)[Z_{2}(\eta,y,x)]^T P_{28}Z_{8}(\eta,y,\theta,\nu)d\eta
 		+\int_{a}^{x}g(\eta,y)[Z_{3}(\eta,y,x)]^T P_{37}Z_{7}(\eta,y,\theta,\nu)d\eta
 		\nonumber\\
 		&\qquad+\int_{\theta}^{b}g(\eta,\nu)[Z_{8}(\eta,\nu,x,y)]^T  P_{82}Z_{2}(\eta,\theta,\nu)]d\eta
 		+\int_{x}^{\theta}g(\eta,\nu)[Z_{8}(\eta,\nu,x,y)]^T P_{83}Z_{3}(\eta,\theta,\nu)d\eta
 		+\int_{a}^{x}g(\eta,\nu)[Z_{9}(\eta,\nu,x,y)]^T P_{93}Z_{3}(\eta,\theta,\nu)d\eta
 		\nonumber\\
 		&\qquad+\int_{y}^{d}g(\theta,\mu)[Z_{6}(\theta,\mu,x,y)]^T P_{64}Z_{4}(\theta,\mu,\nu)d\mu
 		+\int_{\nu}^{y}g(\theta,\mu)[Z_{8}(\theta,\mu,x,y)]^T P_{84}Z_{4}(\theta,\mu,\nu)d\mu
 		+\int_{c}^{\nu}g(\theta,\mu)[Z_{8}(\theta,\mu,x,y)]^T P_{85}Z_{5}(\theta,\mu,\nu)d\mu
 		\nonumber\\
 		&\qquad+\int_{y}^{d}g(x,\mu)[Z_{4}(x,\mu,y)]^T P_{47}Z_{7}(x,\mu,\theta,\nu)d\mu
 		+\int_{\nu}^{y}g(x,\mu)[Z_{5}(x,\mu,y)]^T P_{57}Z_{7}(x,\mu,\theta,\nu)d\mu
 		+\int_{c}^{\nu}g(x,\mu)[Z_{5}(x,\mu,y)]^T P_{59}Z_{9}(x,\mu,\theta,\nu)d\mu
 		\nonumber\\
 		&\qquad+\int_{\theta}^{b}\int_{y}^{d}g(\eta,\mu)[Z_{6}(\eta,\mu,x,y)]^T  P_{66}Z_{6}(\eta,\mu,\theta,\nu)]d\mu d\eta
 		+\int_{x}^{\theta}\int_{y}^{d}g(\eta,\mu)[Z_{6}(\eta,\mu,x,y)]^T  P_{67}Z_{7}(\eta,\mu,\theta,\nu)]d\mu d\eta    \nonumber\\
 		&\qquad\qquad+\int_{a}^{x}\int_{y}^{d}g(\eta,\mu)[Z_{7}(\eta,\mu,x,y)]^T  P_{77}Z_{7}(\eta,\mu,\theta,\nu)]d\mu d\eta
 		+\int_{\theta}^{b}\int_{\nu}^{y}g(\eta,\mu)[Z_{8}(\eta,\mu,x,y)]^T  P_{86}Z_{6}(\eta,\mu,\theta,\nu)]d\mu d\eta    \nonumber\\
 		&\qquad\qquad\qquad+\int_{x}^{\theta}\int_{\nu}^{y}g(\eta,\mu)[Z_{8}(\eta,\mu,x,y)]^T  P_{87}Z_{7}(\eta,\mu,\theta,\nu)]d\mu d\eta
 		+\int_{a}^{x}\int_{\nu}^{y}g(\eta,\mu)[Z_{9}(\eta,\mu,x,y)]^T  P_{97}Z_{7}(\eta,\mu,\theta,\nu)]d\mu d\eta    \nonumber\\
 		&\qquad\qquad\qquad\qquad+\int_{\theta}^{b}\int_{c}^{\nu}g(\eta,\mu)[Z_{8}(\eta,\mu,x,y)]^T  P_{88}Z_{8}(\eta,\mu,\theta,\nu)]d\mu d\eta
 		+\int_{x}^{\theta}\int_{c}^{\nu}g(\eta,\mu)[Z_{9}(\eta,\mu,x,y)]^T  P_{98}Z_{8}(\eta,\mu,\theta,\nu)]d\mu d\eta    \nonumber\\
 		&\qquad\qquad\qquad\qquad\qquad+\int_{a}^{x}\int_{c}^{\nu}g(\eta,\mu)[Z_{9}(\eta,\mu,x,y)]^T  P_{99}Z_{9}(\eta,\mu,\theta,\nu)]d\mu d\eta    \nonumber\\
 		&N_{12}(x,y,\theta,\nu)=[N_{21}(\theta,x,\nu,y)]^T
 		\end{align}
 	}
 	\hrulefill
 	\caption{Parameters $N$ describing the positive PI operator $\mcl{P}[N]=\mcl{Z}^* P\mcl{Z}$ in Proposition~\ref{prop_pos_PI_appendix}}
 	\label{fig_positive_parameters_appendix}
 \end{figure*}

\end{appendices}

\end{document}